\documentclass[11pt]{article}
\usepackage[T1]{fontenc}
\usepackage{lmodern, microtype,soul}

\usepackage[left=1.3in, right=1.3in, top=1.3in, bottom=1.3in]{geometry}
\usepackage[small]{titlesec}

\usepackage{xfrac,enumerate,comment,amssymb,amsmath,amsthm}
\usepackage{fancyhdr,bbm,soul,xcolor,paralist}

\usepackage{hyperref}
\hypersetup{
  colorlinks=true,
  linkcolor=blue,
  citecolor=blue
}

\usepackage{natbib}
\setcitestyle{authoryear,open={(},close={)}}

\newtheorem{theorem}{Theorem}
\newtheorem{lemma}{Lemma}
\newtheorem{proposition}{Proposition}

\theoremstyle{definition}
\newtheorem{definition}{Definition}

\usepackage{graphicx}

\usepackage{pstricks, enumerate, pst-node, pst-text, pst-plot}

\newcommand{\R}{\mathbb{R}}

\newcommand{\eps}{\varepsilon}

\usepackage{xparse}

\DeclareDocumentCommand\Pr{ m g }{\ensuremath{
    {   \IfNoValueTF {#2}
      {\mathbb{P}\left[{#1}\right]}
      {\mathbb{P}\left[{#1}\middle\vert{#2}\right]}
    }
}}
\DeclareDocumentCommand\E{ m g }{\ensuremath{
    {   \IfNoValueTF {#2}
      {\mathbb{E}\left[{#1}\right]}
      {\mathbb{E}\left[{#1}\middle\vert{#2}\right]}%
    }
}}

\DeclareMathOperator*{\Var}{Var}

\newcommand{\ind}[1]{{\mathbbm{1}_{\left\{{#1}\right\}}}}

\def\dd{\mathrm{d}}
\def\ee{\mathrm{e}}

\def\B{\mathcal{B}}
\def\B{\mathcal{B}}

\title{{\bf \Large From Blackwell Dominance in Large Samples to R\'enyi Divergences and Back Again}\thanks{We are grateful to the co-editor and three referees for their comments and suggestions. In addition we would like to thank Kim Border, Laura Doval, Federico Echenique, Tobias Fritz, Drew Fudenberg, George Mailath, Massimo Marinacci, Margaret Meyer, Marco Ottaviani and Peter Norman S\o rensen for helpful discussions.}}
\author{ \large Xiaosheng Mu\thanks{Princeton University. Email: xmu@princeton.edu. Xiaosheng Mu acknowledges the hospitality of Columbia University and the Cowles Foundation at Yale University, which hosted him during parts of this research.} \ \ \ Luciano Pomatto\thanks{Caltech. Email: 	luciano@caltech.edu.} \ \ \ %
Philipp Strack\thanks{Yale University. Email:  philipp.strack@yale.edu.} \ \ \ %
Omer Tamuz\thanks{Caltech. Email: tamuz@caltech.edu. Omer Tamuz was supported by a grant from the Simons Foundation (\#419427), a Sloan research fellowship, and a BSF award (\#2018397).}}

\date{\today}

\begin{document}

\maketitle

\begin{abstract}
 We study repeated independent Blackwell experiments; standard examples include drawing multiple samples from a population, or performing a measurement in different locations. In the baseline setting of a binary state of nature, we compare experiments in terms of their informativeness in large samples. Addressing a question due to \cite{blackwell1951comparison}, we show that generically an experiment is more informative than another in large samples if and only if it has higher R\'enyi divergences.
 
 We apply our analysis to the problem of measuring the degree of dissimilarity between distributions  by means of divergences. A useful property of R\'enyi divergences is their additivity with respect to product distributions.
 Our characterization of Blackwell dominance in large samples implies that every additive divergence that satisfies the data processing inequality is an integral of R\'enyi divergences.

\end{abstract}

%\epigraph{{\em I love show ``Law and Order'' but the @MRbelzer casting is the worst ever. No talent---unwatchable!}}
%{---D.\ J.\ Trump, 6 Nov 2012}
%L: xhahahha 

%\tableofcontents

\section{Introduction}

% replace the discussion of the LLN with a discussion of Blackwell Theorem
%The law of large numbers connects the long-run frequency of an event to its likelihood, and has been one of the first principles giving practical meaning to the notion of probability. It also guides our intuition in decision problems that involve multiple independent risks, as in the case of a physician treating multiple patients, or of an investor managing a large portfolio of assets.

%While the law of large numbers provides exact predictions about limiting frequencies, it does not provide precise indications to decision makers who are concerned with the problem of choosing between risky prospects, even when these prospects are obtained by aggregating a large number of i.i.d.\ risks.

Statistical experiments form a general framework for modeling information: Given a set $\Theta$ of parameters, an \textit{experiment} $P$ produces an observation distributed according to $P_\theta$, given the true parameter value $\theta \in \Theta$. Blackwell's\ celebrated theorem \citep{blackwell1951comparison} provides a partial order for comparing experiments in terms of their informativeness.

As is well known, requiring two experiments to be ranked in the  Blackwell order is a demanding condition. Consider the problem of testing a binary hypothesis $\theta \in \{0,1\}$, based on random samples drawn from one of two experiments $P$ or $Q$. According to Blackwell's ordering, $P$ is more informative than $Q$ if, for every test performed based on observations produced by $Q$, there exists another test based on $P$ that has lower probabilities of both Type-I and Type-II errors \citep{blackwell1979theory}. This is a difficult condition to satisfy, especially in the case where only one sample is produced by each experiment.

In many applications, an experiment does not consist of a single observation but of multiple i.i.d.\ samples. For example, a new vaccine is typically tested on multiple patients, and a randomized control trial assessing the effect of an intervention usually involves many subjects. We study a weakening of the Blackwell order that is appropriate for comparing experiments in terms of their large sample properties. Our starting point is the question, first posed by \cite{blackwell1951comparison}, of whether it is possible for $n$ independent observations from an experiment $P$ to be more informative than $n$ observations from another experiment $Q$, even though $P$ and $Q$ are not comparable in the Blackwell order. The question was answered in the affirmative by \cite{stein1951notes}, \cite{torgersen1970comparison} and \cite{azrieli2014comment}.\footnote{Even though \cite{stein1951notes} is frequently cited in the literature for a first example of this type, we could not gain access to that paper.} However, identifying the precise conditions under which this phenomenon occurs has remained an open problem.

We say that $P$ dominates $Q$ {\em in large samples} if for every $n$ large enough, $n$ independent observations from $P$ are more informative, in the Blackwell order, than $n$ independent observations from $Q$. We focus on a binary set of parameters $\Theta$, and show that generically $P$ dominates $Q$ in large samples if and only if the experiment $P$ has higher R\'enyi divergences than $Q$ (Theorem~\ref{thm:eventual}). R\'enyi divergences are a one-parameter family of measures of informativeness for experiments; introduced and characterized axiomatically in \cite{renyi1961measures}, we show that they capture the informativeness of an experiment in large samples. For any two experiments comparable in terms of R\'enyi divergences, we also provide a simple bound on the sample size that ensures that larger samples of independent experiments are comparable in the Blackwell order (Theorem~\ref{thm:quant}). 

The proof of this result crucially relies on two ingredients. First, we use techniques from large deviations theory to compare sums of i.i.d.\ random variables in terms of stochastic dominance. In addition, we provide and apply a new characterization of the Blackwell order: We associate to each experiment a new statistic, \textit{the perfected log-likelihood ratio}, and show that the comparison of these statistics in terms of first-order stochastic dominance is in fact equivalent to the Blackwell order.

\medskip

 We apply our characterization of Blackwell dominance in large samples to the problem of quantifying the extent to which two probability distributions are dissimilar. This is a common problem in econometrics and statistics, where formal measures quantifying the difference between distributions are referred to as \textit{divergences}.\footnote{See, e.g., \cite{sawa1978information,white1982maximum,critchley1996differential,kitamura1997information,hong2005asymptotic, ullah2002uses}. See \cite{kitamura2013robustness} for a recent application of $\alpha$-divergences, which are a reformulation of R\'enyi divergences.} Well known examples include total variation distance, the Hellinger distance, the Kullback-Leibler divergence, R\'enyi divergences, and more general $f$-divergences.
 
 R\'enyi divergences satisfy two key properties. The first is \textit{additivity}: R\'enyi divergences decompose into a sum when applied to pairs of product distributions. Additivity captures a principle of non-interaction across independent domains, as the total divergence of two unrelated pairs does not change when they are considered together as a bundle. Additivity is a natural property, and in applications it is a crucial simplification for studying i.i.d.\ processes. A second desirable property is described by the \textit{data-processing inequality}, which stipulates that the distributions of two random variables $X$ and $Y$ are at least as dissimilar as those of $f(X)$ and $f(Y)$, for any transformation $f$. As we show, this property is closely related to monotonicity with respect to the Blackwell order.
 
 Using our main result, we show that every additive divergence that satisfies the data-processing inequality and a mild finiteness condition is an integral (i.e., the limit of positive linear combinations) of R\'enyi divergences (Theorem~\ref{thm:divergences}). This result is an improvement over the original characterization of \cite{renyi1961measures}, as well as more modern ones \citep{csiszar2008axiomatic}, because it shows that additivity alone pins down a single class of divergences without making any further assumptions on the functional form.

\medskip

 The study most closely related to ours is \cite{moscarini2002eventual}. In their order, an experiment $P$ dominates another experiment $Q$ if for for every finite decision problem, a large enough sample of observations from an experiment $P$ will achieve higher expected payoff than a sample of the same size of observations from $Q$. In contrast to the order proposed by Blackwell and analyzed in this paper, their definition allows for the critical sample size to depend on the decision problem, and considers a restricted class of decision problems. We provide a detailed discussion of this and other related work in \S\ref{sec:blackwell_lit}.
 
 The paper is organized as follows. In \S\ref{sec:experiments} we provide our main definitions. \S\ref{sec:characterization} contains the characterization of Blackwell dominance in large samples, with proof deferred to \S\ref{sec:proof}. In \S\ref{sec:divergences} we characterize additive divergences. Finally, we further discuss our results and their relation to the literature in \S\ref{sec:blackwell_lit}. 

\section{Model} 
\label{sec:experiments}

\subsection{Statistical Experiments}

 A state of the world $\theta$ can take two possible values, $0$ or $1$. A \emph{Blackwell-Le Cam experiment} $P=(\Omega,P_0,P_1)$ consists of a sample space $\Omega$, which we assume to be a Polish space, and a pair of Borel probability measures $(P_0,P_1)$ defined over $\Omega$, with the interpretation that $P_\theta(A)$ is the probability of observing $A \subseteq \Omega$ in state $\theta \in \{0,1\}$. This framework is commonly encountered in simple hypothesis tests as well as in information economics. In \S\ref{sec:blackwell_lit} we discuss the case of experiments for more than two states: we obtain necessary conditions for dominance in large samples and explain the obstacles to a full characterization.

 Given two experiments $P=(\Omega,P_0,P_1)$ and $Q=(\Xi,Q_0,Q_1)$, we can form the {\em product experiment} $P \otimes Q$ given by 
 $$
    P \otimes Q = (\Omega \times \Xi,P_0 \times Q_0, P_1 \times Q_1).
 $$
 where $P_\theta \times Q_\theta$, given $\theta \in \{0,1\}$, denotes the product of the two measures. Under the experiment $P \otimes Q$ the realizations produced by both $P$ and $Q$ are observed, and the two observations are independent (conditional on the true state). For instance, if $P$ and $Q$ consist of drawing samples from two different populations, then $P \otimes Q$ consists of the joint experiment where a sample from each population is drawn. 
 We denote by
 $$
    P^{\otimes n} = P \otimes \cdots \otimes P
 $$
 the $n$-fold product experiment where $n$ independent observations are generated according to the experiment $P$.

 Consider now a Bayesian decision maker whose prior belief assigns probability $1/2$ to the state being $1$. To each experiment $P = (\Omega,P_0,P_1)$ we associate a Borel probability measure $\pi$ over $[0,1]$ that represents the distribution over posterior beliefs induced by the experiment. Formally, let $p(\omega)$ be the posterior belief that the state is $1$ given the realization $\omega \in \Omega$:
$$
  p(\omega) = \frac{\dd P_1(\omega)}{\dd P_1(\omega) + \dd P_0(\omega)}.
$$
Furthermore, define for every Borel set $B \subseteq [0,1]$
\[
    \pi_\theta(B) = P_\theta\left(\left\{\omega\,:\,p(\omega) \in B \right\} \right) 
\]
as the probability that the posterior belief will belong to $B$, given state $\theta$. We then define $\pi = (\pi_0+\pi_1)/2$ as the unconditional measure over posterior beliefs.
%%
%%  Since later we use cdfs it would be easier to also
%%  use cdfs here instead of measures.
%%  The measures \pi and \tau appear only in a couple of places,
%%  what about \Pi for the cdf? 
%%
 %Equivalently, in a bounded experiment the posterior belief $p$ is bounded away from $0$ and $1$.%; in particular, no observation completely reveals the state. 
%old is strictly included in $(0,1)$, I think the strictly is not necessary here as the support is closed.
%We denote by $\mathcal{B}$ the class of all bounded experiments.

Throughout the paper we restrict our attention to experiments where the measures $P_0$ and $P_1$ are mutually absolutely continuous, so that no signal realization $\omega \in \Omega$ perfectly reveals either state. We say that $P$ is \emph{trivial} if $P_0=P_1$, and \textit{bounded} if the derivative $\dd P_1/\dd P_0$ is bounded above and bounded away from $0$.

\subsection{The Blackwell Order}

We first review the main concepts behind Blackwell's order over experiments \citep*{bohnenblust1949reconnaissance, blackwell1953equivalent}. Consider two experiments $P$ and $Q$ and their induced distribution over posterior beliefs denoted by $\pi$ and $\tau$, respectively. The experiment $P$ \emph{Blackwell dominates} $Q$, denoted $P \succeq Q$, if
\begin{equation}\label{eq:def-blackwell}
 \int_0^1 v(p)\,\dd \pi(p) \geq \int_0^1 v(p)\,\dd \tau(p)
\end{equation}
for every convex function $v \colon (0,1) \to \R$. Equivalently, $P \succeq Q$ if $\pi$ is a mean-preserving spread of $\tau$. We write $P \succ Q$ if $P \succeq Q$ and $Q \not\succeq P$. So, $P \succ Q$ if and only if \eqref{eq:def-blackwell} holds with a strict inequality whenever $v$ is strictly convex, i.e.\ $\pi$ is a mean-preserving spread of $\tau$ and $\pi \neq \tau$.

As is well known, each convex function $v$ can be seen as the indirect utility induced by some decision problem. That is, for each convex $v$ there exists a set of actions $A$ and a utility function $u$ defined on $A \times \{0,1\}$ such that $v(p)$ is the maximal expected payoff that a decision maker can obtain in such a decision problem given a belief $p$. Hence, $P \succeq Q$ if and only if in every decision problem, an agent can obtain a higher payoff by basing her action on the experiment $P$ rather than on $Q$.

Blackwell's theorem shows that the order $\succeq$ can be equivalently defined by ``garbling'' operations: Intuitively, $P \succeq Q$ if and only if the outcome of the experiment $Q$ can be generated from the experiment $P$ by compounding the latter with additional noise, without adding further information about the state.\footnote{Formally, given two experiments $P=(\Omega,P_0,P_1)$ and $Q=(\Xi,Q_0,Q_1)$, $P \succeq Q$ if and only if there is a measurable kernel (also known as ``garbling'') $\sigma: \Omega \to \Delta(\Xi)$, where $\Delta(\Xi)$ is the set of probability measures over $\Xi$, such that for every $\theta$ and every measurable $A \subseteq \Xi$, $
    Q_\theta(A) = \int \sigma(\omega)(A) \,\dd P_\theta(\omega).$
In other terms, there is a (perhaps randomly chosen) measurable map $f$ with the property that for both $\theta=0$ and $\theta=1$, if $X$ is a random quantity distributed according to $P_\theta$ then $Y=f(X)$ is distributed according to $Q_\theta$.} 

As discussed in the introduction, we are interested in understanding the large sample properties of the Blackwell order. This motivates the next definition. 

\begin{definition}[Large Sample Order]\label{def:blackwell-largesamples}
An experiment $P$ dominates an experiment $Q$ {\em in large samples} if there exists an $n_0 \in \mathbb{N}$ such that
\begin{equation}\label{eq:def-eventual-order}
P^{\otimes n} \succeq Q^{\otimes n} \text{~~for every~~} n \geq n_0.
\end{equation}
\end{definition}

 This order was first defined by \cite{azrieli2014comment} under the terminology of \textit{eventual sufficiency}. The definition captures the informal notion that a large sample drawn from $P$ is more informative than an equally large sample drawn from $Q$. Consider, for instance, the case of hypothesis testing. The experiment $P$ dominates $Q$ in the Blackwell order if and only if for every test based on $Q$ there exists a test based on $P$ that has weakly lower probabilities of both Type-I and Type-II errors. Definition \ref{def:blackwell-largesamples} extends this notion to large samples, in line with the standard paradigm of asymptotic statistics: $P$ dominates $Q$ if every test based on $n$ i.i.d.\ realizations of $Q$ is dominated by another test based on $n$ i.i.d.\ realizations of $P$, for sufficiently large $n$. When the two experiments are statistics of a common experiment, dominance in the large sample order implies that one statistic will eventually contain all the information captured by the other.

 As shown by \citet[Theorem 12]{blackwell1951comparison}, dominance of $P$ over $Q$ implies dominance of $P^{\otimes n}$ over $Q^{\otimes n}$, for every $n$. So dominance in large samples is an extension of the Blackwell order. This extension is strict, as shown by examples in \cite{torgersen1970comparison} and \cite{azrieli2014comment}.

 %In general, the relative informativeness of two experiments $P^{\otimes n}$ and $Q^{\otimes n}$ can depend non-trivially on the number $n$ of repetitions. In section \ref{sec:examples} we construct simple examples of experiments where for some $n$, $P^{\otimes n}$ and $Q^{\otimes n}$ are ranked in the Blackwell order, but $P^{\otimes (n+1)}$ and $Q^{\otimes (n+1)}$ are not. This implies, for example, that increasing the sample size might render more difficult for a group to unanimously agree on which information source, between $P$ and $Q$, to adopt. Nevertheless we will show in section x that whenever two repeated experiments are comparable in the Blackwell order, in the sense that $P^{\otimes n} \succeq Q^{\otimes n}$ for some $n$, then there must exist some threshold $N$ such that the uniform ranking \eqref{eq:def-eventual-order} holds.

\subsection{R\'enyi Divergence and the R\'enyi Order}

Our main result relates Blackwell dominance in large samples to a well-established notion of informativeness due to \cite{renyi1961measures}. Given two probability measures $\mu,\nu$ on a measurable space $\Omega$ and a parameter $t>0$, the R\'enyi $t$-divergence is given by
\begin{equation}\label{eq:R_t}
    R_t(\mu \Vert \nu) = \frac{1}{t-1}\log\int_\Omega \left(\frac{\dd \mu}{\dd \nu}(\omega)\right)^{t-1}\,\dd \mu(\omega)
\end{equation}
when $t \neq 1$, and, ensuring continuity,
\begin{equation}\label{eq:R_1}
    R_1(\mu \Vert \nu) = \int_\Omega \log\left(\frac{\dd \mu \hfill}{\dd \nu}(\omega)\right) \,\dd \mu (\omega).
\end{equation}
Equivalently, $R_1(\mu\Vert\nu)$ is the Kullback-Leibler divergence between the measures $\mu$ and $\nu$. As $t$ increases, the value of $R_t$ increases and is continuous whenever it is finite. The limit value as $t \to \infty$, which we denote by $R_\infty(\mu\Vert\nu)$, is the essential maximum of $\log\left(\frac{\dd \mu}{\dd \nu}\right)$, the logarithm of the ratio between the two densities.

%Again extending by continuity, we define $R_\infty(\mu\Vert\nu)$ to be the maximum of the odds ratio $\frac{\dd \mu}{\dd \nu}$.

As a binary experiment precisely consists of a pair of probability measures, we can apply this definition straightforwardly to experiments.
Given an experiment $P = (\Omega, P_0, P_1)$, a state $\theta$, and parameter $t>0$, the R\'enyi $t$-divergence of $P$ under $\theta$ is 
\begin{equation}\label{eq:Renyi divergence}
    R_P^\theta(t) = R_t(P_\theta\Vert P_{1-\theta}).
\end{equation}

Intuitively, observing a sample realization for which the likelihood ratio $\dd P_\theta / \dd P_{1-\theta}$ is high constitutes evidence that favors state $\theta$ over $1-\theta$. For instance, in the case of $t=2$, a higher value of $R_P^\theta(2)$ describes an experiment that, in expectation, more strongly produces evidence in favor of the state $\theta$ when this is the correct state. Varying the parameter $t$ allows to consider different moments for the distribution of likelihood ratios.  R\'enyi divergences have found applications to statistics and information theory \citep{liese2006divergences,csiszar2008axiomatic}, machine learning \citep{poczos2012nonparametric, krishnamurthy2014nonparametric}, computer science \citep{fritz2017resource}, and quantum information \citep{horodecki2009quantum, jensen2019asymptotic}. The Hellinger transform \cite[][p.\ 39]{torgersen1991comparison}, another well known measure of informativeness, is a monotone transformation of the R\'enyi divergences of an experiment.

The two R\'enyi divergences $R_P^1$ and $R_P^0$ of an experiment are related by the identity
\begin{align}
    \label{eq:t-1-t}
  R^1_P(t) = \frac{t}{1-t}R^0_P(1-t).
\end{align}
Hence the values of $R^\theta_P(t)$ for $t \in [0,1/2]$ are determined by the values of $R^{1-\theta}_P(t)$ on the interval $[1/2,1]$. Thus, it suffices to consider values of $t$ in $[1/2,\infty]$.

\begin{definition}[R\'enyi Order] 
An experiment $P$ dominates an experiment $Q$ in the \textit{R\'enyi order} if it holds that for all $\theta \in \{0,1\}$ and all $t > 0$
\[
    R_P^\theta(t) > R_Q^\theta(t) \,.
\]
\end{definition}

The R\'enyi order is a extension of the (strict) Blackwell order. In the proof of Theorem~\ref{thm:eventual} below, we explicitly construct a one-parameter family of decision problems with the property that dominance in the R\'enyi order is equivalent to higher expected payoff with respect to each decision problem in this family. See \S\ref{sec:Renyi decision problems} for details.

A simple calculation shows that if $P = S \otimes T$ is the product of two experiments, then for every state $\theta$,
$$
R_P^\theta  = R_S^\theta + R_T^\theta \,.
$$
A key implication is that $P$ dominates $Q$ in the R\'enyi order if and only if the same relation holds for their $n$-th fold repetitions $P^{\otimes n}$ and $Q^{\otimes n}$, for any $n$. Hence, the R\'enyi order compares experiments in terms of properties that are unaffected by the number of samples. Because, in turn, the R\'enyi order extends the Blackwell order, it follows that dominance in the R\'enyi order is a necessary condition for dominance in large samples.

As a final remark on the definition of the R\'enyi order, it is important to require the comparison for both states $\theta = 0$ and $\theta = 1$, as there exist pairs of experiments $P$ and $Q$ such that $R_P^1(t) > R_Q^1(t)$ for every $t$, but $R_P^0(t) < R_Q^0(t)$ for some $t$.\footnote{A simple example involves the following pair of binary experiments: 
    \begin{equation*}
\begin{array}{c|c|c|}
 & \omega & \omega' \\
\hline
P_0 & 1/3 & 2/3  \\
\hline
P_1 & 2/3 & 1/3  \\
\hline
\end{array} 
\qquad \qquad \qquad
\begin{array}{c|c|c|}
 & \omega & \omega'\\
\hline
Q_0 & 6/9 & 3/9 \\
\hline
Q_1 & 8/9 & 1/9 \\
\hline
\end{array} 
\end{equation*} 
where the entries represent conditional probabilities. Direct computation shows that $R_P^1(t) > R_Q^1(t)$ for every $t > 0$, while $R_P^0(t) < R_Q^0(t)$ for $t > 2$. 
}

\section{Characterization of the Large Sample Order}\label{sec:characterization}
%Recall that we call two bounded experiments $P$ and $Q$ inducing the distributions over posteriors $\pi$ and $\tau$ a \textit{generic pair} if
%\[
%    \max\text{supp}(\pi) \neq \max\text{supp}(\tau) \text{~~~and~~~} \min\text{supp}(\pi) \neq \min\text{supp}(\tau),
%\]
%In analogy with our definition of a generic pair of random variables,
We say two bounded experiments $P$ and $Q$ form a {\em generic} pair if the essential maxima of the log-likelihood ratios $\log\frac{\dd P_1 \hfill}{\dd P_{0}}$ and $\log\frac{\dd Q_1 \hfill}{\dd Q_{0}}$ are different, and if their essential minima are also different. This holds, for example, if for each of the two experiments the set of signal realizations is finite, and there is no posterior beliefs that can be induced by both experiments.
\begin{theorem} \label{thm:eventual}
For a generic pair of bounded experiments $P$ and $Q$, the following are equivalent:
  \begin{enumerate}
      \item $P$ dominates $Q$ in large samples.
      \item $P$ dominates $Q$ in the R\'enyi order. 
  \end{enumerate}
\end{theorem}

That (ii) implies (i) means that for every two experiments $P$ and $Q$ that are ranked in the R\'enyi order, there exists a sample size $n$ such that $n$ or more independent samples of $P$ and $Q$ are ranked in the Blackwell order. The proof of the theorem also establishes an upper bound on $n$; however, as stating this bound requires several additional concepts we defer this result to Theorem \ref{thm:quant} in \S\ref{sec:upper bound}. The complete proof of Theorem~\ref{thm:eventual} appears in \S\ref{sec:proof} below.

We mention that Theorem~\ref{thm:eventual} remains true so long as the dominated experiment $Q$ is bounded (whereas $P$ need not be bounded); see \S\ref{sec:unbounded} in the appendix for discussion of this and another generalization. On the other hand, the theorem does not remain true if we remove the genericity assumption. In \S\ref{sec:eventualFail} in the appendix we discuss the knife-edge case where the maxima or the minima of the log-likelihood ratios are equal. We demonstrate a non-generic pair of experiments $P$ and $Q$ such that $P$ dominates $Q$ in the R\'enyi order, but $P$ does not dominate $Q$ in large samples. Given this example, it seems difficult to obtain an applicable characterization of large sample dominance without imposing some genericity condition. 

% Indeed, in this knife-edge case verifying large sample dominance involves checking for combinatorial conditions that, compared to (i) in Theorem~\ref{thm:eventual}, are less immediate to verify.
%Our genericity assumption plays an additional role. A key tool in the application of a stochastic order is a characterization in terms of a {\em generator}: a class $V$ of functions such that $X$ dominates $Y$ if and only if $\E{\phi(X)} \geq \E{\phi(Y)}$ for every $\phi \in V$. The result in \S\ref{sec:eventualFail} shows that without our genericity assumption, aggregate stochastic dominance does not admit a generator. %There is no class $V$ of functions such that $X$ dominates $Y$ if and only if $\E{\phi(X)} \geq \E{\phi(Y)}$ for every $\phi \in V$.
%We next illustrate Theorem~\ref{thm:eventual} with two examples:

\medskip
 A natural alternative definition of ``Blackwell dominance in large samples'' would require $P^{\otimes n} \succeq Q^{\otimes n}$ to hold for {\em some} $n$, but the resulting order is in fact equivalent under our genericity assumption. This is a consequence of Theorem~\ref{thm:eventual}, because $P^{\otimes n_0} \succeq Q^{\otimes n_0}$ for any $n_0$ implies $P$ dominates $Q$ in the R\'enyi order, which in turn implies $P^{\otimes n} \succeq Q^{\otimes n}$ for all large $n$.\footnote{However, it is not true that $P^{\otimes n_0} \succeq Q^{\otimes n_0}$ for some $n_0$ implies $P^{\otimes n} \succeq Q^{\otimes n}$ for all $n \geq n_0$. The case of $\alpha = 0.305$, $\beta = 0.1$ in Example 2 below provides an example where $P^{\otimes 2}$ Blackwell dominates $Q^{\otimes 2}$, but $P^{\otimes 3}$ does not dominate $Q^{\otimes 3}$.}

\subsection{Examples}

In this section we illustrate Theorem~\ref{thm:eventual} by means of two examples of pairs of experiments that are not Blackwell ranked, but are ranked in large samples.

\paragraph{Example 1.}
% L: Since we already talk about genericity when talking about stochastic dominance this might seem a little repetitive. 

%In \S\ref{sec:eventualFail} we provide a complete characterization \omer{em... but really we don't} of the Blackwell-order that extends Theorem~\ref{thm:eventual} to experiments where the distributions over posteriors admit the same maxima or the same minima.
We first introduce a new example of two such experiments $P$ and $Q$. The first experiment $P$ appears in \cite{smith2000pathological}. The signal space is the interval $[0,1]$, and the measures $P_0$ and $P_1$ are absolutely continuous with densities $f_0(s) = 1$ and $f_1(s) = 1/2 + s$. 
Our second experiment $Q$ is binary, with signal space $\{0,1\}$. The measure $Q_0$ assigns probability $1/2$ to both signals, while the other measure is $Q_1(1)=p$ and $Q_1(0)=1-p$. 

For $p=0.625$, $P$ Blackwell dominates $Q$, as witnessed by the garbling from $[0,1]$ to $\{0,1\}$ that maps all signal realizations above $1/2$ to $1$ and all realizations below $1/2$ to $0$. For larger $p$, $P$ is no longer Blackwell dominant. To see this, consider the decision problem in which the prior belief is uniform, the set of actions is the set of states, and the payoff is one if the action matches the state and zero otherwise. It is easy to check that for $p > 0.625$, the experiment $Q$ yields a larger expected payoff.

Nevertheless, if we choose $p = 0.63$, then as Figure~\ref{fig:renyi} below suggests,  $P$ dominates $Q$ in the R\'enyi order even though the two experiments are not Blackwell ranked.\footnote{The R\'enyi divergences as defined in \eqref{eq:Renyi divergence} are computed to be
\[
  R_P^0(t) = \frac{1}{t-1}\log \left(\frac{(3/2)^{2-t} - (1/2)^{2-t}}{2-t}\right); \quad R_P^1(t) =   \frac{1}{t-1}\log \left(\frac{(3/2)^{t+1} - (1/2)^{t+1}}{t+1}\right)
\]
and
\[
    R_Q^0(t) = \frac{1}{t-1}\log\left(2^{-t}\cdot(p^{1-t}+(1-p)^{1-t})\right); \quad R_Q^1(t) = \frac{1}{t-1}\log \left(2^{t-1}\cdot(p^t+(1-p)^{t})\right).
\]}
\begin{figure}[h]
    \centering
    \includegraphics[scale=0.25]{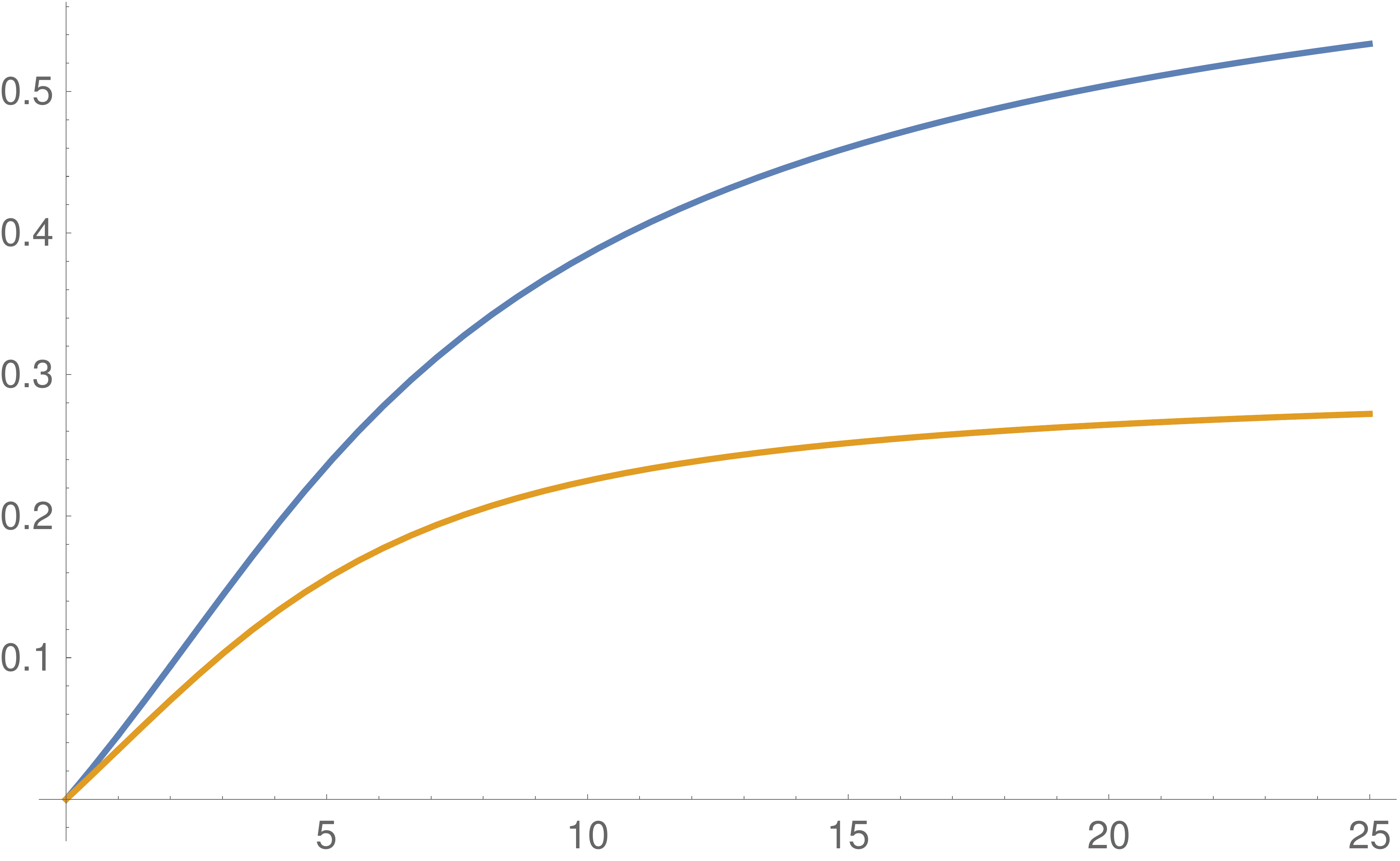}
    \caption{The R\'enyi divergences $R_P^0$ (blue), and $R_Q^0$ (orange) for $p=0.63$ in Example 1. The comparison between $R_P^1$  and $R_Q^1$ yields a similar graph.}
    \label{fig:renyi}
\end{figure}
Thus, by Theorem~\ref{thm:eventual}, there is some $n$ so that $n$ independent samples from $P$ Blackwell dominate $n$ independent samples from $Q$.

%For $p=7/12$, it is easy to see that neither experiment Blackwell dominates the other. The experiment  $Q$ has a larger maximal likelihood ratio. Conversely, $P$ yields a larger expected payoff in the decision problem in which the set of actions equal the set of states, the payoff is 1 if the action matches the state, and zero otherwise (and the decision maker's prior belief is uniform).

%As shown in Figure ??, the experiment $P$ dominates $Q$ in the R\'enyi order. This can also be proven analytically.

%\omer{I think the next proposition should be easy to prove. But I'm not sure we need it.}
The next proposition generalizes the example, showing that a binary experiment $Q$ with the same properties can be constructed for (almost) any experiment $P$.
\begin{proposition}
\label{prop:example}
Let $P$ be a bounded experiment with induced distribution over posteriors $\pi$. Assume that the support of $\pi$ has cardinality at least $3$. Then there is a binary experiment $Q$ such that $P$ and $Q$ are not Blackwell ranked, and $P$ dominates $Q$ in large samples.
\end{proposition}
The proof of this proposition crucially relies on Theorem~\ref{thm:eventual}.

\paragraph{Example 2 and a conjecture by \cite{azrieli2014comment}.}\label{sec:Azrieli}
We next apply Theorem~\ref{thm:eventual} to revisit an example due to \cite{azrieli2014comment} and to complete his analysis. The example provides a simple instance of two experiments that are not ranked in Blackwell order but become so in large samples. Despite its simplicity, the analysis of this example is not straightforward, as shown by \cite{azrieli2014comment}. We will show that applying the R\'enyi order greatly simplifies the analysis and elucidates the logic behind the example.

Consider the following two experiments $P$ and $Q$, parametrized by $\beta$ and $\alpha$, respectively. In each matrix, entries are the probabilities of observing each signal realization given the state $\theta$: 

\begin{equation*}
P: \quad
\begin{array}{cccc}
\hline
\hline
\theta & x_1 & x_2 & x_3\\
\hline
0 & \beta & \frac12 & \frac12 - \beta \\
1 & \frac12 - \beta & \frac12 & \beta \\
\hline
\end{array} 
\qquad \qquad \qquad
Q: \quad 
\begin{array}{ccc}
\hline
\hline
\theta & y_1 & y_2\\
\hline
0 & \alpha & 1-\alpha \\
1 & 1-\alpha & \alpha \\
\hline
\end{array} 
\end{equation*} 

\bigskip

The parameters satisfy $0 \leq \beta \leq 1/4$ and $0 \leq \alpha \leq 1/2$. The experiment $Q$ is a symmetric, binary experiment. The experiment $P$ with probability $1/2$ yields a completely uninformative signal realization $x_2$, and with probability $1/2$ yields an observation from another symmetric binary experiment. As shown by \citet[Claim 1]{azrieli2014comment}, the experiments $P$ and $Q$ are not ranked in the Blackwell order for parameter values $2\beta < \alpha < 1/4 + \beta$.

\cite{azrieli2014comment} points out that a necessary condition for $P$ to dominate $Q$ in large samples is that the R\'enyi divergences are ranked at $1/2$, that is $R_P^1(1/2) > R_Q^1(1/2)$.\footnote{As in his paper, this condition can be written in terms of the parameter values as
\begin{equation*}%\label{eq:AzrieliCondition}
\sqrt{\alpha(1-\alpha)} > \sqrt{\beta(\frac12 - \beta)} + \frac14.
\end{equation*} Thus, when $\alpha = 0.1$ and $\beta = 0$ for example, the experiment $P$ does not Blackwell dominate $Q$ but does dominate it in large samples, as shown by \cite{azrieli2014comment}.} In addition, he conjectures it is also a sufficient condition, and proves it in the special case of $\beta = 0$. We show that for the experiments in the example, the fact that the R\'enyi divergences are ranked at 1/2 is enough to imply dominance in the R\'enyi order, and therefore, by Theorem~\ref{thm:eventual}, dominance in large samples. This settles the above conjecture in the affirmative. 

\begin{proposition}\label{prop:AzrieliConjecture}
In this example, suppose $R_P^1(1/2) > R_Q^1(1/2)$. Then $R_P^1(t) > R_Q^1(t)$ for all $t > 0$ and by symmetry $R_P^0(t) > R_Q^0(t)$, hence $P$ dominates $Q$ in large samples.
\end{proposition}

\subsection{A Quantification of Blackwell Dominance in Large Samples}
%looks good.

The characterization in  Theorem~\ref{thm:eventual} makes it possible to quantify the extent to which one experiment Blackwell dominates another in large samples. We start with the observation that any two experiments, even if not ranked according to dominance in large samples, can be compared by applying different samples sizes. For example, suppose $P$ and $Q$ are not comparable, but $P^{\otimes 50}$ Blackwell dominates $Q^{\otimes 100}$. Then 50 samples from $P$ are more informative than 100 from $Q$, and thus, in an intuitive sense, $P$ is at least twice as informative as $Q$, for large enough samples.

Our formal definition is based on the fact that for any two bounded non-trivial experiments $P$ and $Q$, there exist positive integers $n,m$ such that $P^{\otimes n}$ Blackwell dominates $Q^{\otimes m}$. Reasoning as above, $P$ will be at least $m/n$ times as informative as $Q$ in large samples. We can then consider the largest ratio $m/n$ for which this comparison holds. This leads to a well defined measure of dominance, which we refer to as the {\em dominance ratio} $P/Q$ of $P$ with respect to $Q$:
%\footnote{The proof goes like this: Suppose a uniform prior, and choose $\eps$ so small that the posterior beliefs (about $\theta = 1$) under $Q$ belong to $[2\eps, 1- 2\eps]$. When $k$ is large, there is (at least) probability $0.5-\eps$ each that $P^{\otimes k}$ induces a posterior belief less than $\eps$ or greater than $1-\eps$. Given this, the distribution of posterior beliefs under $P^{\otimes k}$ is a mean-preserving spread of the distribution that puts mass $0.5-\eps$ on each of the two beliefs $\eps$ and $1-\eps$, with remaining mass on $0.5$. This three-point distribution in turn dominates (in the convex order) the two-point distribution supported on $2\eps$ and $1-2\eps$. Thus the distribution of posterior beliefs under $P^{\otimes k}$ dominates the distribution under $Q$, implying Blackwell dominance.} 
%It follows that for given any positive constant $a < 1/k$, $P^{\otimes n}$ dominates $Q^{\otimes \lceil an \rceil}$ for all large $n$, where $\lceil an \rceil$ denotes the smallest integer greater than or equal to $an$. This observation leads to a natural measure of the degree by which, in large samples, $P$ is more informative than $Q$: the factor $a$ can be seen as a lower bound on the information produced by an additional observation from $P$, relative to $Q$.
%With this motivation in mind, we define the {\em dominance ratio} $P/Q$ of an experiment $P$ with respect to $Q$ as
\begin{align*}
    P/Q = \sup \left\{\frac{m}{n}\,:\,P^{\otimes n} \succeq Q^{\otimes m}\right\}.
\end{align*}
Thus, in large samples, each observation from $P$ contributes at least as much as $P/Q$ observations from $Q$. 

An immediate consequence of Theorem~\ref{thm:eventual} is the following characterization of $P/Q$ in terms of the R\'enyi divergences of the two experiments.

\begin{proposition}
\label{prop:index}
Let $P$ and $Q$ be non-trivial, bounded experiments. Then 
\begin{align*}
    P/Q = \inf_{\substack{\theta\in\{0,1\}\\t > 0}} \frac{R_P^\theta(t)}{R_Q^\theta(t)}.
\end{align*}
Furthermore, the dominance ratio $P/Q$ is always positive.\footnote{This characterization, together with Theorem~\ref{thm:eventual}, implies that the following natural alternative definition of $P/Q$ is equivalent:
\begin{align*}
    P/Q = \sup\Big\{a > 0 \,:\, P^{\otimes n} \succeq Q^{\otimes \lceil an \rceil } \text{~for \emph{all} $n$ large enough} \Big\}
\end{align*}
where $\lceil an \rceil$ denotes the smallest integer greater than or equal to $an$.}
\end{proposition}

As discussed, $P/Q$ can be interpreted as an asymptotic \emph{lower bound} on the information produced by one observation from $P$ relative to $Q$. On the other hand, we also have the asymptotic \emph{upper bound} $(Q/P)^{-1}$, where $Q/P$ is the dominance ratio of $Q$ with respect to $P$. We remark that the two bounds are in general (in fact, generically) not equal. However, Proposition~\ref{prop:index} shows that $P/Q \leq (Q/P)^{-1}$ always holds.

%\footnote{In fact, we have the more general inequality $(P/Q) \cdot (Q/R) \leq P/R$ for all non-trivial bounded experiments $P, Q, R$.} 

% We leave a more detailed study of the dominance ratio to future research.

%Thus R\'enyi divergences play a decisive role in the dominance ratio: when  $R_P^\theta(t)$ is uniformly larger than $R_Q^\theta(t)$ then each observation from $P$ is equivalent to many observations from $Q$, in large samples.
\subsection{The Blackwell Order in the Presence of Additional Information}
\label{sec:additional}

The large sample order compares the informativeness of repeated experiments. A related problem is to compare the informativeness of one-shot experiments when additional independent sources of information may be present.
 
Consider a decision maker choosing which of two experiments $P$ and $Q$ to conduct, \emph{on top of} an independent source of information $R$.  The resulting choice is between the compound experiments $P \otimes R$ and $Q \otimes R$. It is intuitive, and immediate from Blackwell's garbling characterization, that if $P$ dominates $Q$ in the Blackwell order, then the same relation must hold between the two compound experiments.

One might expect that if $P$ and $Q$ are incomparable, then no additional independent experiment $R$ can make the compound experiments comparable. Instead, we show that $P \otimes R$ can dominate $Q \otimes R$ even though the two original experiments $P$ and $Q$ were not comparable. Moreover, for generic experiments, this occurs precisely when $P$ has higher R\'enyi divergences than $Q$.
\begin{proposition}
\label{thm:marginal}
  Let $P$ and $Q$ be a generic pair of bounded experiments. Then the following are equivalent:
  \begin{enumerate}
      \item There exists a bounded experiment $R$ such that $P \otimes R \succeq Q \otimes R$.
      \item $P$ dominates $Q$ in the R\'enyi order.
  \end{enumerate}
\end{proposition}

 Proposition~\ref{thm:marginal} suggests that in general, whether two experiments are Blackwell ordered depends on \emph{what} additional sources of information are available. We note that whenever an experiment $R$ makes $P$ dominant over $Q$ (when each is combined with $R$), then the same holds for any experiment $R'$ that is more informative than $R$. It is an interesting question for future work to fully characterize the set of experiments $R$ that make $P$ dominant. 
 
 Proposition \ref{thm:marginal} follows by combining the characterization in Theorem~\ref{thm:eventual} together with the observation that if $P$ dominates $Q$ in the large sample order, then there exists an $R$ such that $P \otimes R$ Blackwell dominates $Q \otimes R$. The latter fact is a consequence of an order-theoretic result from the quantum information literature \citep[][see Lemma~\ref{lem:catalytic} in the appendix]{duan2005multiple, fritz2017resource}.

%great!
 %This theorem follows from the characterization in Theorem~\ref{thm:eventual}, and from a fact relating the large sample order to what is known in the quantum information literature as the {\em catalytic order}.\footnote{Let $\leq$ be a preorder defined on a set $S$, and let $\circ$ be an associative, commutative binary operation on $S$ such that $x \leq y$ implies $x \circ z \leq y \circ z$. The associated catalytic preorder $\leq_c$ on $S$ is defined as follows: $x \leq_c y$ if there exists a $z$ such that $x \circ z \leq y \circ z$. In our example $S$ is the set of bounded experiments, and the binary operation is $\otimes$.} This fact \citep[][see Lemma~\ref{lem:catalytic}]{duan2005multiple, fritz2017resource} implies that if $P$ dominates $Q$ in the large sample order, then there exists an $R$ such that $P \otimes R$ Blackwell dominates $Q \otimes R$.
 %Note that the second condition in the theorem above is weaker than dominance in the R\'enyi order, since we only require a weak inequality.

\section{A Characterization of Additive Divergences}
\label{sec:divergences}

 In this section we apply the characterization of Blackwell dominance in large samples to study measures for quantifying the degree of dissimilarity between distributions, also known as \textit{divergences}. Examples of divergences include total variation distance, the Hellinger distance, the Kullback-Leibler divergence, R\'enyi divergences, and more general $f$-divergences.

 A key property of R\'enyi divergences is additivity. Consider two domains $\Omega_1$ and $\Omega_2$, a pair of measures $\mu_1,\nu_1$ defined on $\Omega_1$, and a pair of measures $\mu_2,\nu_2$ on $\Omega_2$. Additivity states that when the two domains are considered in conjunction, the divergence between the product measures $\mu_1 \times \mu_2$ and $\nu_1 \times \nu_2$, which are both defined on $\Omega_1 \times \Omega_2$, is the sum of the divergences of the two pairs. In words, this condition says that the total divergence of two unrelated pairs should not change when they are considered together as a bundle.
 
 Another property of R\'enyi divergences, which it in fact shares with all the above examples of divergences, is the {\em data processing inequality}, which captures the idea that discarding some information decreases dissimilarity. 
 
 We show that every additive divergence that satisfies the data-processing inequality is an integral of R\'enyi divergences. %Both the Kullback-Leibler divergence and the Hellinger distance are special cases of R\'enyi divergences for suitable choices of the parameter $t$. 
 The proof relies on the characterization of the large sample order together with functional analytic techniques. Since this result does not assume any functional form of the divergence, it improves over the existing characterizations such as in \cite{renyi1961measures} and \cite{csiszar2008axiomatic}.
 
 The result has potential applications for modeling experiments as economic commodities. In recent years, there has been growing interest in modeling the cost and pricing of information. By interpreting a divergence as a cost function over experiments, additivity reflects an assumption of constant marginal costs in information production \citep[an assumption discussed in detail in][]{pomatto2018cost}. By interpreting a divergence as a pricing function over experiments, additivity captures a notion of linearity, appropriate for pricing information in competitive markets.

\subsection{Additive Divergences}

 Given a Polish space $\Omega$, we denote by $\B(\Omega)$ its Borel $\sigma$-algebra and by $\Delta(\Omega)$ the collection of Borel probability measures on $\B(\Omega)$. Given another Polish space $\Xi$, a measurable function $f \colon \Omega \to \Xi$ and a probability measure $\mu \in \Delta(\Omega)$, we denote by $f_*(\mu)$ the push-forward probability measure in $\Delta(\Xi)$ defined as $[f_*(\mu)](E) = \mu(f^{-1}(E))$ for all $E \in \B(\Xi)$.

 Consider, for each $\Omega$, a map
 \[
   D_\Omega \colon \Delta(\Omega) \times \Delta(\Omega) \to \R_+ \cup \{+{\infty}\},
 \]
 and let $D = (D_\Omega)$ be the collection obtained by varying $\Omega$. We say $D$ is a {\em divergence} if $D_\Omega(\mu,\mu)=0$ for all $\Omega$ and all $\mu \in \Delta(\Omega)$.
 
 A divergence satisfies the {\em data processing inequality} if for any measurable $f \colon \Omega \to \Xi$ it holds that 
 $$
   D_{\Xi}(f_*(\mu),f_*(\nu)) \leq D_\Omega(\mu,\nu).
 $$ 
 The data processing inequality captures the idea that the distributions of two random variables $X$ and $Y$ are at least as dissimilar as those of $f(X)$ and $f(Y)$; applying a common deterministic mapping $f$ can only make the distributions more similar.\footnote{Note that the data processing inequality implies that $D$ is invariant to measurable isomorphisms: If $f$ is a bijection then $D_{\Xi}(f_*(\mu),f_*(\nu)) = D_\Omega(\mu,\nu)$. Thus the dissimilarity between measures does not depend on the particular labelling of the domain.} It is a natural concept in signal processing and information theory, and closely related to the Blackwell order over experiments. Indeed, we can see a pair of probability measures as an experiment $(P_0,P_1)$, and hence a divergence $D$ as a functional over experiments. The data-processing inequality states that the value of $D$ decreases when applying a deterministic garbling. %When $D$ is additive, the data processing inequality corresponds to monotonicity with respect to the Blackwell order. 

 We say that the divergence $D$ is {\em additive} if 
 \begin{align*}
   D_{\Omega \times \Xi}(\mu_1 \times \mu_2, \nu_1 \times \nu_2) = D_{\Omega}(\mu_1,\nu_1) + D_\Xi(\mu_2,\nu_2).
 \end{align*} 
 We will henceforth drop the subscript from $D_\Omega(\mu,\nu)$, and write $D(\mu,\nu)$ whenever there is no risk of confusion.
 
 We call a pair $\mu,\nu$ of measures as {\em bounded} if there exists an $M > 0$ such that for any measurable $A \subseteq \Omega$, $\nu(A) \geq  \mu(A)/M$ and $\mu(A) \geq \nu(A)/M$. Equivalently, $\dd \mu/\dd\nu$ is supported on $[1/M,M]$, and hence bounded from above and bounded away from 0. We will restrict our attention to divergences that take finite values on bounded pairs of experiments.
 
\subsection{Representation Theorem}

Our representation theorem shows that all additive divergences that are finite on bounded experiments arise from linear combinations of R\'enyi divergences.
\begin{theorem}\label{thm:divergences}
Let $D$ be an additive divergence that satisfies the data processing inequality and is finite on bounded experiments. Then there exist two finite Borel measures $m_0,m_1$ on $[1/2,\infty]$ such that for every bounded pair $\mu,\nu$ it holds that
\begin{align}
\label{eq:D-rep}
   D(\mu,\nu) = \int_{[1/2,\infty]}R_t(\mu \Vert \nu)\,\dd m_0(t) + \int_{[1/2,\infty]}R_t(\nu \Vert \mu)\,\dd m_1(t),
\end{align}
with $R_t$ given by \eqref{eq:R_t} and \eqref{eq:R_1}.
\end{theorem}

Varying the two measures $m_0$ and $m_1$ leads to some important special cases. When both are finitely supported, $D$ is a linear combination of R\'enyi divergences. Any additive divergence $D$ (finite on bounded experiments) is hence a limit of such combinations. When $m_0$ and $m_1$ are Dirac probability measures concentrated on $1$, $D$ reduces to twice the Jensen-Shannon divergence, which is the symmetric counterpart of the Kullback-Leibler divergence. When instead $m_0$ is a Dirac probability measure concentrated on $1$ and $m_1$ is set to have total mass zero, $D$ reduces to the Kullback-Leibler divergence.

Note that the lower integration bound in \eqref{eq:D-rep} is $1/2$. This is because, as discussed, the values of $R_t(\mu \Vert \nu)$ are related to the values of $R_{1-t}(\nu \Vert \mu)$. Hence it suffices to consider values of $t$ above $1/2$.

\paragraph{Proof Sketch of Theorem~\ref{thm:divergences}.} The first key idea is to see a bounded pair of probability measures as a bounded experiment $(P_0,P_1)$, and hence see a divergence $D$ as a functional over experiments. When $D$ is additive, the data processing inequality implies monotonicity with respect to the Blackwell order. 

The next crucial step is to leverage Theorem~\ref{thm:eventual} to show that additivity renders $D$ monotone in the R\'enyi order. Indeed, if $(P_0,P_1)$ dominates $(Q_0,Q_1)$ in the R\'enyi order, then, by Theorem~\ref{thm:eventual}, there exists a number $n$ of repetitions such that $(P^n_0,P^n_1)$ dominates $(Q^n_0,Q^n_1)$ in the Blackwell order. Hence, by combining Blackwell monotonicity and additivity, we obtain that $D$ must satisfy
\[
   nD(P_0,P_1) = D(P_0^n,P_1^n) \geq D(Q_0^n,Q_1^n) = nD(Q_0,Q_1).
\]
Hence, $D$ is monotone in the R\'enyi order. 

We deduce from this that $D$ is a monotone functional $F(R^0_P,R^1_P)$ of the R\'enyi divergences of the experiment. Additivity of $D$ implies $F$ is also additive. We then use tools from functional analysis to show that $F$ extends to a positive linear functional, leading to the integral representation of Theorem~\ref{thm:divergences}.

\section{Proof of Theorem~\ref{thm:eventual}}\label{sec:proof}

The proof of Theorem~\ref{thm:eventual} is organized as follows. In \S\ref{sec:Renyi decision problems} we first show that the R\'enyi order is necessary for the large sample order. The remaining subsections demonstrate sufficiency. In \S\ref{sec:reduction} we provide a novel characterization of Blackwell dominance, showing that it is equivalent to first-order stochastic dominance of appropriate statistics of the two experiments. \S\ref{sec:application-to-the-renyi-order} applies  this observation, together with techniques from large deviations theory. Omitted proofs are deferred to the appendix.

\subsection{Dominance in Large Samples Implies Dominance in the R\'enyi Order}\label{sec:Renyi decision problems}

%We first show that in Theorem~\ref{thm:eventual} (i) implies (ii).

As discussed above, the comparison of R\'enyi divergences between two experiments is independent of the number of samples. Thus it suffices to show that the R\'enyi order extends the strict Blackwell order.\footnote{Since by assumption the two experiments $P$ and $Q$ form a generic pair, Blackwell dominance of $P^{\otimes n}$ over $Q^{\otimes n}$ necessarily implies strict Blackwell dominance.} We do this by constructing decision problems with the property that higher expected payoff in these problems translates into higher R\'enyi divergences.

For each $t > 1$, the function $v_1(p) = 2p^t(1-p)^{1-t}$ defined for $p \in (0,1)$ is strictly convex, because its second derivative in $p$ is $2t(t-1)p^{t-2}(1-p)^{-1-t}$. Thus $v_1(p)$ is the indirect utility function induced by some decision problem. Moreover, we have that
\begin{equation}\label{eq:Renyi decision problem}
\int_0^1 v_1(p)\,\dd \pi(p) = \int_{\Omega} \left(\frac{\dd P_1(\omega)}{\dd P_0(\omega)}\right)^{t-1} \,\dd P_1(\omega) = \ee^{(t-1)R_P^1(t)}.
\end{equation}
To see this, recall that $\pi_\theta$ is the distribution over posteriors induced by $P$, conditional on state $\theta \in \{0,1\}$, and that 
\begin{equation}\label{eq:dpi and dpi1}
  \dd \pi(p) = \frac{1}{2} (\dd \pi_1(p) + \dd \pi_0(p))~~~~~ \text{and}~~~~~\dd \pi_1(p) = \frac{p}{1-p} \,\dd \pi_0(p).
\end{equation}
Thus $\dd \pi(p) = \frac{1}{2p} \,\dd \pi_1(p)$, which allows us to write
\[
\int_0^1 v_1(p)\,\dd \pi(p) = \int_0^1 2p^t(1-p)^{1-t} \cdot \frac{1}{2p} \,\dd \pi_1(p) = \int_0^1 \left(\frac{p}{1-p}\right)^{t-1} \,\dd \pi_1(p). 
\]
The first equality in \eqref{eq:Renyi decision problem} then follows from a change of variable from signal realizations $\omega$ to posterior beliefs $p = \frac{\dd P_1(\omega)}{\dd P_1(\omega) + \dd P_0(\omega)}$ (with the probability measure changing from $P_1$ to $\pi_1$, holding fixed the true state $\theta = 1$). 

The second equality in \eqref{eq:Renyi decision problem} follows from the definition of R\'enyi divergences. Thus \eqref{eq:Renyi decision problem} holds, which shows that in the decision problem with indirect utility function $v_1(p)$, the ex-ante expected payoff is a monotone transformation of the R\'enyi divergence $R_P^1(t)$. Hence, experiment $P$ yields higher expected payoff in this decision problem than $Q$ if and only if $R_P^1(t) > R_Q^1(t)$.

Similarly, for $t \in (0, 1)$ we consider the indirect utility function $v_2(p)=-2p^t(1-p)^{1-t}$, which is now strictly convex due to the negative sign (its second derivative is $2t(1-t)p^{t-2}(1-p)^{-1-t}$). Then
\[
\int_0^1 v_2(p)\,\dd \pi(p) = - \ee^{(t-1)R_P^1(t)}
\]
is again a monotone transformation of the R\'enyi divergence. So $P$ yields higher expected payoff in this decision problem only if $R_P^1(t) > R_Q^1(t)$. 

For $t = 1$, we consider the indirect utility function $v_3(p) = 2p \log(\frac{p}{1-p})$, which is strictly convex with a second derivative of $2p^{-1}(1-p)^{-2}$. We have
\[
\int_0^1 v_3(p)\,\dd \pi(p) = \int_0^1  \log\left(\frac{p}{1-p}\right) \,\dd \pi_1(p) = \int_{\Omega} \log\left(\frac{\dd P_1(\omega)}{\dd P_0(\omega)}\right) \,\dd P_1(\omega) = R_P^1(1).
\]
Thus $P$ yields higher expected payoff in this problem if and only if $R_P^1(1) > R_Q^1(1)$. 

Summarizing, the above family of decision problems shows that $P$ strictly Blackwell dominates $Q$ only if $R_P^1(t) > R_Q^1(t)$ for all $t > 0$. Since the two states are symmetric, another set of necessary conditions is that $R_P^0(t) > R_Q^0(t)$ for all $t > 0$. Hence dominance in the R\'enyi order is necessary for Blackwell dominance and (due to additivity of R\'enyi divergences) also for dominance in large samples. 

\subsection{Repeated Experiments and Log-Likelihood Ratios}\label{sec:LLR}

We turn to the proof that dominance in the R\'enyi order is (generically) sufficient for dominance in large samples. Recall that $P^{\otimes n}$ Blackwell dominates $Q^{\otimes n}$ if and only if the former induces a distribution over posterior beliefs that is a mean-preserving spread of the latter. However, the distribution over posteriors induced by a product experiment can be difficult to analyze directly. A more suitable approach consists in studying the distribution of the induced log-likelihood ratio
\[
    \log \frac{\dd P_\theta \hfill}{\dd P_{1-\theta}}.
\]
As is well known, given a repeated experiment $P^{\otimes n} = (\Omega^n, P^n_0, P^n_{1})$, its log-likelihood ratio satisfies, for every realization $\omega = (\omega_1,\ldots,\omega_n)$ in $\Omega^n$,
\[
    \log \frac {\dd P^n_1}{\dd P^n_{0}}(\omega) = \sum_{i=1}^n \log \frac{\dd P_1}{\dd P_{0}}(\omega_i).
\]
Moreover, the random variables
\[
    X_{i}(\omega) = \log \frac{\dd P_1}{\dd P_{0}}(\omega_i) \quad i=1,\ldots,n
\]
are i.i.d.\ under $P_\theta^{n}$, for $\theta \in \{0,1\}$. Focusing on the distributions of log-likelihood ratios will allow us to transform the study of repeated experiments to the study of sums of i.i.d.\ random variables. %We can then apply the results from \S\ref{sec:FOSD}.
 
%Given a belief $p \in [0,1]$ that the state is $1$, $\log{\frac{p}{1-p}}$ is the corresponding log-likelihood ratio (henceforth, LLR). We collect here some useful facts on the distributions of LLRs induced by an experiment.

%Let $P$  be an experiment inducing the distribution over posteriors $\pi$. A standard calculation shows that the two conditional distributions over posteriors are mutually absolutely continuous and their derivative can be taken to be
 
%\begin{align}
%  \label{eq:rn}
%  \frac{\dd \pi_1}{\dd \pi_0}(p) = \frac{p}{1-p}.
%\end{align}
%
  
%While studying the  next result shows how Blackwell dominance between two experiments can be described, Blackwell dominance is related to {\em first} order stochastic dominance of the distributions of log-likelihood ratios.

\subsection{From Blackwell Dominance to First-Order Stochastic Dominance}

\label{sec:reduction}
Expressing posterior beliefs in terms of log-likelihood ratios simplifies the analysis of repeated experiments. However, it is not obvious that the Blackwell order admits a simple interpretation in this domain. 

We provide a novel characterization of the Blackwell order, expressed in terms of the distributions of the log-likelihood ratios. Given two experiments $P = (\Omega,P_0,P_1)$ and $Q = (\Xi,Q_0,Q_1)$ we denote by $F_\theta$ and $G_\theta$, respectively, the cumulative distribution function of the log-likelihood ratios conditional on state $\theta$. That is,
%\begin{equation*}
%  F_1(u) = P_1\left(\left\{ \log\frac{\dd P_1}{\dd P_0} \leq u\right\}\right)
%  \text{\quad and \quad}
%  F_0(u) = P_0 \left(\left\{\log\frac{\dd P_0}{\dd P_1} \leq u\right\}\right).
%\end{equation*}
\begin{equation}\label{eq:def-F}
  F_\theta(a) = P_\theta\left(\left\{ \log\frac{\dd P_\theta \hfill}{\dd P_{1-\theta}} \leq a\right\}\right) ~~~\text{for all}~~ a \in \R, ~\theta \in \{0,1\}.
\end{equation}
The c.d.f.\ $G_\theta$ is defined analogously using $Q_\theta$.

We associate to $P$ a new quantity, which we call the \textit{perfected log-likelihood ratio} of the experiment.  Define
\[
    \tilde{L}_1 = \log\frac{\dd P_1}{\dd P_0} - E
\]
where $E$ is a random variable that, under $P_1$, is independent from $\log\frac{\dd P_1}{\dd P_0}$ and distributed according to an exponential distribution with support $\R_+$ and cumulative distribution function $1 - \ee^{-x}$ for all $x \geq 0$. We denote by $\Tilde{F}_1$ the cumulative distribution function of $\tilde{L}_1$ under $P_1$. That is, $\Tilde{F}_1(a) = P_1(\{\tilde{L}_1 \leq a\})$ for all $a \in \R$.

More explicitly, $\Tilde{F}_1$ is the convolution of the distribution $F_1$ with the distribution of $-E$, and thus can be defined as
\begin{equation}\label{eq:def-TildeF}
    \Tilde{F}_1(a) = \int_\R P_1(\{-E \leq a-u\}) \,\dd F_1(u) = F_1(a) + \ee^a \int_{(a,\infty)} \ee^{-u} \,\dd F_1(u).
\end{equation}
The next result shows that the Blackwell order over experiments can be reduced to first-order stochastic dominance of the corresponding perfected log-likelihood ratios.

%Recall that a random variable $X$ first-order stochastically dominates $Y$, denoted $X \succeq_1 Y$, if for any bounded increasing $f \colon \R \to \R$ it holds that $\E{f(X)} \geq \E{f(Y)}$.

\begin{theorem}\label{prop:reduction}
Let $P$ and $Q$ be two experiments, and let $\Tilde{F}_1$ and $\Tilde{G}_1$, respectively, be the associated distributions of perfected log-likelihood ratios. Then
\[
    P \succeq Q \text{\quad if and only if \quad} \tilde{F}_1(a) \leq \tilde{G}_1(a) \text{~for all~} a \in \R.
\]
\end{theorem}

\begin{proof}
Let $\pi$ and $\tau$ be the distributions over posterior beliefs induced by $P$ and $Q$, respectively. As is well known, Blackwell dominance is equivalent to the requirement that $\pi$ is a mean-preserving spread of $\tau$. Equivalently
%, if we slightly abuse notation and let $\pi$ and $\tau$ also denote the c.d.f.\ of posterior beliefs, then 
the functions defined as
\begin{equation}\label{eq:reduction1}
    \Lambda_\pi(p) = \int_{[0,p]} (p-q) \,\dd \pi(q) \text{\quad and \quad } \Lambda_\tau(p) = \int_{[0,p]} (p-q) \,\dd \tau(q)
\end{equation}
must satisfy $\Lambda_\pi(p) \geq \Lambda_\tau(p)$ for every $p \in (0,1)$.

We now express \eqref{eq:reduction1} in terms of the distributions of log-likelihood ratios $F_1$ and $G_1$. We have
\begin{equation}\label{eq:reduction2}
    \Lambda_\pi(p)  = p\left(1 - \int_{(p,1]} 1 \,\dd \pi(q)\right) - \int_{[0,p]} q \,\dd \pi(q).
\end{equation}
To transform the relevant integrals into those that condition on state $1$, we recall that \eqref{eq:dpi and dpi1} implies $\dd \pi(q) = \frac{1}{2q} \,\dd \pi_1(q)$. We then obtain from \eqref{eq:reduction2} that
\[
    2\Lambda_\pi(p)  = p\left(2 - \int_{(p,1]} \frac{1}{q} \,\dd \pi_1(q)\right) - \int_{[0,p]} \,\dd \pi_1(q).
\]
Next, we change variable from posterior beliefs to log-likelihood ratios. Letting $a = \log\frac{p}{1-p}$ and accordingly $u = \log\frac{q}{1-q}$, we have
\begin{equation}\label{eq:reduction3}
    2\Lambda_\pi(p) = \frac{\ee^a}{1+\ee^a}\left(2 - \int_{(a,\infty)} \frac{1+\ee^u}{\ee^{u}} \,\dd F_1(u)\right) - F_1(a).
\end{equation}
Since
\[
    \int_{(a,\infty)} \frac{1+\ee^u}{\ee^{u}} \,\dd F_1(u) = \int_{(a,\infty)} \ee^{-u}\,\dd F_1(u) + 1-F_1(a),
\]
\eqref{eq:reduction3} leads to
\[
    2\Lambda_\pi(p)  = \frac{\ee^a}{1+\ee^a}-\frac{F_1(a)}{1+\ee^a} - \frac{\ee^a}{1+\ee^a} \int_{(a,\infty)} \ee^{-u}\,\dd F_1(u) = \frac{\ee^a}{1+\ee^a} - \frac{\Tilde{F}_1(a)}{1+\ee^a},
\]
where the final equality follows from \eqref{eq:def-TildeF}. It then follows that $\Lambda_\pi(p) \geq \Lambda_\pi(p)$ if and only if $\tilde{F}_1(a) \leq \tilde{G}_1(a)$ for $a = \log\frac{p}{1-p}$. Requiring this for all $p \in (0,1)$ yields the theorem. 
\end{proof}

Intuitively, transferring probability mass from lower to higher values of $\log (\dd P_{\theta} / \dd P_{1-\theta})$ leads to an experiment that, conditional on the state being $\theta$, is more likely to shift the decision maker's beliefs towards the correct state. Hence, one might conjecture that Blackwell dominance of the experiments $P$ and $Q$ is related to stochastic dominance of the distributions $F_\theta$ and $G_\theta$. However, since the likelihood ratio $\dd P_1 / \dd P_0$ must satisfy the change of measure identity $\int \frac{\dd P_0}{\dd P_1} \,\dd P_1 = 1$, the distribution $F_1$ must satisfy
\[
    \int_{\R} \ee^{-u} \,\dd F_1(u) = 1.
\]
Because the function $\ee^{-u}$ is strictly decreasing and convex, and the same identity must hold for $G_1$, it is impossible for $F_1$ to stochastically dominate $G_1$. Theorem~\ref{prop:reduction} shows that a more useful comparison is between the perfected log-likelihood ratios.\footnote{It might appear puzzling that two distributions $F_1$ and $G_1$ that are not ranked by stochastic dominance become ranked after the addition of the same independent random variable. In a different context and under different assumptions, the same phenomenon is studied by \cite*{pomatto2018stochastic}.}
%In the next proposition we retain our notation of two experiments $P$ and $Q$, with corresponding conditional cdfs $F_\theta$ and $G_\theta$. Recall that the standard exponential distribution is  supported on $\R_+$ with density $\ee^{-x}$.
%\begin{proposition}
%Let $X$ and $Y$ be independent random variables with cdfs  $F_1$ and $G_1$, respectively. Let $E$ be an independent standard exponential random variable. Then $P \succeq Q$ if and only if $X-E \succeq_1 Y-E$.
%\end{proposition}

The next lemma simplifies the study of  perfected log-likelihood ratios, by showing that their first-order stochastic dominance can be deduced from comparisons of the original distributions $F_\theta$ and $G_\theta$ over subintervals.

%Despite the fact that $F_\theta$ cannot stochastically dominate $G_\theta$, the next lemma shows that when $F_1(a)$ is smaller that $G_1(a)$ for values of $a$ above some threshold, then $\Tilde{F}_1(a)$ is smaller than $\tilde{G}_1(a)$ within the same range. Likewise, when $F_0(a)$ is smaller than $G_0(a)$ for large $a$, then $\Tilde{F}_1(-a)$ is also smaller than $\tilde{G}_1(-a)$. The intuition is that a dominating experiment should have higher likelihood ratios for state $\theta$, conditional on $\theta$.
%The next lemma will allow us to reduce stochastic dominance between perfected log-likelihood ratios---i.e., the ranking of $\tilde F$ and $\tilde G$---to a ranking of the cumulative distributions $F_\theta$ and $G_\theta$ on a subinterval of their domains.

\begin{lemma}\label{lem:signs}
Consider two experiments $P$ and $Q$. Let $F_\theta$ and $G_\theta$, respectively, be the distributions of the corresponding log-likelihood ratios, and $\tilde{F}_1$ and $\tilde{G}_1$ be the distributions of the  perfected log-likelihood ratios. The following holds:
\begin{enumerate}
    \item If $F_1(a) \leq G_1(a)$ for all $a \geq 0$, then $\Tilde{F}_1(a) \leq \Tilde{G}_1(a)$ for all $a \geq 0$.
    \item If $F_0(a) \leq G_0(a)$ for all $a \geq 0$, then  $\Tilde{F}_1(a) \leq \Tilde{G}_1(a)$ for all $a \leq 0$.
\end{enumerate}
\end{lemma}

\subsection{Large Deviations}

%We now illustrate how dominance in the R\'enyi order translates into properties of the log-likelihood ratios, and provide a sketch of the proof of Theorem~\ref{thm:eventual}. 
%In what follows, given a bounded random variable $X$ we denote by $\max[X] = \min \{a \,:\, \Pr{X\leq a}=1\}$ the essential maximum of $X$; this is the maximum of the support of its distribution. %We define $\min[X]$ analogously.\footnote{That is, $\min[X] = \max \{a \,:\, \Pr{X\geq a}=1\}$}
%We denote by $K_X(t) = \log\mathbb{E}[\ee^{t X}]$ its {\em cumulant generating function.}

The main step in the proof of Theorem~\ref{thm:eventual} relies on the theory of large deviations. Large deviations theory studies low probability events, and in particular the odds with which an i.i.d.\ sum deviates from its expectation. The Law of Large Numbers implies that for a random variable $X$, the probability of the event $ \{X_1 + \cdots + X_n > na\}$ is low for $a > \mathbb{E}[X]$ and large $n$, where $X_1, \dots, X_n$ are i.i.d.\ copies of $X$. A crucial insight due to \cite{cramer1938nouveau} is that the order of magnitude of the probability of this event is determined by the {\em cumulant generating function} of $X$, defined as $$K_X(t) = \log\mathbb{E}[\ee^{t X}]$$ for every $t \in \R$.

As is well known, $K_X$ is strictly convex whenever $X$ is not a constant. We denote by
\begin{equation}\label{eq:fenchel}
    K_{X}^*(a) = \sup_{t \in \R} t\cdot a - K_{X}(t) ~~~a \in \R,
\end{equation}
its Fenchel conjugate. Two facts we will repeatedly apply are that for every $a \in (\min[X],\max[X])$ the problem \eqref{eq:fenchel} has a unique solution $t \in \R$, and such $t$ is non-negative if and only if $a \geq \mathbb{E}[X]$. Moreover, $K_{X}^* \geq 0 \cdot a - K_X(0) = 0$ is non-negative. % $K_X^*(a)$ is continuous (in fact, analytic) where ever it is finite.

Cram\'er's Theorem establishes that for each threshold $a > \mathbb{E}[X]$, the exponential rate at which the probability of the event $ \{X_1 + \cdots + X_n > na\}$ vanishes with $n$ is equal to the value $K_{X}^*(a)$ taken by the Fenchel conjugate at $a$. In this paper we are interested in comparing the probabilities of large deviations across different random variables. Consider, to this end, two random variables $X$ and $Y$ and a threshold $a$ strictly greater than $\mathbb{E}[X]$ and $\mathbb{E}[Y]$. If
\[
    K_Y^*(a) > K_X^*(a),
\]
then the probability of the event $ \{X_1 + \cdots + X_n > na\}$ vanishes more slowly than the probability of the event $ \{Y_1 + \cdots + Y_n > na\}$ . Thus there exists $n$ sufficiently large such that
\[
    \Pr{X_1 + \cdots + X_n > na} \geq \Pr{Y_1 + \cdots + Y_n > na}.
\]
The next proposition establishes a general version of this fact, while also providing a specific number of repetitions sufficient to rank the probability of the two events.

%L: I change $B$ back to $b$, because it is what we have in the proofs. But either one is fine with me.
\begin{proposition}\label{prop:large-deviations-compare}
Let $X$ and $Y$ be random variables taking values in $[-b,b]$ and let $X_1, \dots, X_n$, $Y_1, \dots, Y_n$ be i.i.d.\ copies of $X$ and $Y$ respectively. Suppose $a \geq \mathbb{E}[Y]$, and $\eta > 0$ satisfies $K_Y^*(a)-\eta>K_X^*(a+\eta)$. Then for all 
$
  n \geq 4b^2(1+\eta)\eta^{-3},
$ 
it holds that
\begin{equation}\label{eq:large-deviation-compare}
    \Pr{X_1 + \cdots + X_n > na} \geq \Pr{Y_1 + \cdots + Y_n > na}.
\end{equation}
\end{proposition}

The condition $K_Y^*(a)-\eta>K_X^*(a+\eta)$ ensures that the rate at which the probability of the events $\{Y_1 + \ldots + Y_n > na\}$ vanish with $n$ is larger by a factor of at least $\eta$ than the rate of the events $\{X_1 + \ldots + X_n > n(a+\eta)\}$. Larger values of $\eta$ make this condition more demanding, and imply that a smaller number of repetitions is sufficient to guarantee \eqref{eq:large-deviation-compare} to hold.
%A similar statement is proved in \citet[][Lemma 2]{aubrun2008catalytic}. In the appendix we establish a stronger result, which additionally provides explicit estimates for the threshold $N$ described below.

\subsection{Application to the R\'enyi Order}\label{sec:application-to-the-renyi-order}

Now consider two experiments $P = (\Omega,P_0,P_1)$ and $Q = (\Xi,Q_0,Q_1)$. Denote the corresponding log-likelihood ratios
\[
    X^\theta = \log \frac{\dd P_\theta \hfill}{\dd P_{1-\theta}} \text{\quad and \quad } Y^\theta = \log\frac{\dd Q_\theta \hfill}{\dd Q_{1 - \theta}}
\]
defined over the probability spaces $(\Omega,P_\theta)$ and $(\Xi,Q_\theta)$, respectively. Thus, for instance, $X^1$ is the log-likelihood ratio of state 1 to state 0, distributed conditional on state 1, and $X^0$ is the log-likelihood ratio of state 0 to 1, distributed conditional on state 0. 

The cumulant generating function of the log-likelihood ratio is a simple transformation of the R\'enyi divergences, as defined in \eqref{eq:R_t}, \eqref{eq:R_1} and \eqref{eq:Renyi divergence}:
\begin{equation}\label{eq:K-R}
    K_{X^{\theta}}(t) = t \cdot R_{P}^\theta(t+1).
\end{equation}
Likewise $K_{Y^{\theta}}(t) = t \cdot R_{Q}^\theta(t+1)$. Hence, if $P$ dominates $Q$ in the R\'enyi order then the following relation must hold between the cumulant generating functions:
\begin{align}
    K_{X^\theta}(t) > K_{Y^\theta}(t)~~~&\text{ for } t > 0 \label{eq:KX-KYa}\\
    K_{X^\theta}(t) < K_{Y^\theta}(t)~~~&\text{ for } -1 < t < 0. \label{eq:KX-KYb}
\end{align}
At $t = 0$ we have $K_{X^\theta}(0) = K_{Y^{\theta}}(0) = 0$, but $K_{X^\theta}'(0) > K_{Y^{\theta}}'(0)$ must hold by \eqref{eq:K-R} and the assumption that $R_{P}^\theta(1) > R_{Q}^{\theta}(1)$. It is well known that $K_{X^\theta}'(0) = \mathbb{E}[X^\theta]$, which by definition is the Kullback-Leibler divergence between $P^\theta$ and $P^{1-\theta}$. Hence we also have
\[
\mathbb{E}[X^\theta] > \mathbb{E}[Y^\theta] > 0.\footnote{Throughout the proof we assume $Q$ is a non-trivial experiment, so that $\mathbb{E}[Y^\theta]$ being the Kullback-Leibler divergence between $Q^\theta$ and $Q^{1-\theta}$ is strictly positive. This is without loss, as $P$ clearly dominates $Q$ (in large samples) in case $Q$ is trivial.}
\]

The Fenchel conjugate is an order-reversing operation: From \eqref{eq:fenchel} we see that if $K_X \geq K_Y$ pointwise, then the corresponding conjugates satisfy $K^*_Y \geq K^*_X$ pointwise. The relation between $K_{X^\theta}$ and $K_{Y^\theta}$ established in \eqref{eq:KX-KYa} and \eqref{eq:KX-KYb} is more complicated, and implies the following ranking of their conjugates:
\begin{align*}
    K_{Y^\theta}^*(a) > K_{X^\theta}^*(a)~~~&\text{ for }~~ \mathbb{E}[X^\theta] \leq a \leq \max[Y^\theta] \\
    K_{Y^\theta}^*(a) < K_{X^\theta}^*(a) ~~~&\text{ for }~~  0 \leq a \leq \mathbb{E}[Y^\theta]. 
\end{align*}
This is the content of the next lemma, which in addition shows that the differences between the Fenchel conjugates admit a uniform bound.%, and over slightly larger intervals. 

%The Fenchel transform is an order-reversing transformation, which implies that the conjugate functions $K_{X^\theta}^*$ and $K_{Y^\theta}^*$. Moreover, in the range that will be useful to us, we can uniformly bound the difference.
\begin{lemma}
\label{lem:eta}
Suppose $P$ and $Q$ are a generic pair of bounded experiments such that $P$ dominates $Q$ in the R\'enyi order. Let $(X^\theta)$ and $(Y^\theta)$ be the corresponding log-likelihood ratios. Then there exists $\eta \in (0,1)$ such that in both states $\theta \in \{0,1\}$
\begin{align*}
    K_{Y^\theta}^*(a) - \eta > K_{X^\theta}^*(a + \eta)~~~&\text{ for }~~ \mathbb{E}[X^\theta]-\eta \leq a \leq \max[Y^\theta] \\
   K_{Y^\theta}^*(a-\eta) < K_{X^\theta}^*(a)-\eta ~~~&\text{ for }~~  0 \leq a \leq \mathbb{E}[Y^\theta]+\eta. 
\end{align*}
\end{lemma}
These estimates will allow us to apply the previous Proposition~\ref{prop:large-deviations-compare} and make uniform comparisons of large deviation probabilities. In the range $a \in (\mathbb{E}[Y^\theta]+\eta, \mathbb{E}[X^\theta]-\eta)$ that is not covered by Lemma~\ref{lem:eta}, large deviation techniques are not necessary and it will be sufficient to apply more elementary estimates. %Details are found in the next section.

\subsection{R\'enyi Order Implies Large Sample Order} \label{sec:sufficiency}

%Throughout the proofs, we use the notation introduced in \S\ref{sec:LLR} and \S\ref{sec:reduction} and further discussed in \S\ref{sec:preliminaries}, as well as the notation related to large deviation estimates introduced in \S\ref{sec:large-deviations}.
We now complete the proof of Theorem~\ref{thm:eventual} and show that if two experiments are ranked in the R\'enyi order then they are also ranked in the large sample order.
%We now show that (ii) implies (i). %we prove the stronger, quantitative statement of Theorem~\ref{thm:quant}.
By Theorem~\ref{prop:reduction} we need to show that there exists a sample size $n_0$ such that for all $n \geq n_0$, the perfected log-likelihood ratios of $n$ independent draws from $P$ and $Q$ are ordered in terms of first-order stochastic dominance. 

More concretely, consider the log-likelihood ratios $X^\theta$ and $Y^\theta$ (for a single sample) as defined above, with distributions $F_\theta$ and $G_\theta$ conditional on state $\theta$. Let $F_\theta^{*n}$ be the $n$-th convolution power of $F_\theta$, which represents the distribution of log-likelihood ratios under the product experiment $P^{\otimes n}$; similarly define $G_\theta^{*n}$. By Lemma~\ref{lem:signs}, it suffices to show that for $n \geq n_0$ it holds that
\begin{equation} \label{eq:F1G1}
    F_1^{*n}(na) \leq G_1^{*n}(na)~~~\text{ for all } a \geq 0
\end{equation}
and
\begin{equation} \label{eq:F0G0}
    F_0^{*n}(na) \leq G_0^{*n}(na)~~~\text{ for all } a \geq 0.
\end{equation}

Below we show \eqref{eq:F1G1}; the argument for \eqref{eq:F0G0} is identical after relabelling the states. Assume that $X^1$ and $Y^1$ take values in $[-b,b]$. We will set $n_0 = 8b^2 \eta^{-3}$, where $\eta \in (0,1)$ is as given in Lemma~\ref{lem:eta}. For future use, we note that $\mathbb{E}[X^1] - \eta > \mathbb{E}[Y^1]$.\footnote{Otherwise, the first part of Lemma~\ref{lem:eta} would apply to $a = \mathbb{E}[Y^1]$, leading to $0-\eta > K^*_{X^1}(a+\theta)$. This is impossible as $K^*$ is non-negative.}

Let $X^1_1,\ldots,X^1_n$ be i.i.d.\ copies of $X^1$ and $Y^1_1,\ldots,Y^1_n$ be i.i.d.\ copies of $Y^1$. We can restate \eqref{eq:F1G1} as
\begin{equation} \label{eq:F1G1 restated}
\Pr{X^1_1 + \dots + X^1_n \leq na} \leq \Pr{Y^1_1 + \dots + Y^1_n \leq na}, ~~~\text{ for all } a \geq 0.
\end{equation}
To prove this, we divide into four ranges of values of $a$:

\paragraph{Case 1: $a \geq \max[Y^1]$.} In this case the right-hand side of \eqref{eq:F1G1 restated} is $1$, and hence the result follows trivially.

\paragraph{Case 2: $\mathbb{E}[X^1]-\eta \leq a < \max[Y^1]$.} %(with $\eta$ given in Lemma~\ref{lem:eta}).
From Lemma~\ref{lem:eta} we have that
\[
  K_{Y^1}^*(a) -\eta > K_{X^1}^*(a + \eta).
\]
As $a \geq \mathbb{E}[X^1]-\eta > \mathbb{E}[Y^1]$, we can directly apply Proposition~\ref{prop:large-deviations-compare} and conclude that \eqref{eq:F1G1 restated} holds for all $n \geq 4b^2(1+\eta)\eta^{-3}$. Since $\eta < 1$, it holds for all $n \geq n_0 = 8b^2\eta^{-3}$. 

\paragraph{Case 3: $\mathbb{E}[Y^1]+\eta \leq a < \mathbb{E}[X^1] - \eta$.}
By the Chebyshev inequality,
\[
    \Pr{X^1_1 + \cdots + X^1_n \leq na} \leq \Pr{X^1_1 + \cdots + X^1_n \leq n(\mathbb{E}[X^1]-\eta) } \leq \frac{\Var(X^1_1+\cdots+X^1_n)}{n^2\eta^2}.
\]
Since $\Var(X^1_1+\cdots+X^1_n)  =n\Var(X^1) \leq nb^2$, we have that
\[
    \Pr{X^1_1 + \cdots + X^1_n \leq na}  \leq \frac{b^2}{n\eta^2}.
\]
By a similar argument,
\[
    \Pr{Y^1_1 + \cdots + Y^1_n \leq na} \geq 1-\frac{b^2}{n\eta^2}.
\]
Hence for all $n \geq 2b^2\eta^{-2}$ we have
\[
    \Pr{X^1_1 + \cdots + X^1_n \leq na}  \leq \Pr{Y^1_1 + \cdots + Y^1_n \leq na}.
\]
As $n_0 = 8b^2\eta^{-3}$ is bigger, \eqref{eq:F1G1 restated} holds for $n \geq n_0$.

\paragraph{Case 4: $0 \leq a < \mathbb{E}[Y^1]+\eta$.}
By Lemma~\ref{lem:eta} we have that
\[
  K_{X^1}^*(a) -\eta > K_{Y^1}^*(a-\eta).
\]
For any random variable $Z$, we have $K_{-Z}(t) = \log \E{\ee^{t (-Z)}} = \log \E{\ee^{(-t) Z}} = K_{Z}(-t)$, and $K_{-Z}^*(a) = \sup_{t \in \R} t \cdot a -K_{-Z}(t) = \sup_{t \in \R} (-t) \cdot (-a) -K_{Z}(-t) = K_{Z}^*(-a)$. Therefore
\[
  K_{-X^1}^*(-a) -\eta > K_{-Y^1}^*(-a+\eta).
\]
We can now apply Proposition~\ref{prop:large-deviations-compare} to the random variables $-Y^1$ and $-X^1$, and the threshold $-a > -\mathbb{E}[Y^1]-\eta > \mathbb{E}[-X^1]$. This yields
\[
  \Pr{-Y^1_1 - \cdots - Y^1_n > -na} \geq \Pr{-X^1_1 - \cdots - X^1_n > -na}
\]
for all $n \geq 4b^2(1+\eta)\eta^{-3}$. Hence \eqref{eq:F1G1 restated} holds for $n \geq n_0$.\footnote{The comparison $\Pr{X^1_1 + \cdots + X^1_n < na} \leq \Pr{Y^1_1 + \cdots + Y^1_n < na}$ for all $a$ in this range implies the desired result $\Pr{X^1_1 + \cdots + X^1_n \leq na} \leq \Pr{Y^1_1 + \cdots + Y^1_n \leq na}$, by a standard limit argument.}

\bigskip

This proves \eqref{eq:F1G1 restated} for all $a \geq 0$ and completes the proof of Theorem~\ref{thm:eventual}.

\subsection{Number of Samples Required}\label{sec:upper bound}

The proof of Theorem~\ref{thm:eventual} establishes a stronger statement, and in fact provides an explicit bound on the number of repetitions sufficient to achieve large sample dominance. %L: I added "an explicit bound on the number of.." the original sentence read a little akward

\begin{theorem}
  \label{thm:quant}
  Let $P$ and $Q$ be a generic pair of bounded experiments, with log-likelihood ratios taking values in $[-b,b]$. Assume $P$ dominates $Q$ in the R\'enyi order, and let $\eta \in (0,1)$ be provided by Lemma~\ref{lem:eta}. Then $P^{\otimes n}$ Blackwell dominates $Q^{\otimes n}$ for all $n \geq n_0 = 8b^2 \eta^{-3}$.
\end{theorem}

The constant $n_0$ is decreasing in the parameter $\eta$. This fact follows from a logic analogous to the one behind Proposition~\ref{prop:large-deviations-compare}: Larger values of $\eta$ imply that the probability of unlikely, but very informative, signal realizations decreases at a much slower rate under the experiment $P^{\otimes n}$ than under $Q^{\otimes n}$, as the sample size $n$ becomes large. 

While simple, the constant $n_0$ is far from being tight. For example, our proof of Proposition \ref{prop:large-deviations-compare} uses the Chebyshev inequality, which may be improved by a suitable application of the Berry-Esseen Theorem, at the cost of a more complex bound. It remains an open problem to develop more precise estimates.

\begin{comment}
We can express $\eta$ directly in terms of the R\'enyi divergences of the two experiments. 

\begin{lemma}
Let $P$ and $Q$ be a generic pair of bounded experiments, with log-likelihood ratios taking values in $[-b,b]$ where $b \geq 1$. Define the following positive constants:
\begin{align*}
\xi &= \frac{1}{2} \min_{\theta \in \{0,1\}} R_Q^{\theta}(1); \\
\delta &= \min_{\theta \in \{0,1\}, ~t \geq \frac{\xi}{b^2}} R_P^{\theta}(t) - R_Q^{\theta}(t).
\end{align*}
Then Lemma \ref{lem:eta} holds whenever 
\[
\eta \leq \xi ~~~\text{and} ~~~\eta \leq \frac{\delta^2}{12b^2}. 
\]
\end{lemma}
\end{comment}

\section{Discussion and Related Literature}
\label{sec:blackwell_lit}

\paragraph{Comparison of Experiments.} Blackwell \citeyearpar[p.\ 101]{blackwell1951comparison} posed the question of whether dominance of two experiments is equivalent to dominance of their $n$-fold repetitions. \cite{stein1951notes} and \cite{torgersen1970comparison} provide early examples of two experiments that are not comparable in the Blackwell order, but are comparable in large samples.

\cite{moscarini2002eventual} propose an alternative criterion for comparing repeated experiments. According to their notion, an experiment $P$ dominates an experiment $Q$ if for every decision problem with finitely many actions, there exists some $n_0$ such that the expected payoff achievable from observing $P^{\otimes n}$ is higher than that from observing $Q^{\otimes n}$ whenever $n \geq n_0$. This order is characterized by the {\em efficiency index} of an experiment, defined, in our notation, as the minimum over $t \in (0,1)$ of the function $\ee^{(t-1)R_P^0(t)}$ (where a smaller index means a better experiment). There are two conceptual differences between the order studied in \citeauthor{moscarini2002eventual} and the large sample order that we characterize:
\begin{compactenum}
    \item While in \citeauthor{moscarini2002eventual} the number $n_0$ of repetitions is allowed to depend on the decision problem, dominance in large samples is a criterion for comparing experiments uniformly over decision problems, for fixed sample sizes. Thus the large sample order is conceptually closer to Blackwell dominance.\footnote{Recent work by \cite{hellmanLehrer} generalizes the Moscarini-Smith order to Markov (rather than i.i.d.) sequences of experiments.}
    \item The order proposed in \citeauthor{moscarini2002eventual} restricts attention to decision problems with \emph{finitely} many actions, while dominance in the large sample order implies that observing $P^{\otimes n}$ is better that observing $Q^{\otimes n}$ for every decision problem.
\end{compactenum}

%\textbf{As a consequence of (i), the large sample order allows one to compare Blackwell experiments for finite sample sizes $n \geq n_0$ (with $n_0$ given in Theorem~\ref{thm:quant}), independent of the decision problem. The ranking of two experiments in the Moscarini-Smith order, on the other hand, does not imply Blackwell dominance for any finite sample size. Thus the large sample order is conceptually closer to Blackwell dominance.} \textsl{L: A minor point. Do we need this paragraph?}\textsl{I would support removing it. But should we mention somewhere here that MS is an extension of large samples which is an extension of Blackwell?}

Related to (ii), \cite{azrieli2014comment} shows that the Moscarini-Smith order is a strict extension of dominance in large samples. 
%It follows immediately from their definition that the Moscarini-Smith order is a refinement of ours. This can be seen from the corresponding characterizations, which show that the Moscarini-Smith order is generically complete, while ours is a partial order. 
Perhaps surprisingly, this conclusion is reversed under a modification of their definition: It follows from our results that when extended to consider all decision problems, including problems with infinitely many actions, the Moscarini-Smith order over experiments (generically) coincides with the large sample order.\footnote{Consider the following variant of the Moscarini-Smith order: Say that $P$ dominates $Q$ if for \emph{every} decision problem (with possibly infinitely many actions) there exists an $n_0$ such that the expected payoff achievable from $P^{\otimes n}$ is higher than that from $Q^{\otimes n}$ whenever $n \geq n_0$. Each R\'enyi divergence $R^\theta_P(t)$ corresponds to the expected payoff in some decision problem (see \S\ref{sec:Renyi decision problems}), and for such decision problems the ranking over repeated experiments is independent of the sample size $n$. Thus $P$ dominates $Q$ in this order only if $P$ dominates $Q$ in the R\'enyi order. By Theorem~\ref{thm:eventual}, $P$ must then dominate $Q$ in large samples.}
%L: I added some details. It is now slightly longer but I think more clear.

Our notion of dominance in large samples is prior-free. In contrast, several authors \citep*{kelly1956entropy, lindley1956measure, cabrales2013entropy} have studied a complete ordering of experiments, indexed by the expected reduction of entropy from prior to posterior beliefs (i.e., mutual information between states and signals). We note that unlike Blackwell dominance, dominance in large samples does not guarantee a higher reduction of uncertainty given any prior belief.\footnote{To see this, consider Example 2 above with parameters $\alpha = 0.1$ and $\beta = 0$. Then Proposition~\ref{prop:AzrieliConjecture} ensures that the experiment $P$ dominates $Q$ in large samples. However, given a uniform prior, the residual uncertainty under $P$ is calculated as the expected entropy of posterior beliefs, which is $\frac{1}{2}\log(2) \approx 0.346$. The residual uncertainty under $Q$ is $-\alpha \log \alpha - (1-\alpha)\log(1-\alpha) \approx 0.325$, which is lower.}

\paragraph{Majorization and Quantum Information.}

Our work is related to the study of {\em majorization} in the quantum information literature. Majorization is a stochastic order commonly defined for distributions on countable sets. For distributions with a given support size, this order is closely related to the Blackwell order. Let $P = (\Omega,P_0,P_1)$ and $Q = (\Xi,Q_0,Q_1)$ be two experiments such that $\Omega$ and $\Xi$ are finite and of the same size, and $P_0$ and $Q_0$ are the uniform distributions on $\Omega$ and $\Xi$. Then $P$ Blackwell dominates $Q$ if and only if $P_1$ majorizes $Q_1$ \citep[see][p.\ 264]{torgersen1985majorization}. This no longer holds when $\Omega$ and $\Xi$ are of different sizes.

Motivated by questions in quantum information, \cite{jensen2019asymptotic} asks the following question: Given two finitely supported distributions $\mu$ and $\nu$, when does the $n$-fold product $\mu^{\times n} = \mu \times \cdots \times \mu$ majorize $\nu^{\times n}$ for all large $n$? He shows that for the case that $\mu$ and $\nu$ have {\em different} support sizes, the answer is given by the ranking of their R\'enyi entropies.\footnote{As discussed above, majorization with different support sizes does not imply Blackwell dominance. Indeed, the ranking based on R\'enyi entropies is distinct from our ranking based on R\'enyi divergences unless the support sizes are equal. See \S\ref{sec:jensen} in the appendix for details.} For the case of equal support size, Theorem~\ref{thm:eventual} implies a similar result, which \citet[Remark 3.9]{jensen2019asymptotic} conjectures to be true. We prove his conjecture in \S\ref{sec:jensen} in the appendix.

\cite{fritz2018generalization} uses an abstract algebraic approach to prove a result that is complementary to Proposition~\ref{prop:large-deviations-compare}. While Fritz's theorem does not require our genericity condition, the comparison of distributions is stated in terms of a notion of approximate stochastic dominance. A result similar to Proposition~\ref{prop:large-deviations-compare} (but without the $\eta$ and the quantitative bound on $n$) appears as Lemma 2 in \cite{aubrun2008catalytic}, also in the context of majorization and quantum information theory.

Both \cite{fritz2018generalization} and \cite{jensen2019asymptotic}, in their respective settings, ask a question in the spirit of our dominance ratio, and prove results that are similar to Proposition~\ref{prop:index}.
 
%\paragraph{Stochastic Orders.} \cite{muller2002comparison} and \cite{shaked2007stochastic} are comprehensive sources on stochastic orders. The ordering generated by the functionals of the form $L_X(t)$ for $t > 0$, is known in the literature as the \textit{Laplace Transform Order}, and studied in \cite{reuter1981maximal}, \cite{fishburn1980continua}, \cite{alzaid1991laplace} and \cite{caballe1996mixed}, among others. 

%\cite{hart2009dominance} proposes two complete stochastic orders that refine second-order stochastic dominance: \emph{wealth-uniform dominance} and \emph{utility-uniform dominance}. He further shows that dominance in these orders is characterized by having a smaller riskiness index/measure given in \cite{aumannSerrano} and \cite{fosterHart}, respectively. But since these measures of risk are distinct, an open question left by \cite{hart2009dominance} is whether the two stochastic orders agree on interesting cases beyond second-order stochastic dominance. In \S\ref{sec:hart}, we show that the two uniform dominance orders both refine second-order dominance in large numbers.

\paragraph{Experiments for Many States and Unbounded Experiments.}
Our analysis leaves open a number of questions. The most salient is the extension of Theorem~\ref{thm:eventual}, our characterization of dominance in large samples, to experiments with more than two states. In \S\ref{sec:many states} in the appendix, we identify a set of \emph{necessary conditions} for large sample dominance. These conditions are expressed in terms of the moment generating function of the log-likelihood ratios---which generalizes the ranking of R\'enyi divergences in the two state case. While we conjecture this set of conditions to be also sufficient, our proof technique for sufficiency does not straightforwardly extend to more than two states. In particular, we do not know how to extend the reduction of Blackwell dominance to first-order stochastic dominance (Theorem~\ref{prop:reduction}).\footnote{If such a reduction could be obtained, the remaining obstacle would be the characterization of first-order stochastic dominance between large i.i.d.\ sums of random \emph{vectors}. This would require the development of large deviation estimates in higher dimensions (generalizing Lemma~\ref{lem:large-deviation-lower-bound} in the appendix).} With binary states we have been able to derive this simplification because one-dimensional convex (indirect utility) functions admit an one-parameter family of extremal rays. Going to higher dimensions, the difficulty is that ``the extremal rays are too complex to be of service'' \citep{jewitt2007orders}.

Another extension for future work is to experiments with unbounded likelihood ratios. As we demonstrate in \S\ref{sec:unbounded} in the appendix, our characterization of the large sample order remains valid if the dominant experiment $P$ is unbounded whereas the dominated experiment $Q$ is bounded. The result also extends, under an additional assumption, to pairs of unbounded experiments whose R\'enyi divergences are finite. However, we do not know whether and how our result would generalize to the case of infinite R\'enyi divergences. The technical challenge is that large deviation estimates that are uniform across different thresholds typically require the moment generating function to be finite (so-called ``Cram\'er's condition'').\footnote{Although Cram\'er's result that $\log \Pr{X_1+\cdots+X_n > na} \sim -n\cdot K_X^*(a)$ remains true even when $K_X(t)$ can be infinite, as far as we know the proofs of this generalization do not deliver a quantitative lower bound similar to our Lemma~\ref{lem:large-deviation-lower-bound}. As a consequence, Cram\'er's approximation is not uniform across $a$.}

\newpage
\begin{center}
    {\bf \Large Appendix}
\end{center}
\medskip

\appendix
\addcontentsline{toc}{section}{Appendix}
\addtocontents{toc}{\setcounter{tocdepth}{-1}}

The structure of the appendix follows that of the paper. After reviewing large deviations theory, we complete the proof of Theorem~\ref{thm:eventual} by supplying the proofs of Proposition~\ref{prop:large-deviations-compare}, Lemma~\ref{lem:signs} and Lemma~\ref{lem:eta}. We then provide proofs for our other results in the order in which they appeared. 

\section{Large Deviations}\label{sec:large-deviations}

For every bounded random variable $X$ that is not a constant, we denote by $M_X(t) = \log \mathbb{E}[\ee^{tX}]$ and $K_X(t) = \log M_X(t)$ the moment and cumulant generating functions of $X$.

As is well known, $M_X$ and $K_X$ are strictly convex. We denote by
\[
    K_{X}^*(a) = \sup_{t \in \R} t\cdot a - K_{X}(t)
\]
the Fenchel conjugate of $K_X$. For $a \in (\min[X],\max[X])$ the maximization problem has a unique solution, achieved at some $t \in \R$.  This solution $t$ is non-negative if and only if $a \geq \mathbb{E}[X]$. In addition, as $K_X(0) = 0$, $K_X^*(a) \geq 0 \cdot a - K_X(0) = 0$ is non-negative. The function $K_X^*(a)$ is continuous (in fact, analytic) wherever it is finite. 

The well known Chernoff bound states that if $X,X_1,\ldots,X_n$ are an i.i.d.\ sequence, then
\begin{align*}
    \Pr{X_1+\cdots+X_n > na} \leq \ee^{-n\cdot K_X^*(a)} ~~~\text{for all } a \geq \mathbb{E}[X].
\end{align*}
The next proposition gives a lower bound for this probability.
\begin{lemma}\label{lem:large-deviation-lower-bound}
Let $X,X_1,\ldots,X_n$ be an i.i.d.\ sequence taking values in $[-b,b]$. For all $\eta > 0$, $a \in [\min[X], \max[X]-\eta)$ and $n \geq 1$, it holds that
\begin{align*}
    \Pr{X_1+\cdots+X_n > na} \geq \ee^{-n\cdot K_X^*(a+\eta)}\left(1-\frac{4b^2}{n\eta^2}\right)
\end{align*}
\end{lemma}

\begin{proof}
We first consider the case where $a \geq \mathbb{E}[X] - \eta/2$.
Define $t$ by 
\[
K_X'(t) = a+\eta/2,
\]
so that $K^*_X(a+\eta/2) = (a+\eta/2) \cdot t - K_X(t)$. Such a $t$ is a non-negative finite number, since $\mathbb{E}[X] \leq a+\eta/2 < \max[X]$. 

Denote by $\nu$ the distribution of $X$, and let $\hat{X}$ be a real random variable whose distribution $\hat{\nu}$ is given by
\begin{align*}
    \frac{\dd \hat \nu}{\dd \nu}(x) = \frac{\ee^{tx}}{\mathbb{E}[\ee^{tX}]} =
    \ee^{tx - K_X(t)}.
\end{align*}
This construction ensures that $\hat{\nu}$ is also a probability measure, so that $\hat{X}$ is a well-defined random variable. 

Note that
\begin{align*}
    \mathbb{E}[\hat{X}] = \frac{\mathbb{E}[X\ee^{tX}]}{\mathbb{E}[\ee^{tX}]} = K_X'(t) = a+\eta/2,
\end{align*}
and that the cumulant generating function of $\hat{X}$ is
\[
    K_{\hat X}(s) = \log\mathbb{E}[\ee^{s \hat X}] = \log\mathbb{E}[\ee^{tX-K_X(t)}\ee^{s X}] = K_X(s+t)-K_X(t).
\]
  
Now let $\hat{X}_1,\ldots,\hat{X}_n$ be i.i.d.\ copies of $\hat{X}$. Denote $S_n = X_1+\cdots+X_n$ and $\hat{S}_n = \hat{X}_1+\cdots+\hat{X}_n$. The cumulant generating function of $\hat{S}_n$ is
\[
    K_{\hat{S}_n}(s)=nK_{\hat{X}}(s) = n(K_X(s+t)-K_X(t)) = K_{S_n}(s+t)-K_{S_n}(t),
\]
and so the Radon-Nikodym derivative between the distributions of $\hat{S}_n$ and $S_n$ is $\ee^{tx - K_{S_n}(t)}=\ee^{tx-nK_X(t)}$. Hence
\begin{align*}
    \Pr{S_n > na}
    &= \mathbb{E}[\ind{S_n > na}]\\
    &= \mathbb{E}\left[\ee^{-t\hat S_n + nK_X(t)}\ind{\hat S_n > na}\right]\\
    &= \ee^{nK_X(t)} \cdot \mathbb{E}\left[\ee^{-t \hat S_n}\ind{\hat S_n > na}\right].
\end{align*}
The event $\{\hat S_n > na\}$ contains the event $\{n(a+\eta) > \hat S_n > na\}$, and so
\begin{align*}
    \Pr{S_n > na}
    &\geq \ee^{nK_X(t)} \cdot \mathbb{E}\left[\ee^{-t \hat S_n}\ind{n(a+\eta) > \hat S_n > na}\right]\\
    &\geq \ee^{nK_X(t)-tn(a+\eta)} \cdot \mathbb{E}\left[\ind{n (a+\eta) > \hat S_n > na}\right]\\
    &= \ee^{nK_X(t)-tn(a+\eta)} \cdot \Pr{n( a+\eta) > \hat S_n > na}
\end{align*}
where the second inequality uses $t \geq 0$ and $\hat S_n < n(a+\eta)$ whenever $\ind{n(a+\eta) > \hat S_n > na} > 0$. 

Now, $\hat S_n$ has expectation $n\mathbb{E}[\hat{X}] = n(a+\eta/2)$. Its variance is $n \Var[\hat{X}] \leq n\mathbb{E}[\hat{X}^2] \leq nb^2$, since $\hat{X}$ has the same support of $X$ by construction. Therefore, by the Chebyshev inequality,
\begin{align*}
    \Pr{n(a+\eta) > \hat{S}_n > na} = 1 - \Pr{\vert \hat{S}_n - \mathbb{E}[\hat{S}_n] \vert \geq n\eta/2}\geq 1-\frac{nb^2}{(n\eta/2)^2} = 1 - \frac{4b^2}{n\eta^2}.
\end{align*}
We have thus shown that 
\begin{align*}
    \Pr{S_n > na}
    \geq \ee^{-n(t(a+\eta)-K_X(t))}\left(1-\frac{4b^2}{n\eta^2}\right).
\end{align*}
Now, by definition $K^*_X(a+\eta) \geq t(a+\eta)-K_X(t)$. Hence we arrive at
  \begin{align*}
    \Pr{S_n > na}
    \geq \ee^{-n \cdot K_X^*(a+\eta)}\left(1-\frac{4b^2}{n\eta^2}\right).
  \end{align*}
  
We turn to the case where $a < \mathbb{E}[X]-\eta/2$. In this case, we can directly apply the Chebyshev inequality and obtain
\[
    \Pr{S_n \leq na} \leq \Pr{S_n - \mathbb{E}[S_n] \leq -n\eta/2} \leq \frac{\Var[S_n]}{(n\eta/2)^2} = \frac{n \Var[X]}{(n\eta/2)^2} \leq \frac{4b^2}{n\eta^2}.
\]
Hence
\begin{align*}
    \Pr{S_n > na} \geq 1-\frac{4b^2}{n\eta^2}.
\end{align*}
Since $K_X^*$ is non-negative, we again have
\begin{align*}
    \Pr{S_n > na} \geq \ee^{-n \cdot K_X^*(a+\eta)}\left(1-\frac{4b^2}{n\eta^2}\right). 
\end{align*}
This proves the lemma.
\end{proof}

\subsection{Proof of Proposition~\ref{prop:large-deviations-compare}}

%\begin{proposition}\label{prop:large-deviations-compare}
%Let $X$ and $Y$ be random variables taking values in $[-b,b]$ and let $\eta,a$ be such that $K_Y^*(a)-\eta>K_X^*(a+\eta)$. Then for all 
%$$
%  n > \frac{4(1+\eta)b^2}{\eta^3},
%$$ 
%it holds that
%\begin{equation}\label{eq:large-deviation-compare}
%    \Pr{X_1 + \cdots + X_n > na} \geq \Pr{Y_1 + \cdots + Y_n > na}.
%\end{equation}
%\end{proposition}
%\begin{proof}
%For $a > \max[Y]$ the statement follows trivially, %since the right hand side of %\eqref{eq:large-deviation-compare} vanishes. 

If $a<\min[X]$ then the statement holds since in \eqref{eq:large-deviation-compare} the LHS is equal to 1. Below we assume $a \geq \min[X]$. By assumption, $K_X^*(a+\eta)$ is finite, and hence $a + \eta < \max[X]$. We can thus apply Lemma~\ref{lem:large-deviation-lower-bound} to $X$ and conclude that for every $n \geq 1$,
\begin{align*}
    \Pr{X_1+\cdots+X_n > na}
    \geq \ee^{-n \cdot K_X^*(a+\eta)}\left(1-\frac{4b^2}{n\eta^2}\right).
\end{align*}
By assumption we have that $K_Y^*(a) -\eta\geq K_X^*(a+\eta)$, and so
\begin{align*}
    \Pr{X_1+\cdots+X_n > na}
    &\geq \ee^{-n \cdot K_Y^*(a)}\ee^{n\eta}\left(1-\frac{4b^2}{n\eta^2}\right) \\
    &\geq \ee^{-n \cdot K_Y^*(a)}(1+\eta)\left(1-\frac{4b^2}{n\eta^2}\right) 
\end{align*}
Hence, for $n \geq 4b^2(1+\eta)\eta^{-3}$, 
\begin{align*}
    \Pr{X_1+\cdots+X_n > na} \geq \ee^{-n \cdot K_Y^*(a)}.
\end{align*}
On the other hand, since $a \geq \mathbb{E}[Y]$ by assumption, we have the Chernoff bound
\begin{align*}
    \Pr{Y_1+\cdots+Y_n > na} \leq \ee^{-n \cdot K_Y^*(a)}.
\end{align*}
This proves the desired result \eqref{eq:large-deviation-compare}.

\section{Proof of Lemma~\ref{lem:signs}}

An exponential distribution has probability density function that vanishes for negative $u$ and equals $\ee^{-u}$ for positive $u$. Thus $\Tilde{F}_1$ and $\Tilde{G}_1$ can be written as
\[
    \Tilde{F}_1(a) = \int_{0}^\infty F_1(a + u)\ee^{-u} \,\dd u
\]
and likewise
\[
    \tilde{G}_1(a) = \int_0^{\infty} G_1(a + u)\ee^{-u} \,\dd u.
\]

Consider the first part of the lemma. Suppose $a \geq $0, then by assumption $F_1(a+u) \leq G_1(a+u)$ for all $u \geq 0$, which implies $\Tilde{F}_1(a) \leq \tilde{G}_1(a)$. 

For the second part of the lemma, we will establish the following identities:
\begin{equation}\label{eq:TildeF_alternative}
    \Tilde{F}_1(a) = \int_{-a}^\infty F_0(v)\ee^{-v} \,\dd v \text{\quad and \quad}\tilde{G}_1(a) = \int_{-a}^\infty G_0(v)\ee^{-v} \,\dd v.
\end{equation}
Given this, the result would follow easily: If $F_0(v) \leq G_0(v)$ for all $v \geq 0$, then the above implies $\Tilde{F}_1(a) \leq \tilde{G}_1(a)$ for all $a \leq 0$.

To show \eqref{eq:TildeF_alternative}, we recall \eqref{eq:def-TildeF} and write 
\begin{equation}\label{eq:zzz}
\Tilde{F}_1(a) = \int_{-\infty}^{a} \,\dd F_1(u) + \ee^a \int_{a}^{\infty} \ee^{-u} \,\dd F_1(u).
\end{equation}
The key observation is that $\dd F_1(u) = -\ee^u \,\dd F_0(-u)$. Indeed, $\dd F_1(u)$ is the density under state $1$ that the log-likelihood ratio $\log(\dd P_1/\dd P_0)$ is equal to $u$, which is also the density under state $1$ that the opposite log-likelihood ratio $\log(\dd P_0/\dd P_1)$ is equal to $-u$. By definition of the log-likelihood ratio, this density is scaled by a factor of $e^{-u}$ when we change measure from state $1$ to state $0$. 

Substituting $\dd F_1(u) = -\ee^u \,\dd F_0(-u)$ into \eqref{eq:zzz}, we have \[
\Tilde{F}_1(a) = \int_{-\infty}^{a} -\ee^u \,\dd F_0(-u) + \ee^a \int_{a}^{\infty} -\,\dd F_0(-u) = \int_{-a}^{\infty} \ee^{-v} \,\dd F_0(v) + \ee^a F_0(-a),
\]
where the second equality uses change of variable from $u$ to $v = -u$. Integration by parts then yields \eqref{eq:TildeF_alternative} and completes the proof.

\section{Proof of Lemma~\ref{lem:eta}}

Fix $\theta$, we will show the result holds for all sufficiently small positive $\eta$. Because $P$ dominates $Q$ in the R\'enyi order, and the pair of experiments is generic, the two log-likelihood ratios satisfy $0 < \mathbb{E}[Y^\theta] < \mathbb{E}[X^\theta]$ and $\max[Y^\theta] < \max[X^\theta]$. 

For the first part of the lemma, consider the interval $A = [\mathbb{E}[X^\theta],\max[Y^\theta]]$. If it is empty (i.e., $\mathbb{E}[X^\theta] > \max[Y^\theta]$), the result trivially holds by choosing $\eta$ small. Otherwise, consider any point $a \in A$. Since $a$ is above the expectation of $X^\theta$,
\[
    K^*_{X^\theta}(a) = \sup_{t \geq 0} ta - K_{X^\theta}(t).
\]
And because $a < \max[X]$ the supremum is achieved at some finite $\hat{t} \geq 0$. Dominance in the R\'enyi order implies, by \eqref{eq:KX-KYa},
\[
    K^*_{X^\theta}(a) = \hat{t}a - K_{X^\theta}(\hat{t}) \leq \hat{t}a - K_{Y^\theta}(\hat{t}) \leq K^*_{Y^\theta}(a).
\]
The first inequality can only hold equal if $\hat{t} = 0$ and $a = \mathbb{E}[X^\theta]$, but in that case the second inequality is strict because $a$ is strictly above the expectation of $Y^\theta$. Hence $K_{Y^\theta}^*(a) > K_{X^\theta}^*(a)$ for all $a$ in $A$. Since $A$ is compact and the two Fenchel transforms are continuous, we can find $\eps_1$ positive such that $K_{Y^\theta}^*(a) - \eps_1 > K_{X^\theta}^*(a)$ over all $a \in A$. Choosing positive $\eps_2$ sufficiently small, we in fact have $K_{Y^\theta}^*(a) - \eps_1 > K_{X^\theta}^*(a)$ for all $a$ in the slightly bigger interval $[\mathbb{E}[X^\theta]-\eps_2,\max[Y^\theta]]$. By uniform continuity, any small positive $\eta$ satisfies $K_{X^\theta}^*(a + \eta) - K_{X^\theta}^*(a) < \frac{\eps_1}{2}$ for all $a$ in this interval. If in addition $\eta < \min\{\frac{\eps_1}{2}, \eps_2\}$, then 
\[
K_{Y^\theta}^*(a) - \eta > K_{Y^\theta}^*(a) - \eps_1 + \frac{\eps_1}{2} > K_{X^\theta}^*(a) + \frac{\eps_1}{2} > K_{X^\theta}^*(a + \eta)
\]
for all $a \in [\mathbb{E}[X^\theta]-\eps_2,\max[Y^\theta]]$, and thus for $a \in [\mathbb{E}[X^\theta]-\eta,\max[Y^\theta]]$. This yields the desired result. 

As for the second half, consider a point $a \in [0,\mathbb{E}[Y^\theta]]$. Since $a \leq \mathbb{E}[Y^\theta]$ and $a \geq 0 > \min[Y^\theta]$,\footnote{The latter holds because $\max[Y^{1-\theta}] \geq \mathbb{E}[Y^{1-\theta}] > 0$, and by definition $\min[Y^\theta] = -\max[Y^{1-\theta}]$.} there exists a finite $\tilde{t} \leq 0$ such that $K^*_{Y^\theta}(a) = \tilde{t}a - K_{Y^\theta}(\tilde{t})$. This $\tilde{t}$ satisfies $K'_{Y^\theta}(\tilde{t}) = a$. 

We now show that $\tilde{t} > -1$. The cumulant generating functions of $Y^\theta$ and $Y^{1-\theta}$ satisfy for all $t \in \R$ the relation
\[
  K_{Y^\theta}(t) = K_{Y^{1 - \theta}}(-t - 1)
\]
and hence $K'_{Y^\theta}(-1) = -K'_{Y^{1 - \theta}}(0) = - \mathbb{E}[Y^{1-\theta}] < 0$. Since $K'_{Y^\theta}(\tilde{t}) = a \geq 0$, and $K'_{Y^\theta}$ is increasing, we have $\tilde{t} \in (-1,0]$. Dominance in the R\'enyi order then implies, by \eqref{eq:KX-KYb},
\[
    K^*_{Y^\theta}(a) = \tilde{t}a - K_{Y^\theta}(\tilde{t}) \leq \tilde{t}a - K_{X^\theta}(\tilde{t}) \leq K^*_{X^\theta}(a).
\]
Similar to before, the first inequality can only hold equal if $\tilde{t} = 0$ and $a = \mathbb{E}[Y^\theta]$, but in that case the second inequality is strict because $a$ is strictly below the expectation of $X^\theta$. Hence $K^*_{Y^\theta}(a) < K^*_{X^\theta}(a)$ for all $a \in [0,\mathbb{E}[Y^\theta]]$. Using continuity as before, any sufficiently small $\eta$ makes $K_{Y^\theta}^*(a-\eta) < K^*_{X^\theta}(a) - \eta$ hold for all $a$ in the slightly bigger interval $[0, \mathbb{E}[Y^\theta] + \eta]$. Hence the lemma holds.

\section{Proof of Proposition~\ref{prop:example}}

Let $p_1$ (resp.\ $p_3$) be the essential minimum (resp.\ maximum) of the distribution $\pi$ of posterior beliefs induced by $P$. Since the support of $\pi$ has at least $3$ points, we can find $p_2 \in (p_1, p_3)$ such that $\pi([p_1,p_2]) > \pi(\{p_1\})$ and $\pi([p_2,p_3]) > \pi(\{p_3\})$. 

We use this $p_2$ to construct an experiment $Q$ which has signal space $\{0, 1\}$, and which is a garbling of $P$. Specifically, if a signal realization under $P$ leads to posterior belief below $p_2$, the garbled signal is 0. If the posterior belief under $P$ is above $p_2$, the garbled signal is 1. Finally, if the posterior belief is exactly $p_2$, we let the garbled signal be 0 or 1 with equal probabilities.

Since $\pi([p_1,p_2]) > \pi(\{p_1\})$, the signal realization ``0'' under experiment $Q$ induces a posterior belief that is strictly bigger than $p_1$, and smaller than $p_2$. Likewise, the signal realization ``1'' induces a belief strictly smaller than $p_3$, and bigger than $p_2$. Thus $P$ and $Q$ form a generic pair, and the distribution $\tau$ of posterior beliefs under $Q$ is a strict mean-preserving contraction of $\pi$. We now recall that the R\'enyi divergences are derived from strictly convex indirect utility functions $u(p) = -p^t(1-p)^{1-t}$ for $0 < t < 1$ and $v(p) = p^t(1-p)^{1-t}$ for $t > 1$. Thus, $R_P^\theta(t) > R_Q^\theta(t)$ for all $\theta \in \{0,1\}$ and $t > 0$. 

We will perturb $Q$ to be a slightly more informative experiment $Q'$, such that $P$ still dominates $Q'$ in the R\'enyi order but not in the Blackwell order. For this, suppose that under $Q$ the posterior belief equals $q_1 \in (p_1, p_2)$ with some probability $\lambda$, and equals $q_2 \in (p_2, p_3)$ with remaining probability. Choose any small positive number $\eps$, and let $Q'$ be another binary experiment inducing the posterior belief $q_1 - \eps(1-\lambda)$ with probability $\lambda$, and inducing the posterior belief $q_2 + \eps \lambda$ otherwise. Such an experiment exists, because the expected posterior belief is unchanged. By continuity, $R_P^\theta(t) > R_{Q'}^\theta(t)$ still holds when $\eps$ is sufficiently small.\footnote{Using the relation between $R_P^0(t)$ and $R_P^1(1-t)$, it suffices to show $R_P^\theta(t) > R_{Q'}^\theta(t)$ for $\theta \in \{0, 1\}$ and $t \geq 1/2$. Fixing a large $T$, then by uniform continuity, $R_P^\theta(t) > R_Q^\theta(t)$ implies $R_P^\theta(t) > R_{Q'}^\theta(t)$ for $t \in [1/2, T]$ when $\eps$ is small. This also holds for $t$ large, because as $t \to \infty$ the growth rate of the R\'enyi divergences are governed by the maximum of likelihood ratios, which is larger under $P$ than under $Q'$.} Since $P$ and $Q'$ also form a generic pair, Theorem~\ref{thm:eventual} shows that $P$ dominates $Q'$ in large samples. 

It remains to prove that $P$ does not dominate $Q'$ according to Blackwell. Consider a decision problem where the prior is uniform, the set of actions is $\{0,1\}$, and payoffs are given by $u(\theta = a = 0) = p_2$, $u(\theta = a = 1) = 1-p_2$ and $u(\theta \neq a) = 0$. The indirect utility function is $v(p) = \max\{(1-p)p_2,~ p(1-p_2)\}$, which is piece-wise linear on $[0, p_2]$ and $[p_2, 1]$ but convex at $p_2$. Recall that in constructing the garbling from $P$ to $Q$, those posterior beliefs under $P$ that are below $p_2$ are ``averaged'' into the single posterior belief $q_1$ under $Q$, and those above $p_2$ are averaged into the belief $q_2$. Thus $Q$ achieves the same expected utility in this decision problem as $P$ (despite being a garbling). Nevertheless, observe that $Q'$ achieves higher expected utility in this decision problem than $Q$.\footnote{Formally, since $q_1 - \eps(1-\lambda) < q_1 < p_2$ and $q_2 + \eps \lambda > q_2 > p_2$, it holds that
\[
\lambda \cdot v(q_1 - \eps(1-\lambda) ) + (1-\lambda) \cdot v(q_2 + \eps \lambda) > \lambda \cdot v(q_1) + (1-\lambda) \cdot v(q_2).
\]
} Hence $Q'$ achieves higher expected utility than $P$, implying that it is not Blackwell dominated.

\section{Proof of Proposition~\ref{prop:AzrieliConjecture}}

It is easily checked that the condition $R_P^1(1/2) > R_Q^1(1/2)$ reduces to 
\begin{equation}\label{eq:AzrieliCondition}\sqrt{\alpha(1-\alpha)} > \sqrt{\beta(\frac{1}{2}-\beta)} + \frac{1}{4}.
\end{equation}
Since the experiments form a generic pair, by Theorem~\ref{thm:eventual}, we just need to check dominance in the R\'enyi order. Equivalently, we need to show 
\begin{align}
(\frac12 - \beta)^r \beta^{1-r} + (\frac12 - \beta)^{1-r} \beta^r + \frac{1}{2} &< (1-\alpha)^r \alpha^{1-r} + (1-\alpha)^{1-r} \alpha^r, \quad \forall 0 < r < 1;\label{eq:AzrieliIneq1} \\
(\frac12 - \beta)^r \beta^{1-r} + (\frac12 - \beta)^{1-r} \beta^r + \frac{1}{2} &> (1-\alpha)^r \alpha^{1-r} + (1-\alpha)^{1-r} \alpha^r, \quad \forall r < 0 \text{ or } r > 1;\label{eq:AzrieliIneq2} \\
\beta \cdot \ln(\frac{\beta}{\frac12-\beta}) + (\frac12 - \beta) \cdot \ln(\frac{\frac12-\beta}{\beta}) &> \alpha \cdot \ln(\frac{\alpha}{1-\alpha}) + (1-\alpha) \cdot \ln(\frac{1-\alpha}{\alpha}).\label{eq:AzrieliIneq3} 
\end{align}

To prove these, it suffices to consider the $\alpha$ that makes \eqref{eq:AzrieliCondition} hold with equality.\footnote{It is clear that the inequalities are easier to satisfy when $\alpha$ increases in the range $[0, \frac{1}{2}]$.} We will show that the above inequalities hold for this particular $\alpha$, except that \eqref{eq:AzrieliIneq1} holds equal at $r = \frac12$. Let us define the following function
\[
\Delta(r) := (\frac12 - \beta)^r \beta^{1-r} + (\frac12 - \beta)^{1-r} \beta^r + \frac{1}{2} - (1-\alpha)^r \alpha^{1-r} - (1-\alpha)^{1-r} \alpha^r.
\]
When \eqref{eq:AzrieliCondition} holds with equality, we have $\Delta(0) = \Delta(\frac{1}{2}) = \Delta(1) = 0$. Thus $\Delta$ has roots at $0$, $1$ as well as a double-root at $\frac{1}{2}$. But since $\Delta$ is a weighted sum of $4$ exponential functions plus a constant, it has at most 4 roots (counting multiplicity).\footnote{This follows from Rolle's Theorem and an induction argument.} Hence these are the only roots, and we deduce that the function $\Delta$ has constant sign on each of the intervals $(-\infty, 0), (0, \frac12), (\frac12, 1), (1, \infty)$. 

Now observe that since $2\beta < \alpha \leq \frac{1}{2}$, it holds that $\frac{1/2-\beta}{\beta} > \frac{1-\alpha}{\alpha} > 1$. It is then easy to check that $\Delta(r) \to \infty$ as $r \to \infty$. Thus $\Delta(r)$ is strictly positive for $r \in (1, \infty)$. As $\Delta(1) = 0$, its derivative is weakly positive. But recall that we have enumerated the 4 roots of $\Delta$. So $\Delta$ cannot have a double-root at $r = 1$, and it follows that $\Delta'(1)$ is strictly positive. Hence \eqref{eq:AzrieliIneq3} holds. 

Note that $\Delta'(1) > 0$ and $\Delta(1) = 0$ also implies $\Delta(1-\eps) < 0$. Thus $\Delta$ is negative on $(\frac{1}{2}, 1)$. A symmetric argument shows that $\Delta$ is positive on $(-\infty, 0)$ and negative on $(0, \frac{1}{2})$. Hence \eqref{eq:AzrieliIneq1} and \eqref{eq:AzrieliIneq2} both hold, completing the proof.

\section{Proof of Proposition~\ref{prop:index}}
Denote $r = \inf_{\theta,t}\frac{R_P^\theta(t)}{R_Q^\theta(t)}$. We would like to show that $P/Q = r$. Let $n,m$ be such that $P^{\otimes n} \succeq Q^{\otimes m}$. Then, since ranking of the R\'enyi divergences is a necessary condition for Blackwell dominance, and by the additivity of R\'enyi divergences, $
   n \cdot R_P^\theta(t) \geq m \cdot R_Q^\theta(t) 
$ 
for all $\theta \in \{0,1\}$ and $t > 0$. Thus any such $m/n$ is bounded above by $r$, and so $P/Q \leq r$.

In the other direction, take any rational number $m/n < r$. Then, again by the additivity of the R\'enyi divergences, $P^{\otimes n}$ dominates $Q^{\otimes m}$ in the R\'enyi order. Furthermore, the fact that $\lim_{t \to \infty} \frac{R_P^\theta(t)}{R_Q^\theta(t)} > m/n$ implies the pair $P^{\otimes n}$ and $Q^{\otimes m}$ is generic. Therefore, by Theorem~\ref{thm:eventual}, we have that for some $k$ large enough, 
$
    P^{\otimes nk} \succeq Q^{\otimes mk}.
$
Thus $P/Q \geq mk/nk = m/n$. Since this holds for every rational $m/n$ that is less than $r$, we can conclude that $P/Q \geq r$. Finally, note that each of the functions $R_P^\theta$ and $R_Q^\theta$ are positive, increasing and bounded on $(0,\infty)$. Furthermore, using \[
\frac{R_P^\theta(t)}{R_Q^{\theta}(t)} = \frac{R_P^{1-\theta}(1-t)}{R_Q^{1-\theta}(1-t)}, 
\]
for $t \in (0,1)$, we can rewrite 
\[
P/Q = \inf_{\substack{\theta \in \{0,1\},\\t >0}} \frac{R_P^\theta(t)}{R_Q^\theta(t)} = \inf_{\substack{\theta \in \{0,1\},\\t \geq \frac{1}{2}}}\frac{R_P^\theta(t)}{R_Q^\theta(t)}.
\]
Recall that $R_P^{\theta}(t), R_Q^{\theta}(t)$ are positive, continuous in $t$ and approach $\max[X^{\theta}]$ and $\max[Y^{\theta}]$ as $t \to \infty$. Thus a compactness argument shows that $P/Q$ is always positive.

\bigskip

\bibliography{refs}

\newpage
\begin{center}{{\bf \Large Online Appendix}}\end{center}

\section{Proof of Proposition~\ref{thm:marginal}}
 
That (i) implies (ii) follows from the fact that R\'enyi divergences are monotone in the Blackwell order, and additive with respect to independent experiments.
 
To show (ii) implies (i), we introduce some notation. Given two experiments $P = (\Omega,P_0,P_1)$ and $Q = (\Xi,Q_0,Q_1)$, for each $\alpha \in [0,1]$ we denote by $\alpha P + (1-\alpha)Q = (\Psi,M_0,M_1)$ the mixed experiment where the sample space is the disjoint union $\Psi = \Omega \sqcup \Xi$ endowed with the corresponding $\sigma$-algebra, and the measures $M_0, M_1$ satisfy for every measurable $E \subseteq \Psi$
\[
    M_\theta(E) = \alpha P_\theta(E \cap \Omega) + (1-\alpha)Q_\theta(E \cap \Xi).
\]
Intuitively, the mixed experiment corresponds to a randomized experiment where $P$ is carried out with probability $\alpha$ and $Q$ with probability $1-\alpha$. The mixture operation and the product operation satisfy
$
   (\alpha P + (1-\alpha) Q) \otimes R = \alpha (P \otimes R) + (1-\alpha)(Q \otimes R).
$
 
%Let $\pi$ and $\tau$ be the distribution over posterior beliefs induced by the experiments $P$ and $Q$, respectively. It is immediate to verify that the distribution over posterior beliefs induced by the mixed experiment $\alpha P + (1-\alpha)Q$ is the convex combination $\alpha \pi + (1-\alpha)\tau$. An implication is that $P \succeq Q$ holds if and only if for every $\alpha$ and every third experiment $S$, $\alpha P + (1-\alpha)S \succeq \alpha Q + (1-\alpha)S$.
 
Now suppose $P$ dominates $Q$ in the R\'enyi order, then by Theorem~\ref{thm:eventual}, $P$ dominates $Q$ in the large sample order. The next lemma concludes the proof. 
 
\begin{lemma}
\label{lem:catalytic}
Let $P,Q$ be bounded experiments such that $P$ dominates $Q$ in the large sample order. Then there exists a bounded experiment $R$ such that $P \otimes R$ Blackwell dominates $Q \otimes R$.
\end{lemma} 
This lemma replicates a more general statement that appears in \cite{duan2005multiple,fritz2017resource}. \begin{proof}[Proof of Lemma~\ref{lem:catalytic}]
Assume $P^{\otimes n} \succeq Q^{\otimes n}$. Let
$$
  R = \frac{1}{n}\left(Q^{\otimes n}+P\otimes Q^{\otimes(n-1)}+P^{\otimes 2}\otimes Q^{\otimes(n-2)}+\cdots+P^{\otimes(n-2)}\otimes Q^{\otimes 2}+P^{\otimes(n-1)}\otimes Q\right).
$$  
Then
\begin{align*}
  P \otimes R
  &= P \otimes \frac{1}{n}\left(Q^{\otimes n}+P\otimes Q^{\otimes(n-1)}+\cdots+P^{\otimes(n-2)}\otimes Q^{\otimes 2}+P^{\otimes(n-1)}\otimes Q\right)\\
  &= \frac{1}{n}\left(P\otimes Q^{\otimes n}+P^{\otimes 2}\otimes Q^{\otimes(n-1)}+\cdots+P^{\otimes(n-1)}\otimes Q^{\otimes 2}+P^{\otimes n}\otimes Q\right)\\
  &\succeq \frac{1}{n}\left(P\otimes Q^{\otimes n}+P^{\otimes 2}\otimes Q^{\otimes(n-1)}+\cdots+P^{\otimes(n-1)}\otimes Q^{\otimes 2}+Q^{\otimes (n+1)}\right)\\
  &= Q\otimes \frac{1}{n}\left(Q^{\otimes n}+P\otimes Q^{\otimes (n-1)}+\cdots+P^{\otimes(n-1)}\otimes Q\right)\\
  &= Q \otimes R,
\end{align*}
where the middle step uses the assumption $P^{\otimes n} \succeq Q^{\otimes n}$, so that $P^{\otimes n} \otimes Q \succeq Q^{\otimes (n+1)}$. 
\end{proof} 

\section{Proof of Theorem~\ref{thm:divergences}}
 
Throughout this section, we denote by $D$ an additive divergence that satisfies the data-processing inequality and is finite on bounded experiments.

\begin{lemma}
 If a bounded experiment $P=(\Omega,P_0,P_1)$ dominates another bounded experiment $Q = (\Xi, Q_0,Q_1)$ in the Blackwell order, then $D(P_0,P_1) \geq D(Q_0,Q_1)$.
\end{lemma}

\begin{proof}
 By Blackwell's Theorem there exists a measurable function $\sigma  \colon \Omega \to \Delta(\Xi)$ such that $Q_\theta(A) = \int \sigma(\omega)(A)\,\dd P_\theta(\omega)$ for every measurable $A \subseteq \Xi$ and every $\theta$. Let $\lambda$ be the Lebesgue measure on $[0,1]$. Since $\Omega$ and $\Xi$ are Polish spaces,  there exists a measurable function $f \colon \Omega \times [0,1] \to \Xi$ such that for every $\omega \in \Omega$, $\sigma(\omega) = f(\omega,\cdot)_*(\lambda)$, where $f(\omega,\cdot)_*(\lambda)$ is the push-forward of  $\lambda$ induced by the function $f(\omega,\cdot)$ \cite[see, for example, Proposition 10.7.6 in][]{bogachev2007measure}.
Hence,
\[
    Q_\theta(A) = \int \lambda(\{t \in [0,1] : f(\omega,t) \in A\}) \,\dd P_\theta(\omega) = f_*(P_\theta \times \lambda)(A)
\]
where now $f_*(P_\theta \times \lambda)$ is the pushforward of $P_\theta \times \lambda$ induced by $f$. Being a divergence, $D$ satisfies $D(\lambda,\lambda) = 0$. Moreover, by additivity, $D(P_0 \times \lambda, P_1 \times \lambda) = D(P_0 , P_1)$. The data processing inequality then implies $D(P_0 , P_1) = D(P_0 \times \lambda, P_1 \times \lambda) \geq D(Q_0, Q_1)$.
\end{proof}
 
\begin{lemma}\label{lem:proof-divergences-1}
If the bounded experiments $P=(P_0,P_1)$ and $Q = (Q_0,Q_1)$ satisfy $R_P^\theta (t) \geq R_Q^\theta (t)$ for every $t > 0$ and $\theta \in \{0,1\}$, then $D(P_0,P_1) \geq D(Q_0,Q_1)$.
\end{lemma}

\begin{proof}
Suppose first that the strict inequality $R_P^\theta (t) > R_Q^\theta (t)$ holds for every $t > 0$, including at the limit $t = \infty$ (corresponding to the genericity assumption in the main text). Then, by Theorem~\ref{thm:eventual} there exists $n$ such that $P^{\otimes n}$ dominates $Q^{\otimes n}$ in the Blackwell order. Hence, by applying the previous lemma and by additivity, we obtain
\[
    nD(P_0,P_1) = D(P_0^n,P_1^n) \geq D(Q_0^n,Q_1^n) = nD(Q_0,Q_1).
\]

More generally, suppose we only have the weak inequality $R_P^\theta (t) \geq R_Q^\theta (t)$ for $t > 0$. Fix a bounded and non-trivial experiment $S = (S_0,S_1)$. Then, for every $k \in \mathbb{N}$ we have
\[
    R_{P^{\otimes k} \otimes S}^\theta (t) = k R_{P}^\theta(t) + R_S^\theta (t) > k R_Q^\theta(t) = R_{Q^{\otimes k}}^\theta (t)
\]
for every $t \in (0, \infty]$ and $\theta \in \{0,1\}$. Given what we just proved, it follows that 
\[
D(P_0^k \times S_0, P_1^k \times S_1) \geq D(Q_0^k, Q_1^k).
\]
By additivity, $D(P_0,P_1) + \frac{1}{k}D(S_0,S_1) \geq D(Q_0,Q_1)$. Since this holds for every $k$ and $D(S_0, S_1)$ is finite, the proof is concluded.
\end{proof}
 
Let $\overline{\R} = [-\infty,\infty]$ be the extended real line. Given a bounded experiment $P$ we define the function $H_P \colon \overline{\R} \to \R$ as
\[
H_P(t) = 
    \begin{cases}
    R^1_P(t) ~~\text{if}~~ t \geq 1/2 \\
    R^0_P(1 - t) ~~ \text{if}~~ t \leq 1/2
    \end{cases}
\]
Recall that the R\'enyi divergences of an experiment $P$ satisfy the relation $(1-t)R^1_P(t) =tR^0_P(1-t)$. This implies that the function $H_P$ is well defined, continuous, and bounded. It is a convenient representation of the R\'enyi divergences that retains the main properties of the latter, and has the advantage of being strictly positive whenever $P$ is nontrivial. Since $H_P(t)$ is continuous and has a compact domain, it is furthermore bounded away from $0$. The functional $P \mapsto H_P$ satisfies two additional properties. An experiment $P$ dominates an experiment $Q$ in the R\'enyi order if and only if $H_P(t) > H_Q(t)$ for every $t$. Moreover, the functional is additive: $H_{P \otimes Q}(t) = H_P(t) + H_Q(t)$ for every $t$. 
 
Thus, to prove Theorem \ref{thm:divergences} it suffices to show that under the hypotheses of the theorem there exists a finite measure $m$ on $\overline{\R}$ such that for every bounded pair of measures $P_0, P_1$
\[
   D(P_0,P_1) = \int_{\overline{\R}} H_P(t)\,\dd m(t)
\]
where $P$ is the experiment $(P_0, P_1)$. 
The theorem's conclusion \eqref{eq:D-rep} follows easily from this by setting $\dd m_0(t) = -\dd m(1-t)$ and $\dd m_1(t) = \dd m(t)$ for $t \geq \frac{1}{2}$.
 
\bigskip
 
Let $C(\overline{\R})$ be the space of continuous functions defined over the compact set $\overline{\R}$. Each function $H_P$ belongs to $C(\overline{\R})$. Consider the set
\[
    \mathcal{H} = \{H_P: P \text{~is a bounded experiment}\} \subseteq C(\overline{\R}).
\]
By Lemma~\ref{lem:proof-divergences-1}, if $H_P = H_Q$ then $D(P_0, P_1) = D(Q_0, Q_1)$. Thus there exists a map $F\colon \mathcal{H} \to \mathbb{R}$ such that $D(P_0,P_1) = F(H_P)$.
 
By Lemma~\ref{lem:proof-divergences-1} the functional $F$ is monotone. It is moreover additive: Given two experiments $P$ and $Q$, the additivity of $D$ and the additivity of $P \mapsto H_P$ imply
\begin{align*}
     F(H_P) + F(H_Q) 
     &= D(P_0, P_1) + D(Q_0,Q_1)\\
     &= D(P_0 \times Q_0, P_1 \times Q_1)\\ 
     &= F(H_{P {\otimes} Q}) \\
     &= F(H_P + H_Q).
\end{align*}
 
Next, we define
$
    \textup{cone}_{\mathbb{Q}}(\mathcal{H}) = \left\{\sum_{i=1}^n \alpha_i H_{P^i} : \alpha_i \in \mathbb{Q}_+, P^i \text{~is a bounded experiment} \right\}
$
to be the rational cone generated by $\mathcal{H}$, where coefficients $(\alpha_i)$ are positive rational numbers. Similarly define 
\[
\textup{cone}(\mathcal{H}) = \left\{\sum_{i=1}^n \alpha_i H_{P^i} : \alpha_i \in \mathbb{R}_+, P^i \text{~is a bounded experiment} \right\}
\]
to be the cone generated by $\mathcal{H}$, where coefficients can be all positive numbers. Below we extend the functional $F$ from $\mathcal{H}$ to $\textup{cone}_{\mathbb{Q}}(\mathcal{H})$ and then to $\textup{cone}(\mathcal{H})$.

Because $P \mapsto H_P$ is additive, $\mathcal{H}$ is itself closed under addition. This implies
\[
    \textup{cone}_{\mathbb{Q}}(\mathcal{H}) = \bigcup_{n \geq 1} \frac{1}{n}\mathcal{H}.
\]
Define $G \colon \textup{cone}_{\mathbb{Q}}(\mathcal{H}) \to \R$ as $G(\frac{1}{n}H_P) = \frac{1}{n}F(H_P)$. The functional $G$ is well-defined: If $\frac{1}{n} H_P = \frac{1}{m} H_Q$ then $H_{P^{\otimes m}} = m H_P = n H_Q = H_{Q^{\otimes n}}$, which implies $m F(H_P) = n F(H_Q)$ by the additivity of $F$. Similarly, $G$ inherits the monotonicity and additivity of $F$ on the larger domain $\textup{cone}_{\mathbb{Q}}(\mathcal{H})$.
 
\bigskip 
 
We now show $G$ is a Lipschitz functional, where we endow the space $C(\overline{\R})$ with the sup norm. Let $S_0$ be a nontrivial experiment, so that $H_{S_0}(t)$ is positive and in fact bounded away from $0$ for every $t$. By letting $S = S_0^{\otimes k}$ for large $k$, we obtain that $H_S(t) > 1$ for every $t$. Given two functions $f, \hat{f} \in \textup{cone}_{\mathbb{Q}}(\mathcal{H})$, we have the pointwise comparison
\[
    f(t) \leq \hat{f}(t) + \Vert f - \hat{f} \Vert \times H_S(t).
\]
Let $r > \Vert f - \hat{f} \Vert$ be a rational number. The additivity and the monotonicity of $G$ imply
\[
    G(f) \leq G(\hat{f} + rH_S) = G(\hat{f}) + rG(H_S).
\]
Symmetrically $G(\hat{f}) \leq G(f + rH_S) = G(f) + rG(H_S)$, so that $\vert G(f) - G(\hat{f}) \vert \leq r G(H_S)$. By taking the limit $r \to \Vert f - \hat{f} \Vert$ we obtain that $G$ is Lipschitz with Lipschitz constant $G(H_S) < \infty$, i.e.
\[
    \vert G(f) - G(\hat{f}) \vert \leq \Vert f - \hat{f} \Vert \cdot G(H_S).
\]

 Thus $G$ can be extended to a Lipschitz functional $\overline{G}$ defined on the closure of $\textup{cone}_{\mathbb{Q}}(\mathcal{H})$, which contains $\textup{cone}(\mathcal{H})$.

 We now verify that $\overline{G}$ is still monotone on $\textup{cone}(\mathcal{H})$. Let $f \geq \hat{f}$ be two functions in $\textup{cone}(\mathcal{H})$, and take any two sequences $\{\frac{1}{p_n} H_{P_n}\}$ and $\{\frac{1}{q_n} H_{Q_n}\}$ in $\textup{cone}_{\mathbb{Q}}(\mathcal{H})$ that converge to $f$ and $\hat{f}$ as $n \to \infty$. For any positive integer $m$, convergence in the sup-norm implies $\frac{1}{p_n} H_{P_n} \geq f - \frac{1}{2m} H_S$ for all large $n$, where $S$ is the experiment with $H_S > 1$ everywhere. Similarly $\frac{1}{q_n} H_{Q_n} \leq \hat{f} + \frac{1}{2m} H_S$. Since $f \geq \hat{f}$, we thus have $\frac{1}{p_n} H_{P_n} \geq \frac{1}{q_n} H_{Q_n} - \frac{1}{m} H_S$ for all large $n$. By monotonicity and additivity of $G$, $G(\frac{1}{p_n} H_{P_n}) \geq G(\frac{1}{q_n} H_{Q_n}) - \frac{1}{m} G(H_S)$, which implies $\overline{G}(f) \geq \overline{G}(\hat{f}) - \frac{1}{m} G(H_S)$ by taking $n \to \infty$. As $m$ is arbitrary, we have shown that $\overline{G}$ is monotonic. 

 We show $\overline{G}$ is additive and satisfies $\overline{G}(a f + b\hat{f}) = a \overline{G}(f) + b \overline{G}(\hat{f})$ for any functions $f, \hat{f} \in \textup{cone}(\mathcal{H})$ and $a, b \in \R_+$. To show this, first suppose $a, b$ are rational numbers. Consider $\{\frac{1}{p_n} H_{P_n}\} \to f$ and $\{\frac{1}{q_n} H_{Q_n}\} \to \hat{f}$ as above, where $f$ need not be bigger than $\hat{f}$. Then the sequence of functions $\{\frac{a}{p_n} H_{P_n} + \frac{b}{q_n} H_{Q_n}\} \in \textup{cone}_{\mathbb{Q}}(\mathcal{H})$ converges to $af + b\hat{f}$. It follows that 
\begin{align*}
\overline{G}(af+b\hat{f}) &= \lim_{n \to \infty} G(\frac{a}{p_n} H_{P_n} + \frac{b}{q_n} H_{Q_n}) \\
&= a \cdot \lim_{n \to \infty} G(\frac{1}{p_n} H_{P_n}) + b \cdot \lim_{n \to \infty} G(\frac{1}{q_n} H_{Q_n}) = a \cdot \overline{G}(f) + b \cdot \overline{G}(\hat{f}). 
\end{align*}
If $a, b$ are real numbers, we can deduce the same result by the Lipschitz property of $\overline{G}$. 

\bigskip

Consider next $V = \textup{cone}(\mathcal{H}) - \textup{cone}(\mathcal{H})$, which is vector subspace of $C(\overline{\R})$. $\overline{G}$ can be further extended to a functional $I \colon V \to \R$, defined as
\[
    I(M_1  - M_2) = \overline{G}(M_1) - \overline{G}(M_2)
\]
for all $M_1,M_2 \in \textup{cone}(\mathcal{H})$. The functional $I$ is well defined and linear because $\overline{G}$ is affine. Moreover, by monotonicity of $\overline{G}$, $I(f) \geq 0$ for any non-negative function $f \in V$.

The following theorem, a generalization of the Hahn-Banach Theorem \citep[see, e.g., Theorem 8.32 in][]{guide2006infinite}, shows that $I$ can be further extended to a positive linear functional on the entire space $C(\overline{\R})$: 
\begin{theorem}[\cite{kantorovich1937moment}]
Let $V$ be a vector subspace of $C(\overline{\R})$ with the property that for every $f \in C(\overline{\R})$ there exists a function $g \in V$ such that $g \geq f$. Then every positive linear functional on $V$ extends to a positive linear functional on $C(\overline{\R})$.
\end{theorem}

\noindent The ``majorization'' condition $g \geq f$ is satisfied because every function in $C(\overline{\R})$ is bounded by some $n$, and $V$ contains the function $nH_{S}$ which takes values greater than $n$ everywhere. 

\bigskip

To summarize, we have obtained a positive linear functional $J$ defined on $C(\overline{\R})$ that extends the original functional $F(H_P) = D(P_0, P_1)$. By the Riesz Representation Theorem for positive linear functionals over spaces of continuous functions on compact sets, we conclude that $J(f) = \int_{\overline{\R}} f(t) \, \dd m(t)$ for some finite measure $m$. Hence $D(P_0, P_1) = F(H_P) = J(H_P)$ is an integral of the R\'enyi divergences of $P$, completing the proof of Theorem \ref{thm:divergences}.

\section{Necessity of the Genericity Assumption}\label{sec:eventualFail}

Here we present examples to show that Theorem~\ref{thm:eventual} does not hold without the genericity assumption.

Consider the experiments $P$ and $Q$ described in Example 2 in \S\ref{sec:Azrieli}. Fix $\alpha = \frac{1}{4}$ and $\beta = \frac{1}{16}$, which satisfy \eqref{eq:AzrieliCondition}. Then by Proposition~\ref{prop:AzrieliConjecture}, $P$ dominates $Q$ in large samples. 

We will perturb these two experiments by adding another signal realization (to each experiment) which strongly indicates the true state is $1$. The perturbed conditional probabilities are given below:
\begin{equation*}
\tilde{P}: \quad
\begin{array}{ccccc}
\hline
\hline
\theta &x_0 & x_1 & x_2 & x_3\\
\hline
0 & \eps & \frac{1}{16} & \frac12 & \frac{7}{16} - \eps \\
1 & 100\eps & \frac{7}{16} & \frac12 & \frac{1}{16} - 100\eps \\
\hline
\end{array} 
\qquad \qquad \qquad
\tilde{Q}: \quad 
\begin{array}{cccc}
\hline
\hline
\theta & y_0 & y_1 & y_2\\
\hline
0 & \eps & \frac14 & \frac34-\eps \\
1 & 100\eps & \frac34 & \frac14-100\eps \\
\hline
\end{array} 
\end{equation*} 

If $\eps$ is a small positive number, then by continuity $\tilde{P}$ still dominates $\tilde{Q}$ in the R\'enyi order. Nonetheless, we show below that $\tilde{P}^{\otimes n}$ does not Blackwell dominate $\tilde{Q}^{\otimes n}$ for any $n$ and $\eps>0$. 

To do this, let $\overline{p} := \frac{100^{n-1}}{100^{n-1}+1}$ be a threshold belief. We will show that a decision maker whose indirect utility function is $(p-\overline{p})^{+}$ strictly prefers $\tilde{Q}^{\otimes n}$ to $\tilde{P}^{\otimes n}$. Indeed, it suffices to focus on posterior beliefs $p > \overline{p}$; that is, the likelihood ratio should exceed  $100^{n-1}$. Under $\tilde{Q}^{\otimes n}$, this can only happen if every signal realization is $y_0$, or all but one signal is $y_0$ and the remaining one is $y_1$. Thus, in the range $p > \overline{p}$, the posterior belief has the following distribution under $\tilde{Q}^{\otimes n}$: 
\[
p = 
\begin{cases}
\frac{100^n}{100^n+1} ~~~~~~\text{w.p.}~~ \frac{1}{2}(100^n+1)\eps^n \\
\frac{3\cdot100^{n-1}}{3\cdot100^{n-1}+1} ~~\text{w.p.}~~ \frac{n}{8}(3\cdot100^{n-1}+1)\eps^{n-1}
\end{cases}
\]
Similarly, under $\tilde{P}^{\otimes n}$ the relevant posterior distribution is
\[
p = 
\begin{cases}
\frac{100^n}{100^n+1} ~~~~~~\text{w.p.}~~ \frac{1}{2}(100^n+1)\eps^n \\
\frac{7\cdot100^{n-1}}{7\cdot100^{n-1}+1} ~~\text{w.p.}~~ \frac{n}{32}(7\cdot100^{n-1}+1)\eps^{n-1}
\end{cases}
\]

Recall that the indirect utility function is $(p-\overline{p})^{+}$. So $\tilde{Q}^{\otimes n}$ yields higher expected payoff than $\tilde{P}^{\otimes n}$ if and only if
\[
\frac{n}{8}(3\cdot 100^{n-1} + 1)\eps^{n-1} \cdot \left(\frac{3\cdot100^{n-1}}{3\cdot100^{n-1}+1} - \overline{p}\right) > \frac{n}{32}(7\cdot 100^{n-1} + 1)\eps^{n-1} \cdot \left(\frac{7\cdot100^{n-1}}{7\cdot100^{n-1}+1} - \overline{p} \right).
\]
That is, 
\[
4(3\cdot 100^{n-1} + 1)\cdot \left(\frac{3\cdot100^{n-1}}{3\cdot100^{n-1}+1} - \frac{100^{n-1}}{100^{n-1}+1}\right) > (7\cdot 100^{n-1} + 1) \cdot \left(\frac{7\cdot100^{n-1}}{7\cdot100^{n-1}+1} - \frac{100^{n-1}}{100^{n-1}+1} \right). 
\]
The LHS is computed to be $\frac{8 \cdot 100^{n-1}}{100^{n-1}+1}$, while the RHS is $\frac{6 \cdot 100^{n-1}}{100^{n-1}+1}$. Hence the above inequality holds, and it follows that $\tilde{P}^{\otimes n}$ does not Blackwell dominate $\tilde{Q}^{\otimes n}$.

\section{Generalization to Unbounded Experiments} \label{sec:unbounded}

In this section we present two generalizations of Theorem~\ref{thm:eventual} to experiments that may have unbounded likelihood ratios. Note that the R\'enyi divergences for an unbounded experiment can still be defined by \eqref{eq:R_t}, \eqref{eq:R_1} and \eqref{eq:Renyi divergence}, so long as we allow these divergences to take the value $+\infty$. 

The first result shows that Theorem~\ref{thm:eventual} hold without change so long as the dominated experiment $Q$ is bounded. 

\begin{theorem}\label{thm:bounded Q}
For a generic pair of experiments $P$ and $Q$ where $Q$ is bounded, the following are equivalent:
\begin{enumerate}
    \item $P$ dominates $Q$ in large samples.
    \item $P$ dominates $Q$ in the R\'enyi order. 
\end{enumerate}
\end{theorem}

To interpret the statement, ``generic'' means (as in the main text) that $\log \frac{\dd P_1}{\dd P_0}$ has different essential maximum and minimum from $\log \frac{\dd Q_1}{\dd Q_0}$. In the current setting $P$ may be unbounded, so that its log-likelihood ratio may have essential maximum $+\infty$ and/or minimum $-\infty$. In those cases the the genericity assumption is automatically satisfied. 

We also reiterate that dominance in the R\'enyi order means the R\'enyi divergences of $P$ and $Q$ are ranked as $R_P^{\theta}(t) > R_Q^{\theta}(t)$ for all $t > 0$ and $\theta \in \{0,1\}$. Since $Q$ is by assumption bounded, $R_Q^{\theta}(t)$ is always finite. Thus the requirement in (ii) is that $R_P^{\theta}(t)$ is either a bigger finite number, or it is $+\infty$. 

\bigskip

Our second result in this section deals with pairs of experiments where both $P$ and $Q$ may be unbounded, but they still have finite R\'enyi divergences. To state the result, we need to generalize the notion of genericity as follows: Say $P$ and $Q$ form a \emph{generic} pair, if for both $\theta = 0$ and $\theta = 1$, 
\begin{equation}\label{eq:unbounded genericity} 
\liminf_{t \to \infty} \vert R_P^{\theta}(t) - R_Q^{\theta}(t) \vert > 0.
\end{equation}
Note that when $P$ and $Q$ are bounded, $R_P^{\theta}(t) \to \max[X^{\theta}]$ and $R_Q^{\theta}(t) \to \max[Y^{\theta}]$ as $t \to \infty$. So in this special case the genericity assumption reduces to the one we introduced in the main text. 

The following result shows that under one extra assumption, Theorem~\ref{thm:eventual} once again extends.
\begin{theorem}\label{thm:unbounded Q}
Suppose $P$ and $Q$ are a generic pair of (possibly unbounded) experiments with finite R\'enyi divergences. Let $(X^\theta), (Y^\theta)$ be the corresponding log-likelihood ratios, and suppose further that their cumulant generating functions satisfy $\sup_{t \in \mathbb{R}} K''_{X^\theta}(t) < \infty$ and $\sup_{t \in \mathbb{R}} K''_{Y^\theta}(t) < \infty$.\footnote{Since $K_{X^0}(t) = K_{X^1}(-1-t)$, it suffices to check the assumptions $\sup_{t \in \mathbb{R}} K''_{X^\theta}(t) < \infty$ and $\sup_{t \in \mathbb{R}} K''_{Y^\theta}(t) < \infty$ for one of the two states.}  Then the following are equivalent:
\begin{enumerate}
    \item $P$ dominates $Q$ in large samples.
    \item $P$ dominates $Q$ in the R\'enyi order. 
\end{enumerate}
\end{theorem}

We note that if a random variable $X$ is bounded between $-b$ and $b$, then its R\'enyi divergences are finite, and $K_X''(t) \leq b^2$ for every $t$.\footnote{The latter follows by showing $K_X''(t)$ to be the variance of some random variable $\hat{X}$ that shares the same support as $X$. See Proposition~\ref{prop:unbounded large-deviations-compare} and its proof.} Thus Theorem~\ref{thm:unbounded Q} is another strict generalization of Theorem~\ref{thm:eventual} beyond bounded experiments.

More generally, the following is a sufficient condition for Theorem~\ref{thm:unbounded Q} to apply. Roughly speaking, we require the log-likelihood ratios $X^{\theta}, Y^{\theta}$ to have \emph{tails decaying faster than some Gaussian distribution}. 

\begin{lemma}\label{lem:Kx second derivative}
Let $X$ be a random variable whose distribution admits a density $h(x)$ that is positive and twice continuously differentiable. Suppose there exists $\epsilon > 0$ and $M > 0$ such that the following holds:
\[
\frac{\partial^2 \log h(x)}{\partial x^2} \leq -\epsilon ~~~ \text{for all} ~~ \vert x \vert > M. 
\]
Then the cumulant generating function $K_X(t)$ is finite for every $t$, and $\sup_{t \in \mathbb{R}} K_X''(t) < \infty$. 
\end{lemma}

Note that $\frac{\partial^2 \log h(x)}{\partial x^2} \leq -\epsilon$ implies the standard assumption that the density $h$ is (strictly) log-concave. The requirement that the same $\epsilon$ works for all large $x$ makes our assumption stronger, and in particular rules out densities such as $h_1(x) = c_1 \cdot e^{-\lambda_1 \vert x \vert}$ or $h_2(x) = c_2 \cdot e^{-\lambda_2 \vert x \vert^{1.99}}$.\footnote{It is easy to see that the random variable with density $h_1(x)$ does not have finite R\'enyi divergences everywhere. It can also be shown that the random variable with density $h_2(x)$ has a cumulant generating function with $K_X''(t) \to \infty$ as $t \to \infty$. Thus, it seems difficult to substantially weaken the condition in Lemma \ref{lem:Kx second derivative} while maintaining the same result.} Nonetheless, any Gaussian density $h$ satisfies the assumption regardless of how big the variance is, and so does any other density that decays faster at infinity. Hence Theorem \ref{thm:unbounded Q} is applicable to a broad class of unbounded experiments. 

\bigskip

Below we prove Theorem~\ref{thm:bounded Q}, Theorem~\ref{thm:unbounded Q} and Lemma~\ref{lem:Kx second derivative} in turn. 

\subsection{Proof of Theorem~\ref{thm:bounded Q}}
That (i) implies (ii) follows from the same argument as in \S\ref{sec:Renyi decision problems}. To prove (ii) implies (i), the idea is to garble $P$ into a bounded experiment $\tilde{P}$ that still has higher R\'enyi divergences than $Q$. By Theorem~\ref{thm:eventual}, $\tilde{P}^{\otimes n}$ Blackwell dominates $Q^{\otimes n}$ for all large $n$. But since $P$ Blackwell dominates $\tilde{P}$, $P^{\otimes n}$ also Blackwell dominates $\tilde{P}^{\otimes n}$. Therefore, by transitivity, we would be able to conclude that $P^{\otimes n}$ Blackwell dominates $Q^{\otimes n}$ for all large $n$. 

To construct such a $\tilde{P}$, we first note that by taking $t \to \infty$, $R_P^1(t) > R_Q^1(t)$ implies $\max[X^1] \geq \max[Y^1]$ where $X^1$ and $Y^1$ are the log-likelihood ratios. Similarly $\max[X^0] \geq \max[Y^0]$. By the genericity assumption, both comparisons are in fact strict. We can thus find a pair of positive numbers $b_1 \in (\max[Y^1], \max[X^1])$ and $b_0 \in (\max[Y^0], \max[X^0]) = (-\min[Y^1], -\min[X^1])$. These numbers will be fixed throughout. 

Now take any positive number $B \geq \max\{b_1, b_0\}$. We construct a garbling of $P$, denoted $P_B$, as follows: All signal realizations under $P$ that induce a log-likelihood ratio $\log \frac{\dd P_1}{\dd P_0}$ greater than $B$ (if any) are garbled into a single signal $\overline{s}$, and similarly all realizations with log-likelihood ratio less than $-B$ are garbled into another signal $\underline{s}$. The remaining signal realizations under $P$ (with log-likelihood ratio in $[-B, B]$) are unchanged under $P_B$. It is easy to see that not only is $P_B$ a garbling of $P$, but more generally $P_{B}$ is a garbling of $P_{B'}$ whenever $B' > B$. Thus, as $B$ increases, the experiment $P_B$ becomes more informative in the Blackwell sense. 

Let $R_{P_B}^{\theta}(t)$ denote the R\'enyi divergences of $P_B$. Since the R\'enyi order extends the Blackwell order, we know that as $B$ increases, $R_{P_B}^{\theta}(t)$ also increases for each $\theta$ and $t$, with an upper bound of $R_{P}^{\theta}(t)$. In fact, we can show that for fixed $\theta$ and $t$, 
\[
\lim_{B \to \infty} R_{P_B}^{\theta}(t) = R_P^{\theta}(t). 
\]
The proof is technical and deferred to later. Assuming this, we next show that for sufficiently large $B$, $R_{P_B}^{\theta}(t) > R_{Q}^{\theta}(t)$ holds for \emph{all} $t \geq 1/2$ (thus for all $t > 0$, by \eqref{eq:t-1-t}). This will prove $P_B$ as the desired garbling $\tilde{P}$ that dominates $Q$ in the R\'enyi order, which will complete the proof of the theorem.\footnote{Note that $B \geq \max\{b_1, b_2\}$ ensures $P_B$ and $Q$ is a generic pair, so we can apply Theorem~\ref{thm:eventual} to deduce $P_B^{\otimes n} \succeq Q^{\otimes n}$ for large $n$. Therefore $P^{\otimes n} \succeq P_B^{\otimes n} \succeq Q^{\otimes n}$.} 
 
\bigskip
 
To this end, fix $\theta = 1$, and define for each $B$ a set
\[
T_B = \{t \geq 1/2: R_{P_B}^{1}(t) \leq R_{Q}^{1}(t)\}.
\]
By continuity of the R\'enyi divergences, $T_B$ is a closed set.  Moreover, as $t \to \infty$ we have $R_{P_B}^{1}(t) \to \max[X_B^1]$, where $X_B^1$ is the log-likelihood ratio of state $1$ to state $0$, distributed under the experiment $P_B$ and true state $1$. By the assumption $B \geq b_1$ and the construction of $P_B$, we have that 
\[
\mathbb{P}[X_B^1 \geq b_1] = \mathbb{P}[X^1 \geq b_1],
\]
which is positive because $b_1 < \max[X^1]$. Thus $\max[X_B^1] \geq b_1$. It follows that 
\[
\lim_{t \to \infty} R_{P_B}^{1}(t) \geq b_1 > \max[Y^1] = \lim_{t \to \infty} R_{Q}^{1}(t). 
\]
Hence $R_{P_B}^{1}(t) > R_{Q}^{1}(t)$ for all large $t$ and $T_B$ is a bounded set. 

We have shown that each $T_B$ is compact set. Note also that because $R_{P_B}^{1}(t)$ increases in $B$, the set $T_B$ shrinks as $B$ increases. Therefore, by the finite intersection property, either there exists some $t$ that belongs to every $T_B$, or $T_B$ is the empty set for all large $B$. The former is impossible because $R_{P_B}^1(t) \leq R_Q^1(t)$ for all $B$ would imply $R_P^1(t) \leq R_Q^1(t)$ in the limit, contradicting the assumption in (ii). 

We thus conclude that $T_B$ must be empty for all large $B$. In other words, when $B$ is large $R_{P_B}^{1}(t) > R_{Q}^{1}(t)$ holds for all $t \geq \frac{1}{2}$. A symmetric argument shows that $R_{P_B}^{0}(t) > R_{Q}^{0}(t)$ holds for all $t \geq \frac{1}{2}$, completing the proof. 

\bigskip

It remains to show $\lim_{B \to \infty} R_{P_B}^{\theta}(t) = R_P^{\theta}(t)$. We again fix $\theta = 1$ for easier exposition. Consider the following three cases:

\paragraph{Case 1: $t > 1$.} We recall that $R_{P_B}^{1}(t) = \frac{1}{t-1} \log \mathbb{E}[\ee^{(t-1)X_B^1}]$. So we need to show 
\[
\lim_{B \to \infty} \mathbb{E}[\ee^{(t-1)X_B^1}] = \mathbb{E}[\ee^{(t-1)X^1}].
\]
Since $R_{P_B}^{1}(t) \leq R_P^1(t)$ for each $B$, the LHS above is weakly smaller than the RHS. On the other hand, by construction $X_B^1$ coincides with $X^1$ conditional on being in the interval $[-B, B]$. As the exponential function is always positive, we have
\begin{align*}
\mathbb{E}[\ee^{(t-1)X_B^1}] &\geq \Pr{\vert X_B^1 \vert \leq B} \cdot \mathbb{E}[\ee^{(t-1)X_B^1} ~\mid~ \vert X_B^1 \vert \leq B] \\ &= \Pr{\vert X^1 \vert \leq B} \cdot \mathbb{E}[\ee^{(t-1)X^1} ~\mid~ \vert X^1 \vert \leq B].
\end{align*}
Taking the limit as $B \to \infty$, we obtain $\lim_{B \to \infty} \mathbb{E}[\ee^{(t-1)X_B^1}] \geq \mathbb{E}[\ee^{(t-1)X^1}]$, which proves they are equal. 

\paragraph{Case 2: $t = 1$.} Here we have $R_{P_B}^{1}(1) = \mathbb{E}[X_B^1]$. So we need to show 
\[
\lim_{B \to \infty} \mathbb{E}[X_B^1] = \mathbb{E}[X^1].
\]
Once again we already know the LHS is weakly smaller, so it suffices to show the opposite inequality. By construction, $X_B^1$ coincides with $X^1$ on the interval $[-B, B]$. Other than this part, there is probability $\mathbb{P}[X^1 > B]$ that signal $\overline{s}$ occurs under the experiment $P_B$; when this happens we also have $X_B^1 > B$, which contributes a positive amount to $\mathbb{E}[X_B^1]$. 

With remaining probability $\mathbb{P}[X^1 < -B]$, the signal $\underline{s}$ occurs, and the induced log-likelihood ratio $X_B^1$ is at least $\log \mathbb{P}[X^1 < -B]$ (since this event occurs with probability at most one under state $0$). Here the contribution to $\mathbb{E}[X_B^1]$ can be negative, but is no less than $\mathbb{P}[X^1 < -B] \cdot \log \mathbb{P}[X^1 < -B]$. 

Summarizing, for each $B$ we have
\[
\mathbb{E}[X_B^1] \geq \Pr{\vert X^1 \vert \leq B} \cdot \mathbb{E}[X^1 ~\mid~ \vert X^1 \vert \leq B] ~~+~~ \mathbb{P}[X^1 < -B] \cdot \log \mathbb{P}[X^1 < -B]. 
\]
Taking the limit as $B \to \infty$, the first summand on the RHS converges to $\mathbb{E}[X^1]$. In addition, the second summand vanishes because $\mathbb{P}[X^1 < -B] \to 0$ and $\lim_{x \to 0} x \log x = 0$. We thus obtain $\lim_{B \to \infty} \mathbb{E}[X_B^1] \geq \mathbb{E}[X^1]$ as desired. 

\paragraph{Case 3: $t \in (0,1)$.} In this case we will again show
\[
\lim_{B \to \infty} \mathbb{E}[\ee^{(t-1)X_B^1}] = \mathbb{E}[\ee^{(t-1)X^1}].
\]
Since $R_{P_B}^{1}(t) \leq R_P^1(t)$, and $R_{P_B}^{1}(t) = \frac{1}{t-1} \log \mathbb{E}[\ee^{(t-1)X_B^1}]$, the negative factor $\frac{1}{t-1}$ implies that the LHS above is now weakly \emph{bigger} than the RHS. 

To prove it is smaller, we proceed as in Case 2. With probability $\mathbb{P}[X^1 > B]$ the signal $\overline{s}$ occurs, and the induced log-likelihood ratio $X_B^1$ is \emph{at least} $\log \mathbb{P}[X^1 > B]$. As $t-1$ is negative here, the contribution of this part to $\mathbb{E}[\ee^{(t-1)X_B^1}]$ is \emph{at most}
\[
\mathbb{P}[X^1 > B] \cdot \mathbb{E}[\ee^{(t-1)\log \mathbb{P}[X^1 > B]}] = (\mathbb{P}[X^1 > B])^t.
\]
Similarly the contribution of the signal $\underline{s}$ is at most 
$(\mathbb{P}[X^1 < -B])^t$. We thus have
\[
\mathbb{E}[\ee^{(t-1)X_B^1}] \leq \Pr{\vert X^1 \vert \leq B} \cdot \mathbb{E}[\ee^{(t-1)X^1} ~\mid~ \vert X^1 \vert \leq B] ~~+~~ (\mathbb{P}[X^1 > B])^t ~~+~~ (\mathbb{P}[X^1 < -B])^t. 
\]
As $B \to \infty$, both $(\mathbb{P}[X^1 > B])^t$ and $(\mathbb{P}[X^1 < -B])^t$ vanish since $t > 0$. We therefore conclude $\lim_{B \to \infty} \mathbb{E}[\ee^{(t-1)X_B^1}] \leq \mathbb{E}[\ee^{(t-1)X^1}]$, completing the whole proof. 

\subsection{Proof of Theorem~\ref{thm:unbounded Q}}
We only need to prove (ii) implies (i). Here we will follow the arguments in \S\ref{sec:sufficiency} and make necessary modifications. Since Lemma~\ref{lem:signs} remains valid, it suffices to prove \eqref{eq:F1G1 restated}, i.e., 
\[
\Pr{X^1_1 + \dots + X^1_n \leq na} \leq \Pr{Y^1_1 + \dots + Y^1_n \leq na}, ~~~\text{ for all } a \geq 0.
\]
The analysis of the four cases in \S\ref{sec:sufficiency} relies on Lemma~\ref{lem:eta} and Proposition~\ref{prop:large-deviations-compare}. We will show later that Lemma~\ref{lem:eta} continues to hold even if $P$ and $Q$ are unbounded (but have finite R\'enyi divergences). On the other hand, Proposition~\ref{prop:large-deviations-compare} cannot hold as stated, but we do have the following modified version where $b^2$ is replaced by $\sup_{t \in \mathbb{R}} K_X''(t)$:

\begin{proposition}\label{prop:unbounded large-deviations-compare}
Let $X$ and $Y$ be random variables with finite cumulant generating functions $K_X(t)$ and $K_Y(t)$. Further let $X_1, \dots, X_n$, $Y_1, \dots, Y_n$ be i.i.d.\ copies of $X$ and $Y$ respectively. Suppose $a \geq \mathbb{E}[Y]$, and $\eta > 0$ satisfies $K_Y^*(a)-\eta>K_X^*(a+\eta)$. Then for all 
$
  n \geq 4(1+\eta)\eta^{-3} \cdot \sup_{t \in \mathbb{R}} K_X''(t),
$ 
it holds that
\[
    \Pr{X_1 + \cdots + X_n > na} \geq \Pr{Y_1 + \cdots + Y_n > na}.
\]
\end{proposition}

Using Lemma~\ref{lem:eta} and Proposition~\ref{prop:unbounded large-deviations-compare}, we can replicate the results in Cases 1, 2 and 4 in \S\ref{sec:sufficiency}. Specifically, let $M = \max \{\sup_{t \in \mathbb{R}} K_{X^1}''(t), \sup_{t \in \mathbb{R}} K_{Y^1}''(t) \}$, then for all $n \geq 4M(1+\eta)\eta^{-3}$ the inequality $\Pr{X^1_1 + \dots + X^1_n \leq na} \leq \Pr{Y^1_1 + \dots + Y^1_n \leq na}$ holds for values of $a$ outside of the interval $(\mathbb{E}[Y] + \eta, \mathbb{E}[X] - \eta)$ in Case 3. 

Turning to $a \in (\mathbb{E}[Y] + \eta, \mathbb{E}[X] - \eta)$, we can still use the Chebyshev inequality to deduce 
\[
\Pr{X^1_1 + \dots + X^1_n \leq na} \leq \frac{\Var[X^1]}{n\eta^2} = \frac{K_{X^1}''(0)}{n\eta^2} \leq \frac{M}{n\eta^2}. 
\]
Similarly we also have 
\[
\Pr{Y^1_1 + \dots + Y^1_n \leq na} \geq 1 - \frac{\Var[Y^1]}{n\eta^2} \geq 1 - \frac{M}{n\eta^2}. 
\]
Thus $\Pr{X^1_1 + \dots + X^1_n \leq na} \leq \Pr{Y^1_1 + \dots + Y^1_n \leq na}$ holds for all $n \geq 2M\eta^{-2}$, and hence for all $n \geq 4M(1+\eta)\eta^{-3}$. This then implies that $P^{\otimes n}$ Blackwell dominates $Q^{\otimes n}$ for all $n \geq 4M(1+\eta)\eta^{-3}$. 

Below we supply the proofs for Lemma~\ref{lem:eta} (for unbounded experiments) and Proposition~\ref{prop:unbounded large-deviations-compare}.

\begin{proof}[Proof of Lemma~\ref{lem:eta} for unbounded experiments]
We note that the second part $K_{Y^\theta}^*(a-\eta) < K_{X^\theta}^*(a) - \eta$ continues to hold. This is because, by the same argument as in the case of bounded experiments, $K_{Y^\theta}^*(a) < K_{X^\theta}^*(a)$ holds for all $a$ in the \emph{compact} interval $[0, \mathbb{E}[Y^{\theta}]]$. Thus by (uniform) continuity, we can ``squeeze in'' a small positive $\eta$ without changing the inequality. 

The first part of Lemma~\ref{lem:eta} also holds so long as $\max[Y^\theta]$ is finite, in which case the range of $a$ under consideration is again compact. If instead $\max[Y^\theta] = \infty$, we use a new argument that takes advantage of the genericity assumption. Note that by assumption, $R_P^{\theta}(t) - R_Q^{\theta}(t)$ is positive for each $\theta$ and $t$. Given this, the genericity assumption \eqref{eq:unbounded genericity} further implies this difference is bounded away from zero as $t \to \infty$. That is, there exists small $\epsilon > 0$ and large $T > 1$ such that
\[
R_P^\theta(t) - R_Q^\theta(t) > \epsilon ~~~\text{for all} ~~ \theta \in \{0,1\}, ~t > T. 
\]
Since $K_X^{\theta}(t) = t R_P^\theta(t+1)$, we deduce
\begin{equation}\label{eq:unbounded Kx-Ky}
K_X^{\theta}(t) - K_Y^{\theta}(t) > \epsilon t > \frac{\epsilon}{2}(t+1) ~~~\text{for all} ~~ \theta \in \{0,1\}, ~t > T. 
\end{equation}

We can now prove the first part of Lemma~\ref{lem:eta}. Define $\delta > 0$ by $K_{X^\theta}'(T) = \mathbb{E}[X^\theta]+\delta$. The original proof of Lemma~\ref{lem:eta} yields that for all sufficiently small $\eta > 0$, 
\[
K_{Y^\theta}^*(a) - \eta > K_{X^\theta}^*(a + \eta)~~~\text{holds for }~~ \mathbb{E}[X^\theta]-\eta \leq a \leq \mathbb{E}[X^\theta]+\delta.
\]
Note that $\mathbb{E}[X^\theta]+\delta$ is finite, so the range of $a$ considered above is compact, enabling us to use the original argument. 
We claim that by choosing $\eta < \epsilon/2$, where $\epsilon$ is defined earlier, the same inequality holds even if $a$ is bigger than $\mathbb{E}[X^\theta]+\delta$. For this define $\hat{t}$ by $K_{X^\theta}'(\hat{t}) = a + \eta$, then $\hat{t} > T$ by the convexity of $K_X$. Therefore, by \eqref{eq:unbounded Kx-Ky},
\begin{align*}
K_{X^\theta}^*(a+\eta) &= \hat{t}(a+\eta) - K_{X^\theta}(\hat{t}) \\ 
&< \hat{t}(a+\eta) - K_{Y^\theta}(\hat{t}) - \frac{\epsilon}{2}(\hat{t}+1) \\
&< \hat{t}(a+\eta) - K_{Y^\theta}(\hat{t}) - \eta(\hat{t}+1) \\
&= \hat{t}a - K_{Y^\theta}(\hat{t}) - \eta \\
&\leq K_{Y^\theta}^*(a) - \eta.
\end{align*}
This completes the proof of Lemma~\ref{lem:eta} for unbounded experiments. 
\end{proof}

\begin{proof}[Proof of Proposition~\ref{prop:unbounded large-deviations-compare}]
Following the original proof of Proposition~\ref{prop:large-deviations-compare}, we just need to show a modified version of Lemma~\ref{lem:large-deviation-lower-bound} (with $\sup_{t \in \mathbb{R}} K_X''(t)$ replacing $b^2$):
\[
    \Pr{X_1+\cdots+X_n > na} \geq \ee^{-n\cdot K_X^*(a+\eta)}\left(1-\frac{4 \cdot \sup_{t \in \mathbb{R}} K_X''(t)}{n\eta^2}\right).
\]
This follows the same proof as in \S\ref{sec:large-deviations}, except that in applying the Chebyshev inequality, we now use 
\[
\Var[\hat{S}_n] = n \Var[\hat{X}] = n \cdot K_X''(t) \leq n \cdot \sup_{\hat{t} \in \mathbb{R}} K_X''(\hat{t})
\]
instead of $\Var[\hat{S}_n] \leq nb^2$. The key equality $\Var[\hat{X}] = K_X''(t)$ holds because 
\[
\Var[\hat{X}] = \mathbb{E}[\hat{X}^2] - \mathbb{E}[\hat{X}]^2 =  \frac{\mathbb{E}[X^2\ee^{tX}]}{\mathbb{E}[\ee^{tX}]} - \left(\frac{\mathbb{E}[X\ee^{tX}]}{\mathbb{E}[\ee^{tX}]}\right)^2 = K_X''(t).
\]
Hence the result.
\end{proof}

\subsection{Proof of Lemma~\ref{lem:Kx second derivative}}
We first prove $K_X$ is everywhere finite, i.e., $\log \mathbb{E}[\ee^{tX}]$ is finite for every $t$. Using the density $h(x)$, we can write 
\[
\mathbb{E}[\ee^{tX}] = \int_{-\infty}^{\infty} h(x) \ee^{tx} \,\dd x = \int_{-\infty}^{\infty} \ee^{tx + l(x)} \,\dd x, 
\]
where we define $\ell(x) = \log h(x)$. Since by assumption $\ell''(x) \leq -\epsilon$ for $\vert x \vert > M$, it is easy to show $\ell(x) \leq -\frac{\epsilon}{4} x^2$ as $\vert x \vert \to \infty$. Hence the above integral is finite. 

\bigskip

To prove $K_X''$ is bounded, we begin with the formula
\[
K_X''(t) = \frac{\mathbb{E}[X^2\ee^{tX}] \cdot \mathbb{E}[\ee^{tX}] - \mathbb{E}[X\ee^{tX}]^2}{\mathbb{E}[\ee^{tX}]^2}. 
\]
Let $X_1, X_2$ be i.i.d.\ copies of $X$. Then the denominator above is $\mathbb{E}[\ee^{tX_1}] \cdot \mathbb{E}[\ee^{tX_2}] = \mathbb{E}[\ee^{t(X_1+X_2)}]$. The numerator can be rewritten as
\begin{align*}
&\mathbb{E}[X_1^2\ee^{tX_1}] \cdot \mathbb{E}[\ee^{tX_2}] - \mathbb{E}[X_1\ee^{tX_1}] \cdot \mathbb{E}[X_2\ee^{tX_2}] \\
=& \mathbb{E}[(X_1^2 - X_1X_2) \cdot \ee^{t(X_1+X_2)}] \\
=& \mathbb{E}[\frac{X_1^2 - X_1X_2 + X_2^2 - X_1X_2}{2} \cdot \ee^{t(X_1+X_2)}] \\
=& \mathbb{E}[\frac{(X_1-X_2)^2}{2} \cdot \ee^{t(X_1+X_2)}],
\end{align*}
where the penultimate step uses the symmetry between $X_1$ and $X_2$. Define 
\[
D(s) = \mathbb{E}[(X_1 - X_2)^2 \mid X_1+X_2 = s].
\]
Then we have shown that
\[
K_X''(t) = \frac{\frac{1}{2}\mathbb{E}[D(X_1+X_2) \cdot \ee^{t(X_1+X_2)}]}{\mathbb{E}[\ee^{t(X_1+X_2)}]}. 
\]
Thus, in order to show $K_X''$ is bounded, it suffices to show $D(s)$ is bounded as $s$ varies. 

\bigskip

Recall that by assumption $\ell''(x) \leq -\epsilon$ for $\vert x \vert > M$. We will show (with proof deferred to later) there exists $S > 2M$, such that 
\begin{equation}\label{eq:unbounded ell(x)}
\ell'(x) - \ell'(s-x) \leq -\frac{\epsilon}{2}(2x-s) ~~~ \text{for all} ~~ s > S, ~x > \frac{s}{2}. 
\end{equation}
Note that \eqref{eq:unbounded ell(x)} in particular implies $\ell'(x) - \ell'(s-x) \leq -1$ for $x > \frac{s}{2} + C$, with $C = \epsilon^{-1}$. Given this, we can show $D(s)$ is bounded.

Without loss consider $s \geq 0$. We use the density $h(x)$ to write 
\begin{equation}\label{eq:unbounded Delta(s)}
D(s) = \frac{\int_{-\infty}^{\infty} h(x) h(s-x) (2x-s)^2 \, \dd x}  {\int_{-\infty}^{\infty} h(x) h(s-x) \, \dd x} = \frac{\int_{s/2}^{\infty} h(x) h(s-x) (2x-s)^2 \, \dd x}  {\int_{s/2}^{\infty} h(x) h(s-x) \, \dd x}
\end{equation}
Since $D(s)$ is continuous, it suffices to prove it is bounded when $s > S$, where $S$ is given earlier. We now break the integral in \eqref{eq:unbounded Delta(s)} into two parts, with cutoff $s/2 + 2C$:
\begin{align*}
D(s) &= \frac{\int_{s/2}^{s/2 + 2C} h(x) h(s-x) (2x-s)^2 \, \dd x}{\int_{s/2}^{\infty} h(x) h(s-x) \, \dd x} ~~+~~ \frac{\int_{s/2 + 2C}^{\infty} h(x) h(s-x) (2x-s)^2 \, \dd x}{\int_{s/2}^{\infty} h(x) h(s-x) \, \dd x}.
\end{align*}
The first term is bounded by $16C^2$, which is the maximum value of $(2x-s)^2$ for $x \in [s/2, s/2+2C]$. To bound the second term, we rewrite it as
\begin{equation}\label{eq:unbounded Delta(s) upper bound}
\int_{s/2 + 2C}^{\infty} \frac{\ee^{l(x) + l(s-x)}}{\int_{s/2}^{\infty}\ee^{l(y) + l(s-y)} \,\dd y} \cdot (2x-s)^2 \,\dd x. 
\end{equation}
As $l'(y) - l'(s-y) \leq -1$ for $y \geq s/2 + C$, we have $l(y) + l(s-y) \geq x - y + l(x) + l(s-x)$ for all $x \geq y \geq s/2 + C$. Thus 
\[
\int_{s/2}^{\infty}\ee^{l(y) + l(s-y)} \,\dd y \geq \int_{s/2 + C}^{x}\ee^{l(y) + l(s-y)} \,\dd y \geq \int_{s/2 + C}^{x}\ee^{x-y + l(x) + l(s-x)} \,\dd y = (\ee^{x-s/2-C} - 1) \ee^{l(x) + l(s-x)}. 
\]
Plugging back into \eqref{eq:unbounded Delta(s) upper bound}, the second term contributing to $D(s)$ is bounded above by 
\[
\int_{s/2 + 2C}^{\infty} \frac{1}{\ee^{x-s/2-C} - 1} \cdot (2x-s)^2 \,\dd x = \int_{C}^{\infty} \frac{1}{\ee^{u} - 1} \cdot (2u+2C)^2 \,\dd u,
\]
where we used change of variable from $x$ to $u = x - s/2 - C$. Since the RHS is a finite constant independent of $s$, we conclude that $D(s)$ is bounded even as $s \to \infty$. 

\bigskip

It remains to prove \eqref{eq:unbounded ell(x)}. We write the difference on the LHS as $\int_{s-x}^{x} \ell''(u) \,\dd u$. If $s-x > M$, the result follows from the fact that $\ell''(u) \leq -\epsilon \leq -\frac{\epsilon}{2}$ for every $u$ in the range of integration. Suppose instead that $s-x \leq M$, thus $x \geq s - M$. In this case because $\ell''(u)$ can only be positive 
for $u \in [-M, M]$, we have
\begin{align*}
\int_{s-x}^{x} \ell''(u) \,\dd u &\leq -\epsilon(2x - s - 2M) + \int_{-M}^{M} \vert \ell''(u) \vert \,\dd u \\
=& -\epsilon(x - s/2) - \epsilon(x - s/2 - 2M) + \int_{-M}^{M} \vert \ell''(u) \vert \,\dd u \\
&\leq -\epsilon(x - s/2) - \epsilon(s/2 - 3M) + \int_{-M}^{M} \vert \ell''(u) \vert \,\dd u \\
&\leq -\epsilon(x - s/2).
\end{align*}
The penultimate inequality uses $x \geq s-M$, whereas the last inequality holds when $s$ is sufficiently large (since $\int_{-M}^{M} \vert \ell''(u) \vert \,\dd u$ is finite by the assumption that $h$ is positive and twice continuously differentiable). This completes the proof.

\section{Necessary Condition for Large Sample Dominance with Many States}\label{sec:many states}

In this section we show that the R\'enyi order can be generalized to more than two states to yield a general necessary condition for large sample dominance. Consider $k+1$ states $\theta \in \{0, 1, \dots, k\}$ and two experiments $P = (\Omega, (P_{\theta}))$, $Q = (\Xi, (Q_{\theta}))$ revealing information about these states. Conditioning on $\theta = 0$, we consider the moment generating function of the log-likelihood ratio vector $(\frac{\dd P_0}{\dd P_1}, \dots, \frac{\dd P_0}{\dd P_k})$, given by 
\begin{equation}\label{eq:multidimensional mgf}
M_{X^0}(t) = \int_{\Omega} \ee^{\sum_{j = 1}^{k} t_j \log \frac{\dd P_0(\omega)}{\dd P_j(\omega)}} \, \dd P_0(\omega)
\end{equation}
with $t = (t_1, \dots, t_k) \in \R^k$. Similarly define $M_{Y^0}(t)$ for the experiment $Q$. 

By the same argument as in \S\ref{sec:Renyi decision problems} (see the derivation of \eqref{eq:Renyi decision problem}), $M_{X^0}(t)$ \emph{would be} the ex-ante expected payoff from observing $P$, in a decision problem with uniform prior and indirect utility function 
\[
v(p) = (k+1)p_0^{1+t_1+\dots+t_k} \cdot p_1^{-t_1} \cdots p_k^{-t_k},
\]
where $p = (p_0, p_1, \dots, p_k)$ represents the belief about the $k+1$ states. If the function $v(p)$ were convex in $p$, then it is indeed an indirect utility function. Blackwell dominance of $P$ over $Q$ then requires $M_{X^0}(t) \geq M_{Y^0}(t)$. Since the moment generating function is raised to the $n$-th power when $n$ i.i.d.\ samples are drawn, we would be able to conclude that $M_{X^0}(t) \geq M_{Y^0}(t)$ also has to hold if $P$ dominates $Q$ in large samples. If instead $v(p)$ were concave, then $-v(p)$ is an indirect utility function, leading to the reverse ranking between the moment generating functions.

We can characterize those parameters $t = (t_1, \dots, t_k)$ that make the function $v(p)$ globally convex/concave. To make the result easy to state, we make the variables symmetric and consider a function of the form
\[
v(p) = (k+1) p_0^{\alpha_0} \cdot p_1^{\alpha_1} \cdots p_k^{\alpha_k}
\]
with $\alpha_0 + \alpha_1 + \dots + \alpha_k = 1$. 

\begin{lemma}\label{lem:v(p) convex or concave}
Consider the function $v(p)$ defined above, over the domain $p \in int(\Delta^k)$. Suppose $\alpha_0 + \alpha_1 + \dots + \alpha_k = 1$ and $\alpha_0 > 0$. Then $v(p)$ is convex in $p$ if and only if $\alpha_1, \dots, \alpha_k$ are all non-positive. Conversely, $v(p)$ is concave in $p$ if and only if $\alpha_1, \dots, \alpha_k$ are non-negative. Moreover, the convexity/concavity is strict when $\alpha_1, \dots, \alpha_k$ are strictly negative/positive.
\end{lemma}
\noindent The proof of this lemma is deferred to the end of the section. Note that unlike the case of two states, there are situations where $v(p)$ is neither convex nor concave. 

\bigskip

By rewriting $\alpha_j = - t_j$ for $1 \leq j \leq k$, we obtain the following necessary condition for Blackwell dominance in large samples. Say the experiments $P$ and $Q$ form a generic pair, if for every pair of states $i \neq j$, the maximum and minimum of $\log \frac{\dd P_i}{\dd P_j}$ differ from those of $\log \frac{\dd Q_i}{\dd Q_j}$. 

\begin{proposition}\label{prop:many states necessary condition} 
Suppose $P$ and $Q$ are a generic pair of bounded experiments for $k+1$ states. If $P$ Blackwell dominates $Q$ in large samples, then the following conditions hold:\footnote{We exclude $t = \{\mathbf{0}\}$ from the conditions because $M_X(\mathbf{0}) = M_Y(\mathbf{0}) = 1$ always holds.} 
\begin{enumerate}
    \item For all $t \in \R_{+}^k\backslash \{\mathbf{0}\}$, $M_{X^0}(t) > M_{Y^0}(t)$ and  symmetrically $M_{X^i}(t) > M_{Y^i}(t)$ if we define the moment generating functions for true state $i$ analogously to \eqref{eq:multidimensional mgf};
    \item For all $t \in \R_{-}^k \backslash \{\mathbf{0}\}$ such that $\sum_{j=1}^{k} t_j > -1$, $M_{X^0}(t) < M_{Y^0}(t)$ and symmetrically $M_{X^i}(t) < M_{Y^i}(t)$ for $1 \leq i \leq k$; 
    \item For every pair of states $i \neq j$, the Kullback-Leibler divergence between $P_i$ and $P_j$ exceeds the divergence between $Q_i$ and $Q_j$:
    \[
        \int_{\Omega} \log \frac{\dd P_i(\omega)}{\dd P_j(\omega)} \, \dd P_i(\omega) > \int_{\Xi} \log \frac{\dd Q_i(\xi)}{\dd Q_j(\xi)} \, \dd Q_i(\xi). 
    \]
\end{enumerate}
\end{proposition}

To understand Proposition \ref{prop:many states necessary condition}, note from \eqref{eq:multidimensional mgf} that when $t_j$ are \emph{all} positive, a bigger value of $M_{X^0}(t)$ indicates higher likelihood ratios $\frac{\dd P_0}{\dd P_j}$ between state $0$ and \emph{every} other state $j$, when state $0$ is the true state. It is intuitive that in this case $M_{X^0}(t) > M_{Y^0}(t)$ corresponds to $P$ being (on average) a more informative experiment than $Q$.\footnote{To prove the strict inequality $M_{X^0}(t) > M_{Y^0}(t)$, suppose that $t_1, \dots, t_l$ are positive whereas $t_{l+1}, \dots, t_{k}$ are zero, for some $1 \leq l \leq k$. Let $\tilde{P} = (\Omega, (P_0, \dots, P_l))$ be the restriction of the experiment $P$ to the first $l+1$ states; similarly define $\tilde{Q}$. Then $P^{\otimes n} \succeq Q^{\otimes n}$ implies $\tilde{P}^{\otimes n} \succeq \tilde{Q}^{\otimes n}$, which must in fact be a strict comparison by the genericity assumption. Therefore, as the indirect utility function $\tilde{v}(p_0, \dots, p_l) = (k+1)p_0^{1+t_1+\dots+t_l} \cdot p_1^{-t_1} \cdots p_k^{-t_l}$ is \emph{strictly} convex on the smaller belief space $\Delta^l$ (Lemma \ref{lem:v(p) convex or concave}), the ex-ante expected payoff $M_{X^0}(t)$ must be strictly higher than $M_{Y^0}(t)$.} This is the content of part (i), which generalizes the comparison of R\'enyi divergences $R_P^{\theta}(t) > R_Q^{\theta}(t)$ in the two state case, for $t > 1$. 

Conversely, part (ii) says that when $t_j$ are all negative (subject to the extra condition $\sum_j t_j > -1$), informativeness is captured by the reverse ranking $M_{X^0}(t) < M_{Y^0}(t)$. In this case, the smaller value of $M_{X^0}(t)$ actually indicates higher likelihood ratios $\frac{\dd P_0}{\dd P_j}$ under true state $0$. This part generalizes the comparison $R_P^{\theta}(t) > R_Q^{\theta}(t)$ for $t \in (0, 1)$. 

Finally, part (iii) directly imposes the R\'enyi comparison $R_P^{\theta}(1) > R_Q^{\theta}(1)$ when it is applied to every pair of states. 

We conjecture that the set of necessary conditions identified in Proposition \ref{prop:many states necessary condition} are also sufficient for large sample Blackwell dominance; see \S\ref{sec:blackwell_lit} for discussion of the difficulties. 

\bigskip

Below we supply the proof of Lemma \ref{lem:v(p) convex or concave}:

\begin{proof}[Proof of Lemma \ref{lem:v(p) convex or concave}]
The Hessian matrix of $v(\cdot)$ at $p$ is computed as
\[
Hess_v(p) = v(p) \times
\left( \begin{array}{ccc} \frac{\alpha_0(\alpha_0-1)}{p_0^2} & \frac{\alpha_0\alpha_1}{p_0p_1} & \dots \\ \frac{\alpha_0\alpha_1}{p_0p_1} & \frac{\alpha_1(\alpha_1-1)}{p_1^2} & \dots \\
\dots & \dots & \dots 
\end{array} \right).
\]
For any direction $(x_0, x_1, \dots, x_k)$, the directional second derivative of $v(\cdot)$ at $p$ is thus 
\begin{equation}\label{eq:v(p) directional second derivative}
(x_0, x_1, \dots) \cdot \left( \begin{array}{ccc} \frac{\alpha_0(\alpha_0-1)}{p_0^2} & \frac{\alpha_0\alpha_1}{p_0p_1} & \dots \\ \frac{\alpha_0\alpha_1}{p_0p_1} & \frac{\alpha_1(\alpha_1-1)}{p_1^2} & \dots \\
\dots & \dots & \dots 
\end{array} \right) \cdot \left(\begin{array}{c} x_0 \\ x_1 \\ \dots \end{array}\right) ~~~=~~~ \left(\sum_{i = 0}^{k} \frac{\alpha_i x_i}{p_i}\right)^2 - \sum_{i = 0}^{k} \frac{\alpha_i x_i^2}{p_i^2},
\end{equation}
where for simplicity we have ignored the positive factor $v(p)$ as it does not affect the sign.

We first use this to show that if $\alpha_1 > 0$ (or any $\alpha_j > 0$), then the function $v(p)$ is \emph{not} convex for $p \in int(\Delta^k)$. Indeed, consider the direction $(1, -1, 0, 0, \dots, 0)$, which maintains $p \in int(\Delta^k)$. The directional second derivative can be computed as 
\[
\frac{\alpha_0(\alpha_0-1)}{p_0^2} - \frac{2\alpha_0\alpha_1}{p_0p_1} + \frac{\alpha_1(\alpha_1-1)}{p_1^2}. 
\]
Suppose $p_0 = \alpha_0 x$, $p_1 = \alpha_1 x$ for some small positive number $x$, and $p_2, p_3, \dots$ are arbitrary. Then the above second derivative simplifies to $-\frac{(\alpha_0+\alpha_1)}{\alpha_0\alpha_1x^2} < 0$. 
Thus $v(p)$ is not convex along this direction. 

Suppose instead $\alpha_1, \dots, \alpha_k \leq 0$, we will show $v(p)$ is convex. For this it suffices to show the RHS of \eqref{eq:v(p) directional second derivative} is non-negative. Indeed, by the Cauchy-Schwartz inequality, 
\begin{align*}
&\left(\left(\sum_{i = 0}^{k} \frac{\alpha_i x_i}{p_i}\right)^2 + \frac{-\alpha_1x_1^2}{p_1^2} + \cdots + \frac{-\alpha_kx_k^2}{p_k^2}\right) \cdot (1 + (-\alpha_1) + \dots + (-\alpha_k)) \\
\geq &\left(\sum_{i = 0}^{k} \frac{\alpha_i x_i}{p_i} + \frac{-\alpha_1x_1}{p_1} + \cdots + \frac{-\alpha_kx_k}{p_k}\right)^2 = \left(\frac{\alpha_0x_0}{p_0}\right)^2.
\end{align*}
Using $\alpha_0 + \alpha_1 + \dots + \alpha_k = 1$ to simplify, this exactly implies $\left(\sum_{i = 0}^{k} \frac{\alpha_i x_i}{p_i}\right)^2 \geq \sum_{i = 0}^{k} \frac{\alpha_i x_i^2}{p_i^2}$ as desired. In fact, $v(p)$ is convex for all $p \gg 0$, including $p \in int(\Delta^k)$.

Moreover, if $\alpha_1, \dots, \alpha_k$ are \emph{strictly} negative, then the equality condition of the Cauchy-Schwartz inequality above requires $\sum_{i = 0}^{k} \frac{\alpha_i x_i}{p_i} = \frac{x_1}{p_1} = \dots = \frac{x_k}{p_k}$, which in turn implies that $x_0, x_1, \dots, x_k$ have the same sign (under the assumption $\alpha_0 > 0 > \alpha_1, \dots, \alpha_k$). Thus, for any direction $(x_0, x_1, \dots, x_k)$ with $x_0 + x_1 + \dots + x_k = 0$, the directional second derivative of $v$ is strictly positive. So $v$ is strictly convex for $p \in int(\Delta^k)$. 

\bigskip

Next, we will show that if $\alpha_1 < 0$ (or any $\alpha_j < 0$), then the function $v(p)$ is \emph{not} concave for $p \in int(\Delta^k)$. For this we again consider the second derivative along the direction $(1, -1, 0, 0, \dots, 0)$, which is $\frac{\alpha_0(\alpha_0-1)}{p_0^2} - \frac{2\alpha_0\alpha_1}{p_0p_1} + \frac{\alpha_1(\alpha_1-1)}{p_1^2}$. As $\alpha_1 < 0$, we have $\alpha_1(\alpha_1-1) > 0$. Thus for $p_0$ close to $1$ and $p_1$ close to $0$, the above second derivative is positive and $v(p)$ is not concave along this direction.  

Finally, we show that if $\alpha_1, \dots, \alpha_k \geq 0$, then the function $v(p)$ is concave. By the Cauchy-Schwartz inequality, 
\begin{align*}
\left(\sum_{i = 0}^{k} \frac{\alpha_i x_i^2}{p_i^2}\right) \cdot \left(\sum_{i = 0}^{k} \alpha_i\right) \geq \left(\sum_{i = 0}^{k} \frac{\alpha_i x_i}{p_i}\right)^2.
\end{align*}
Since $\sum_{i = 0}^{k} \alpha_i = 1$, this implies the RHS of \eqref{eq:v(p) directional second derivative} is non-positive. Hence $v$ has non-positive directional second derivatives and must be globally concave. 

Moreover, if $\alpha_1, \dots, \alpha_k$ are strictly positive, then the equality condition of the Cauchy-Schwartz inequality requires $\frac{x_0}{p_0} = \frac{x_1}{p_1} = \dots = \frac{x_k}{p_k}$, which in turn requires $x_0, x_1, \dots, x_k$ to have the same sign. By the same argument as above, we conclude that in this case $v$ is strictly concave for $p \in int(\Delta^k)$. 
\end{proof}

\section{Proof of a Conjecture Regarding Majorization}
\label{sec:jensen}

\cite{jensen2019asymptotic} studies the majorization order on finitely supported distributions. Given two such distributions $\mu$
and $\nu$, $\mu$ is said to {\em majorize} $\nu$ if for every $n \geq 1$ it holds that the sum of the largest $n$ probabilities in $\mu$ is greater than or equal to the sum of the $n$ largest probabilities in $\nu$. The R\'enyi entropy of a distribution $\mu$ defined on a finite set $S$ is given by
$$
H_\mu(\alpha) = \frac{1}{1-\alpha}\log\left(\sum_{s \in S} \mu(s)^\alpha \right),
$$
for $\alpha \in [0,\infty) \setminus \{1\}$.
As with our definition of R\'enyi divergences, this definition is extended to $\alpha = 1$ by continuity to equal the Shannon entropy, and extended to $\alpha = \infty$ to equal $-\log\max_s \mu(s)$. Hence $H_\mu$ is defined on $[0,\infty]$. 

Note that $H_\mu(0)$ is the size of the support of $\mu$. In his Proposition 3.7, Jensen shows that if $H_\mu(\alpha) < H_\nu(\alpha)$ for all $\alpha \in [0,\infty]$ then the $n$-fold product $\mu^{\times n}$ majorizes $\nu^{\times n}$.

Commenting on his Proposition 3.7, Jensen writes ``The author cautiously conjectures that \dots the requirement of a sharp inequality at $0$ could be replaced by a similar condition regarding the
$\alpha$-R\'enyi entropies for negative $\alpha$.''

To understand this statement in terms of the nomenclature and notation of our paper, we identify each distribution $\mu$ whose support is a finite set $S$ with the experiment $P^\mu = (S,P_1,P_0)$, where $P_1 = \mu$ and $P_0$ is the uniform distribution  on $S$. There is a simple connection between the R\'enyi entropy of $\mu$ and the R\'enyi divergence of $P^\mu$. For $\alpha  \geq 0$,
\begin{equation}\label{eq:entropy and divergence positive alpha}
   H_\mu(\alpha) =  \log |S| - R_P^{1}(\alpha).
\end{equation}
As Jensen suggests, $H_\mu(\alpha)$ for negative $\alpha$ is also important, as it relates to $R_P^0$. For $\alpha \leq 0$,
\begin{equation}\label{eq:entropy and divergence negative alpha}
  H_\mu(\alpha) = \log |S| - \frac{\alpha}{1-\alpha}R_P^0(1-\alpha),
\end{equation}
which extends to $\alpha = -\infty$ to equal $-\log\min_{s} \mu(s)$. Moreover, note that 
\begin{equation}\label{eq:entropy and divergence zero alpha}
H_\mu'(0) = - R_P^0(1) = \log |S| + \frac{1}{|S|} \sum_{s \in S} \log \mu(s). 
\end{equation}

As shown by \citet[p.\ 264]{torgersen1985majorization}, when $\mu$ and $\nu$ have the same support size, then majorization of $\nu$ by $\mu$ is equivalent to Blackwell dominance of $P^\mu$ over $P^\nu$. Thus Jensen's Proposition 3.7, which assumes that the support sizes are different, has no implications for Blackwell dominance. However, our result on Blackwell dominance does have implications for majorization. In particular, the following proposition follows immediately from the application of Theorem~\ref{thm:eventual} to experiments of the form $P^\mu$.
\begin{proposition}
Let $\mu,\nu$ be finitely supported distributions with the same support size (i.e., $H_\mu(0) = H_\nu(0)$), and such that $H_\mu(\infty) \neq H_\nu(\infty)$ and $H_\mu(-\infty) \neq H_\nu(-\infty)$. Then the following are equivalent:
\begin{enumerate}
    \item $H_\mu(\alpha) < H_\nu(\alpha)$ for all $\alpha \in (0, \infty]$, $H_\mu(\alpha) > H_\nu(\alpha)$ for all $\alpha \in [-\infty, 0)$ and $H_\mu'(0) < H_\nu'(0)$.\footnote{This last condition is necessary for majorization, but it was not recognized in the original conjecture of \cite{jensen2019asymptotic}.}
    \item There exists an $n_0$ such that $\mu^{\times n}$ majorizes $\nu^{\times n}$ for every $n \geq n_0$.
\end{enumerate}
\end{proposition}

\begin{proof}
For notational ease, let $P$ denote $P^\mu$ and $Q$ denote $P^\nu$. The assumption $H_\mu(\alpha) < H_\nu(\alpha)$ for all $\alpha > 0$ is equivalent, via \eqref{eq:entropy and divergence positive alpha}, to $R_P^1(t) > R_Q^1(t)$ for all $t > 0$, and to $R_P^0(t) > R_Q^0(t)$ for all $t \in (0,1)$, using $R_P^0(t) = \frac{t}{1-t} R_P^1(1-t)$ for $0 < t < 1$.

On the other hand, $H_\mu(\alpha) > H_\nu(\alpha)$ for all $\alpha < 0$ and $H_\mu'(0) < H_\nu'(0)$ is equivalent, via \eqref{eq:entropy and divergence negative alpha} and \eqref{eq:entropy and divergence zero alpha}, to $R_P^0(t) > R_Q^0(t)$ for all $t \geq 1$. So (i) is equivalent to $P$ dominating $Q$ in the R\'enyi order. 

Finally, the assumptions that $H_\mu(\infty) \neq H_\nu(\infty)$ and $H_\mu(-\infty) \neq H_\nu(-\infty)$ translate into $\max_{s} \mu(s) \neq \max_{s} \nu(s)$ and $\min_{s} \mu(s) \neq \min_{s} \nu(s)$, which are in turn equivalent to requiring that $P$ and $Q$ be a generic pair. Therefore, by Theorem~\ref{thm:eventual}, (i) is equivalent to $P^{\otimes n}$ Blackwell dominates $Q^{\otimes n}$ for every large $n$. It follows from \cite{torgersen1985majorization} that (i) is equivalent to (ii).
\end{proof}

\end{document}